\documentclass[11pt]{amsart}

\usepackage[margin=1.23in]{geometry}
\usepackage{amsmath,amssymb,amsfonts,amsthm,mathtools,mathrsfs}
\usepackage{bm}
\usepackage{enumitem}
\usepackage{booktabs}
\usepackage{multirow}
\usepackage{graphicx}
\usepackage{placeins}
\usepackage{microtype}
\usepackage{xcolor}
\usepackage[normalem]{ulem}
\usepackage[colorlinks=true,linkcolor=black,citecolor=black,urlcolor=black]{hyperref}

\allowdisplaybreaks
\numberwithin{equation}{section}

\newtheorem{theorem}{Theorem}[section]
\newtheorem{lemma}[theorem]{Lemma}
\newtheorem{proposition}[theorem]{Proposition}

\theoremstyle{definition}

\newtheorem{definition}[theorem]{Definition}
\newtheorem{numexample}{Example}[section]
\theoremstyle{remark}
\newtheorem{remark}[theorem]{Remark}

\newcommand{\R}{\mathbb{R}}
\newcommand{\G}{\mathcal{G}}

\newcommand{\Kc}{\mathscr{K}}
\newcommand{\Th}{\mathcal{T}_h}
\newcommand{\Hrel}{\mathcal{H}}
\newcommand{\Qrel}{\mathcal{Q}}

\newcommand{\dd}{\mathrm{d}}

\newcommand{\ext}{\mathrm{ext}}
\newcommand{\norm}[1]{\lVert#1\rVert}

\newif\ifcmmShowDel
\cmmShowDeltrue

\newcommand{\cmmAdd}[1]{{\color{blue}#1}}

\newcommand{\cmmNote}[1]{}

\title[Fully Discrete Multi-Entropy Stability for MHD]
{Fully Discrete Multi-Entropy Stability of High-Order Schemes
	for Compressible MHD: A Weak-to-Strong Framework}

\author[H. Cao]{Huihui Cao}
\address[H. Cao]{Shenzhen International Center for Mathematics, Southern University of Science and Technology, 1088 Xueyuan Avenue, Shenzhen 518055, Guangdong, China}
\email{caohh@sustech.edu.cn}

\author[K. Wu]{Kailiang Wu}
\address[K. Wu]{Department of Mathematics and Shenzhen International Center for Mathematics, Southern University of Science and Technology, 1088 Xueyuan Avenue, Shenzhen 518055, Guangdong, China}
\email{wukl@sustech.edu.cn}
\thanks{Corresponding author: Kailiang Wu.}

\author[C. Yuan]{Caiyou Yuan}
\address[C. Yuan]{Department of Mathematics and Shenzhen International Center for Mathematics, Southern University of Science and Technology, 1088 Xueyuan Avenue, Shenzhen 518055, Guangdong, China}
\email{yuancy@sustech.edu.cn}

\date{}

\begin{document}
\begin{abstract}
	We establish a fully discrete weak-to-strong (W2S) multi-entropy stability
	theory for arbitrarily high-order finite-volume and discontinuous Galerkin (DG)
	approximations of the ideal compressible magnetohydrodynamics (MHD) equations on general polytopal meshes,
	whereby a single numerical update simultaneously satisfies discrete entropy inequalities
	for any prescribed finite family of convex Harten entropy pairs. 
	Because physical entropies are defined only for positive density and pressure,
	the central analytical difficulty lies in reconciling discrete entropy stability with
	positivity preservation, the magnetic divergence constraint, and the nonconservative Godunov--Powell coupling.

	Building upon the provably positivity-preserving MHD framework of [K.~Wu and C.-W.~Shu, \emph{Numer.~Math.}~\textbf{142} (2019), 995--1047],
	we develop a weak multi-entropy analysis of the common cell-average evolution of finite-volume and DG methods. 
	Through convex decomposition and relative entropy, we establish two sufficient criteria for this weak stability in one dimension and on general polytopal meshes. 
	In multiple dimensions, compatible magnetic corrections and locally divergence-free approximations cooperatively offset the nonconservative Godunov--Powell coupling and provide the discrete cancellations essential to both positivity-preservation and entropy stability. 
	A W2S lifting then unifies positivity and entropy limiting into a single cellwise scaling limiter, achieving strong multi-entropy stability while strictly preserving cell averages and the locally divergence-free magnetic field. 
	High-order temporal accuracy follows from variable-step strong-stability-preserving multistep methods. 
	Numerical experiments confirm the theoretical findings and demonstrate computational robustness in near-vacuum, high-Mach, and strongly magnetized regimes. 
	This framework provides a rigorous, fully discrete foundation for unifying physical admissibility, magnetic divergence control, and multi-entropy stability in high-order MHD approximations.
\end{abstract}   

\maketitle

\section{Introduction}\label{sec:intro}

The ideal compressible magnetohydrodynamic (MHD) equations describe the
macroscopic dynamics of electrically conducting fluids interacting with
magnetic fields, with wide applications in astrophysics, space physics,
geophysics, and magnetic confinement fusion \cite{GoedbloedPoedts2004,Toth2000}.
In $d$ spatial dimensions ($d\in\{1,2,3\}$), the system can be written in conservative form as
\begin{equation}\label{eq:intro-ideal-MHD}
	\frac{\partial\bm U}{\partial t}
	+\sum_{i=1}^d\frac{\partial\bm F_i(\bm U)}{\partial x_i}=\bm 0.
\end{equation}
The vector of conserved variables $\bm U$ and the flux
$\bm F=(\bm F_1,\ldots,\bm F_d)$ are given by
\begin{equation}\label{eq:intro-state}
	\bm U=(\rho,\bm m^\top,\bm B^\top,E)^\top,
\end{equation}
\begin{equation}\label{eq:intro-MHD-flux}
	\bm F_i(\bm U)=
	\begin{pmatrix}
		\rho v_i\\
		v_i\bm m-B_i\bm B+p_{\rm tot}\bm e_i\\
		v_i\bm B-B_i\bm v\\
		v_i(E+p_{\rm tot})-B_i(\bm v\cdot\bm B)
	\end{pmatrix},
	\qquad i=1,\ldots,d.
\end{equation}
Here $\rho$ denotes the fluid density, $\bm m = \rho\bm v$ the momentum density, $\bm v=(v_1, v_2, v_3)^\top$
the fluid velocity, $\bm B=(B_1, B_2, B_3)^\top$ the magnetic field, and $E$ the total energy
density. Let $\bm x=(x_1,\ldots,x_d)^\top\in\R^d$, and let $\bm e_i$
denote the $i$th column of the $3\times3$ identity matrix. With the
magnetic permeability normalized to unity, the total pressure and total
energy density are defined by
\begin{equation}\label{eq:intro-pressure-energy}
	p_{\rm tot}=p+\frac{|\bm B|^2}{2},
	\qquad
	E=\rho e+\frac{\rho|\bm v|^2}{2}+\frac{|\bm B|^2}{2},
\end{equation}
where $p$ is the thermal pressure and $e$ the specific internal energy.
Throughout this paper, the system is closed by the ideal gas equation
of state (EOS),
\begin{equation}\label{eq:intro-EOS}
	p=(\gamma-1)\rho e
	=(\gamma-1)\left(E-\frac{|\bm m|^2}{2\rho}
	-\frac{|\bm B|^2}{2}\right),
\end{equation}
where $\gamma>1$ is the adiabatic index.

The mathematical structure of \eqref{eq:intro-ideal-MHD} is governed
by three interdependent properties: physical admissibility, the
divergence-free constraint, and entropy inequality. 
First, the physically admissible states form an open convex set \cite{Wu2018}
\begin{equation}\label{eq:intro-admissible-set}
	\G=\left\{\bm U = (\rho,\bm m^\top,\bm B^\top,E)^\top:\rho>0,\quad
	E-\frac{|\bm m|^2}{2\rho}-\frac{|\bm B|^2}{2}>0\right\},
\end{equation}
which is equivalent to positivity of the density and the thermal pressure under \eqref{eq:intro-EOS}.
Outside $\G$, thermodynamic quantities lose their physical meaning and the equations \eqref{eq:intro-ideal-MHD} can lose their hyperbolic character.
Second, the divergence-free constraint on the magnetic field,
\begin{equation}\label{eq:intro-divergence-free}
	\nabla\cdot\bm B:=\sum_{i=1}^d\frac{\partial B_i}{\partial x_i}=0,
\end{equation}
is an involution of \eqref{eq:intro-ideal-MHD}. Indeed, taking the
divergence of the induction equation yields $\partial_t(\nabla\cdot\bm B)=0$,
which implies that an initially divergence-free magnetic field remains solenoidal for smooth solutions.
This constraint is also closely linked to positivity:
once violated, even exact smooth solutions of the conservative
MHD system \eqref{eq:intro-ideal-MHD} may develop negative pressure \cite{WuShu2018}.
Thirdly, because nonlinear hyperbolic systems can develop discontinuities
in finite time, entropy conditions are required to select the physically relevant weak solution.
Under \eqref{eq:intro-EOS}, with appropriate nondimensionalization,
the physical entropy is defined as
\[
s=\log p-\gamma\log\rho.
\]
We consider the generalized Harten entropy pairs
\begin{equation}\label{eq:intro-harten-entropy}
	\eta_f(\bm U)=-\rho f(s),
	\qquad
	\bm q_f(\bm U)=\eta_f(\bm U)\bm v,
\end{equation}
where $f\in C^2(\R)$ satisfies
\begin{equation}\label{eq:intro-harten-conditions}
	f'(s)>0,
	\qquad
	f'(s)-\gamma f''(s)>0,
	\qquad s\in\R.
\end{equation}
These conditions ensure the strict convexity of $\eta_f$ with respect to $\bm U$ on $\G$.
The corresponding entropy family for the Euler equations was characterized in
\cite{HartenLaxLevermoreMorokoff1998}, and generalized Harten
entropies have been employed in entropy-split formulations of
ideal MHD \cite{SjogreenYee2024}.
Under divergence-free condition \eqref{eq:intro-divergence-free}, 
entropy-admissible weak solutions of the conservative MHD system obey the entropy inequality
\begin{equation}\label{eq:intro-entropy-inequality}
	\partial_t\eta_f+\nabla\cdot\bm q_f\le0,
\end{equation}
for any $f \in C^2(\R)$ satisfying \eqref{eq:intro-harten-conditions}, in the sense of distributions, with equality holding for smooth solutions.
The solenoidal constraint is essential to this balance.
If relaxed \cite{LiuShuZhang2018}, smooth admissible solutions of the conservative system instead satisfy
\begin{equation}\label{eq:intro-entropy-divergence}
	\partial_t\eta_f+\nabla\cdot\bm q_f
	=
	\frac{(\gamma-1)\rho f'(s)}{p}
	(\bm v\cdot\bm B)\nabla\cdot\bm B.
\end{equation}
The right-hand side is sign-indefinite. To restore the entropy equality
for smooth solutions and obtain a symmetric hyperbolic formulation in entropy
variables, one considers the Godunov--Powell form \cite{Godunov1972,PowellEtAl1999}
\begin{equation}\label{eq:intro-GP}
	\begin{gathered}
		\bm U_t+\nabla\cdot\bm F(\bm U)
		=-(\nabla\cdot\bm B)\bm S(\bm U),\\
		\bm S(\bm U)
		=\bigl(0,\bm B^\top,\bm v^\top,\bm v\cdot\bm B\bigr)^\top.
	\end{gathered}
\end{equation}
The nonconservative Godunov--Powell source term cancels the sign-indefinite contribution in \eqref{eq:intro-entropy-divergence}.
When $\nabla\cdot\bm B=0$, this source vanishes and \eqref{eq:intro-GP}
reduces to the conservative system \eqref{eq:intro-ideal-MHD}.
For entropy-stable formulations based on the Godunov--Powell source, the nonconservative discretization raises a tension between entropy stability and conservation.
As emphasized in \cite{ChandrashekarKlingenberg2016, LiuShuZhang2018}, ``There is still a conflict between the entropy stability which requires the non-conservative source terms, and the conservation property which is lost due to these source terms.'' 

As evidenced by \eqref{eq:intro-entropy-divergence}, discrete entropy stability in MHD requires a compatible treatment of the solenoidal constraint \eqref{eq:intro-divergence-free}. Established divergence-control strategies include projection methods \cite{BrackbillBarnes1980}, constrained transport \cite{EvansHawley1988,BalsaraSpicer1999}, and hyperbolic divergence cleaning \cite{DednerEtAl2002}. Alternatively, the constraint can be embedded directly into the discrete space: locally divergence-free (LDF) methods enforce vanishing divergence cellwise \cite{LiShu2005}, while globally divergence-free formulations additionally maintain normal trace continuity across element interfaces \cite{LiXuYakovlev2011,FuLiXu2018,BalsaraKumarChandrashekar2021}.

Because physical entropies and their dual entropy variables are defined only on $\G$, any rigorous entropy-stability theory for MHD must first guarantee positivity preservation.
Representative positivity-preserving (PP) frameworks for hyperbolic conservation laws include the Zhang--Shu approach \cite{ZhangShu2010,ZhangShu2011}, parametrized flux limiters \cite{Xu2014,XiongQiuXu2016}, and invariant-domain-preserving methods with convex limiting \cite{GuermondPopov2016,GuermondPopovTomas2019,Kuzmin2021}; see \cite{WuZhangShu2026} for a comprehensive review.
Early PP schemes for MHD include finite-volume, discontinuous Galerkin (DG), and weighted essentially nonoscillatory (WENO) methods \cite{Waagan2009,Balsara2012,ChengEtAl2013,ChristliebEtAl2015}.
Wu \cite{Wu2018} uncovered the intrinsic differential--algebraic link between the positivity of conservative MHD schemes and a discrete divergence-free condition.
Building on this discovery, Wu and Shu developed provably PP high-order schemes combining LDF spaces with a compatible discretization of the Godunov--Powell source on Cartesian meshes with Lax--Friedrichs fluxes \cite{WuShu2018} and subsequently on general polytopal meshes with Harten--Lax--van Leer (HLL)-type fluxes \cite{WuShu2019}.
The geometric quasilinearization (GQL) framework \cite{WuShuGQL2023} characterizes nonlinear admissibility sets via equivalent families of supporting hyperplanes, elucidating the differential--algebraic coupling \cite{Wu2018} and guiding the design of PP MHD schemes \cite{WuJiangShu2023,DingWu2024}.
Further developments include globally divergence-free PP DG methods \cite{YanCaoWu2025,LiuWuYuan2026}, constrained-transport and active-flux solvers \cite{PangWu2025,LiuPangAbgrallWu2025}, and structure-preserving or oscillation-eliminating formulations \cite{DaoNazarovTomas2024,LiuWu2025}.

Discrete entropy inequalities provide a rigorous foundation for nonlinear stability, motivating extensive research on entropy-stable (ES) schemes.
Building on Tadmor's foundational entropy-conservative and ES framework \cite{Tadmor1987,Tadmor2003}, high-order extensions include TeCNO schemes \cite{FjordholmMishraTadmor2012},
summation-by-parts (SBP) finite-difference schemes \cite{FisherCarpenter2013}, split-form DG spectral element methods \cite{Gassner2013}, quadrature-based DG methods with SBP operators \cite{ChenShu2017}, and DG flux-differencing formulations \cite{Chan2018}; see \cite{ChenShu2020} for a review.
Complementary developments include Abgrall's residual-distribution framework \cite{Abgrall2018} alongside fully discrete approaches based on
entropy-conservative space--time discretizations \cite{LeFlochMercierRohde2002}, relaxation Runge--Kutta methods \cite{RanochaSayyariDalcinParsaniKetcheson2020} and their locally entropy-stable extensions \cite{RanochaDalcinParsani2020}, and algebraic flux limiting \cite{KuzminHajdukRupp2022}.
For MHD, entropy stability also requires a compatible treatment of the magnetic divergence constraint.
Existing MHD constructions include entropy-conservative fluxes \cite{ChandrashekarKlingenberg2016,WintersGassner2016}, dissipation operators acting on jumps in entropy variables \cite{DerigsEtAl2017}, and entropy-consistent divergence cleaning \cite{DerigsEtAl2018}.
High-order developments include DG methods with suitable quadrature \cite{LiuShuZhang2018}, SBP-based DG spectral element methods for resistive MHD \cite{BohmEtAl2020}, Gauss collocation methods \cite{RuedaRamirezHindenlangChanGassner2023}, and entropy split formulations \cite{SjogreenYee2024}.
Further extensions comprise ES finite-volume schemes on adaptive unstructured meshes \cite{ZhangZhengFengLiu2024}, essentially oscillation-free DG methods \cite{LiuLuShu2025}, and globally divergence-free nodal DG methods for conservative ideal MHD \cite{LiuGuoJiangZhang2025}.
Efforts bridging entropy stability with positivity preservation include low-order Godunov-type and relaxation solvers \cite{Gallice2003,BouchutKlingenbergWaagan2007,BouchutKlingenbergWaagan2010,WaaganFederrathKlingenberg2011}, high-order schemes enforcing pressure positivity alongside entropy dissipation \cite{DerigsWintersGassnerWalch2016,DzanicWitherden2023}, and a nodal DG scheme combining an LDF projection, PP limiting, and an ES HLL flux for semidiscrete stability of a single entropy pair \cite{YueWuShu2026}.

Despite these advances, three fundamental challenges persist in the design of high-order ES schemes for MHD.
First, most available entropy-stability results are established only at the semidiscrete level. Standard explicit time integrators generally fail to inherit the continuous entropy decay, whereas implicit ES integrators require solving costly nonlinear algebraic systems at each time step.
Second, almost all existing ES schemes are tailored to a single entropy pair. Because the dual entropy variables and the associated numerical dissipation depend explicitly on this chosen pair, stability for one specific entropy does not automatically imply stability for any other pair under the same numerical update. For instance, the schemes in \cite{ChandrashekarKlingenberg2016,WintersGassner2016,LiuShuZhang2018} use, up to a positive constant factor, the standard physical entropy pair corresponding to $f(s)=s$ in \eqref{eq:intro-harten-entropy}:
\[
\eta_{*}(\bm U)=-\rho s,
\qquad
\bm q_{*}(\bm U)=\eta_{*}(\bm U)\bm v.
\]
Their stability proofs do not extend to $\eta_f(\bm U)=-\rho f(s)$ for general functions $f$ satisfying \eqref{eq:intro-harten-conditions}.
Third, entropy stability does not ensure physical admissibility. Many ES constructions assume $\bm U\in\G$ a priori without proving its preservation. If positivity is violated, the physical entropy and dual entropy variables are no longer defined, and the stability analysis loses its meaning. A provable PP mechanism is thus an indispensable prerequisite for any rigorous entropy-stability theory.

\begin{samepage}
These challenges motivate the central question addressed in this paper:
\begin{center}
	\emph{How can one construct uniformly high-order, provably PP schemes for compressible MHD\\*
		possessing fully discrete multi-entropy stability?}
\end{center}
\end{samepage}
Here, \emph{multi-entropy stability} refers to the simultaneous satisfaction of fully discrete entropy inequalities under a single numerical update for an arbitrary prescribed finite family $\{(\eta_{f_r},\bm q_{f_r})\}_{r=1}^{M}$ of generalized Harten entropy pairs \eqref{eq:intro-harten-entropy} satisfying \eqref{eq:intro-harten-conditions}.

To address this question, we build upon the provably PP LDF framework of Wu and Shu \cite{WuShu2019} and extend the weak-to-strong (W2S) strategy of the entropy-positivity-oscillation (EPO) framework \cite{WuEPO2026} to the compressible MHD system with nonconservative Godunov--Powell symmetrization.
In the EPO framework, a weak entropy estimate for the updated cell average governs a high-order polynomial correction that enforces fully discrete quadrature-level entropy inequalities. 
An entropy limiter for DG, similar to that in EPO \cite{WuEPO2026}, was independently developed in \cite{LiuGuoJiangSun2026}. 
Applying the EPO methodology to MHD, however, raises a major analytical challenge: the original EPO theory was developed for strictly conservative systems, whereas the provably PP, LDF discretization of MHD relies essentially on the nonconservative Godunov--Powell source in \eqref{eq:intro-GP}.
Through relative entropy estimates and discrete balance identities, we identify the precise mechanism by which discrete magnetic divergence errors enter the entropy balance: non-solenoidal fields produce sign-indefinite residual terms that undermine weak entropy stability.
We then prove that cellwise LDF approximations combined with a compatible Godunov--Powell interface discretization exactly eliminate these residuals, offsetting nonconservative errors both within element interiors and across interfaces. 
This exact cancellation enables the simultaneous attainment of magnetic divergence control, strict positivity preservation, and fully discrete multi-entropy stability within an arbitrarily high-order computational framework.

The main contributions of this paper are summarized as follows:
\begin{itemize}[leftmargin=2em,itemsep=0.5em]
	
	\item We establish two sufficient criteria, based on convex decomposition and relative entropy analysis, for the fully discrete weak multi-entropy stability of the arbitrarily high-order PP HLL schemes \cite{WuShu2019} in one dimension and on general multidimensional polytopal meshes.
	The analysis rigorously identifies the sign-indefinite residuals caused by magnetic divergence; by introducing source-corrected HLL intermediate states and compatible magnetic flux corrections, we account for the Godunov--Powell coupling, bound the corrected relative entropy flux by the relative entropy, and prove exact discrete divergence cancellation on LDF spaces under explicit CFL conditions.
	
	\item We construct an arbitrarily high-order weak-to-strong (W2S) lifting operator that converts the weak cell-average entropy estimates into strong, simultaneous entropy inequalities while guaranteeing physical admissibility at all specified points.
	Formulated as a cellwise convex scaling limiter, this operator strictly preserves cell averages, the invariant normal magnetic component in one dimension, and the cellwise LDF property in multiple dimensions.
	
	\item We extend the fully discrete ES theory to high-order temporal accuracy via variable-step strong-stability-preserving (SSP) multistep methods under compatible step-size restrictions.
	Because SSP multistep updates are convex combinations of forward Euler-type steps, performing a single W2S lifting per time step suffices for fully discrete multi-entropy stability, without requiring W2S limiting for intermediate stages.
	
	\item Numerical tests on one- and two-dimensional benchmarks confirm the theoretical properties of the schemes: the designed high-order accuracy, multi-entropy decay, and robust performance in extreme regimes involving strong shocks, low plasma beta, and near-vacuum states.
\end{itemize}

The paper is organized as follows. 
Section~\ref{sec:1d} establishes the two sufficient criteria for weak multi-entropy stability in one dimension.
Section~\ref{sec:md} extends the theory to multidimensional general meshes and establishes the exact discrete divergence cancellation.
Section~\ref{sec:lifting} constructs the high-order weak-to-strong lifting operator.
Section~\ref{sec:high-order-time} formulates the fully discrete variable-step SSP multistep time discretization.
Section~\ref{sec:numer} presents numerical experiments validating accuracy, multi-entropy stability, and robustness.
Section~\ref{sec:con} offers concluding remarks.
Appendix~\ref{app:harten} details the convexity and compatibility proofs for the generalized Harten entropy family.

Throughout, $r$ indexes the prescribed entropy family, $j$ a one-dimensional cell, and $K$ a general mesh cell; spatial and interface indices precede the entropy index.
Superscripts $-$ and $+$ denote interface traces, $0$ initial speeds or reference time steps, and $\star$ polynomials before limiting.
Gradients and Hessians are taken with respect to the conservative variables, and $Dg(\bm U)[\bm Z]$ denotes the directional derivative of $g$ along $\bm Z$.
Single bars $|\cdot|$ denote the scalar absolute value, Euclidean length of spatial vectors, or cell/face measure, while $\norm{\cdot}$ denotes the Euclidean norm on state space.

\section{Weak multi-entropy stability in one dimension}
\label{sec:1d}

In one dimension, the divergence constraint $\partial_x B_1=0$
and the induction equation $\partial_t B_1=0$ imply that the normal magnetic
field is invariant in space and time. Denoting its constant value by $B_{\rm const}$,
we restrict our analysis to the convex slice
$\G_{B_{\rm const}}:=\{\bm U\in\G:B_1=B_{\rm const}\}$.
We first introduce the notions of weak and strong entropy stability, and then
establish two sufficient criteria for the weak multi-entropy stability of forward Euler cell-average updates.

\subsection{Weak and strong entropy stability in one dimension}
\label{subsec:entropy-notions}

Following \cite{WuEPO2026}, we distinguish
an entropy inequality satisfied by the updated cell average from one satisfied by the
updated polynomial approximation. The qualifiers \emph{weak} and \emph{strong} refer
to these two discrete levels of stability, rather than the regularity
of the underlying solution.

With the entropy family fixed in Section~\ref{sec:intro}, let
$q_r(\bm U):=\bm q_r(\bm U)\cdot\bm e_1=v_1\eta_r(\bm U)$
denote the scalar entropy flux in one dimension.
By Lemma~\ref{lem:harten-convexity}, the Hessian of each $\eta_r$ is
positive definite on $\G$, so $\eta_r$ is strictly convex on
the admissible slice $\G_{B_{\rm const}}$.

On a finite uniform mesh
$\{I_j=[x_{j-1/2},x_{j+1/2}]\}_{j=1}^{N_x}$ of width
$\Delta x$, let $\bm U_j^n\in[\mathbb P^k(I_j)]^8$ denote
a conservative finite-volume reconstruction or a DG approximation,
where $k\ge0$ is the polynomial degree and $\mathbb P^k$ denotes polynomials
of degree at most $k$. We require
$B_{1,j}^n(x)\equiv B_{\rm const}$ throughout each cell.
Choose an $L$-point quadrature rule exact on $\mathbb P^k(I_j)$,
with nodes $\widehat x_j^{[\alpha]}\in I_j$ and positive weights
$\widehat\omega_\alpha>0$ normalized such that
$\sum_{\alpha=1}^L\widehat\omega_\alpha=1$.
The cell average is
\begin{equation}\label{eq:1d-cell-average}
	\overline{\bm U}_j^n
	:=\frac1{\Delta x}\int_{I_j}\bm U_j^n(x)\,\dd x
	=\sum_{\alpha=1}^L\widehat\omega_\alpha
	\bm U_j^n(\widehat x_j^{[\alpha]}).
\end{equation}
Assuming all quadrature states belong to $\G_{B_{\rm const}}$,
the discrete cell quadrature entropy is
\begin{equation}\label{eq:1d-quadrature-entropy}
	\mathcal S_{j,r}(\bm U_j^n)
	:=\sum_{\alpha=1}^L\widehat\omega_\alpha
	\eta_r(\bm U_j^n(\widehat x_j^{[\alpha]})).
\end{equation}
For a forward Euler step $\tau=t^{n+1}-t^n>0$, set
$\lambda=\tau/\Delta x$. For each entropy, we prescribe a consistent
numerical entropy flux $\widehat q_r$ satisfying
$\widehat q_r(\bm W,\bm W)=q_r(\bm W)$ for
$\bm W\in\G_{B_{\rm const}}$, with a single interface flux
$\widehat q_{j+1/2,r}$ shared by adjacent cells. These fluxes are evaluated
from the solution at time level $t^n$ and remain frozen during polynomial lifting.
Define the local entropy budget by
\begin{equation}\label{eq:1d-weak-bound}
	\mathcal B_{j,r}^n
	:=\mathcal S_{j,r}(\bm U_j^n)
	-\lambda(\widehat q_{j+1/2,r}-\widehat q_{j-1/2,r}).
\end{equation}
Subsection~\ref{subsubsec:1d-entropy-bound} specifies the HLL
entropy fluxes employed throughout.

\begin{definition}[Local weak and strong multi-entropy stability in one dimension]
	\label{def:weak-strong-entropy}
	A cell-average update is \emph{locally weakly entropy stable}
	for the prescribed entropy pair $(\eta_r,q_r)$ if, in every cell,
	\begin{equation}\label{eq:1d-weak-entropy}
		\overline{\bm U}_j^{n+1}\in\G_{B_{\rm const}},
		\qquad
		\eta_r(\overline{\bm U}_j^{n+1})\le\mathcal B_{j,r}^n.
	\end{equation}
	A polynomial update $\bm U_j^{n+1}\in[\mathbb P^k(I_j)]^8$
	sharing this cell average and satisfying
	$B_{1,j}^{n+1}\equiv B_{\rm const}$ is
	\emph{locally strongly entropy stable} for $(\eta_r,q_r)$ if
	all its quadrature states belong to $\G_{B_{\rm const}}$ and
	\begin{equation}\label{eq:1d-strong-entropy}
		\mathcal S_{j,r}(\bm U_j^{n+1})\le\mathcal B_{j,r}^n,
		\qquad 1\le j\le N_x.
	\end{equation}
	The update is \emph{weakly} or \emph{strongly multi-entropy stable}
	if the respective conditions hold simultaneously for every $1\le r\le M$, using
	the same updated solution, time step, and interface wave speeds (omitting
	the prefix \emph{multi-} when $M=1$).
	A fully discrete scheme possesses the corresponding property if it holds for every
	admissible input satisfying the magnetic and quadrature assumptions under the stated time-step restriction.
\end{definition}

\begin{remark}[Hierarchy and distinction between the two entropy levels]
	\label{rem:weak-strong-distinction}
	Owing to quadrature exactness for polynomials and the convexity of $\eta_r$, Jensen's
	inequality yields
	\begin{equation}\label{eq:1d-entropy-hierarchy}
		\eta_r(\overline{\bm U}_j^{n+1})
		\le\mathcal S_{j,r}(\bm U_j^{n+1})
		\le\mathcal B_{j,r}^n
	\end{equation}
	whenever the polynomial update is strongly stable. Thus strong stability implies
	weak stability, with the two notions coinciding for piecewise constant approximations ($k=0$).
	For higher-order polynomials ($k\ge 1$), the weak inequality bounds the entropy of the cell average
	while leaving intra-element variations uncontrolled. Furthermore, its right-hand
	side contains $\mathcal S_{j,r}(\bm U_j^n)$, comparing
	different entropy functionals across successive time levels.
	Both definitions are formulated with respect to the prescribed quadrature rule, rather than
	continuous spatial entropy integrals.
\end{remark}

\begin{remark}[Global inequalities and temporal recursion]
	\label{rem:1d-global-entropy}
	Under the corresponding admissibility requirements, the
	\emph{global weak} and \emph{global strong} entropy inequalities read
	\begin{equation}\label{eq:1d-global-weak-boundary}
		\sum_{j=1}^{N_x}\Delta x\,\eta_r(\overline{\bm U}_j^{n+1})
		\le\sum_{j=1}^{N_x}\Delta x\,\mathcal S_{j,r}(\bm U_j^n)
		-\tau(\widehat q_{N_x+1/2,r}-\widehat q_{1/2,r}),
	\end{equation}
	\begin{equation}\label{eq:1d-global-strong-boundary}
		\sum_{j=1}^{N_x}\Delta x\,\mathcal S_{j,r}(\bm U_j^{n+1})
		\le\sum_{j=1}^{N_x}\Delta x\,\mathcal S_{j,r}(\bm U_j^n)
		-\tau(\widehat q_{N_x+1/2,r}-\widehat q_{1/2,r}),
	\end{equation}
	respectively. Summing the local inequalities over all cells cancels the interior
    numerical entropy fluxes and yields these global estimates. Under periodic boundary
	conditions, the boundary flux contributions cancel and \eqref{eq:1d-global-strong-boundary} simplifies to
	\begin{equation}\label{eq:1d-global-entropy}
		\sum_{j=1}^{N_x}\Delta x\,\mathcal S_{j,r}(\bm U_j^{n+1})
		\le\sum_{j=1}^{N_x}\Delta x\,\mathcal S_{j,r}(\bm U_j^n).
	\end{equation}
	Hence consecutive strongly stable forward Euler steps on a fixed mesh
	ensure that the total discrete quadrature entropy remains bounded by its initial value.
	In contrast, the weak inequality alone does not provide such a temporal recursion,
    which is why the weak-to-strong lifting of Section~\ref{sec:lifting} is needed.
	For SSP multistep methods, \eqref{eq:time-multistep-bound} defines the local entropy budget
	and \eqref{eq:time-ms-global} provides the corresponding global bound.
	Unless explicitly qualified as global, entropy stability is understood throughout in the local sense.
\end{remark}

\subsection{Finite-volume and DG discretizations}
\label{subsec:1d-scheme-pp}
Because the weak entropy analysis operates at the level of cell averages, it applies
equally to high-order finite-volume and DG schemes that share the same forward Euler update.
Before establishing the entropy estimates, we review the positivity-preserving (PP)
framework of Wu and Shu \cite{WuShu2019}.

\subsubsection{The common cell-average update}
With the time step $\tau>0$ and mesh ratio $\lambda=\tau/\Delta x$,
finite-volume and DG discretizations advance cell averages via the common forward Euler update
\begin{equation}\label{eq:1d-mean-update}
	\overline{\bm U}_j^{n+1}
	=\overline{\bm U}_j^n-\lambda
	(\widehat{\bm F}_{1,j+1/2}-\widehat{\bm F}_{1,j-1/2}),
\end{equation}
where $\widehat{\bm F}_{1,j+1/2} :=\widehat{\bm F}_1(\bm U_{j+1/2}^-,\bm U_{j+1/2}^+)$
is evaluated from the left and right interface traces
\[
\bm U_{j+1/2}^-:=\bm U_j^n(x_{j+1/2}),
\qquad
\bm U_{j+1/2}^+:=\bm U_{j+1}^n(x_{j+1/2}).
\]
Here and below, the time index on interface traces, numerical fluxes, and wave speeds is suppressed for brevity.

We employ the HLL numerical flux \cite{HartenLaxvanLeer1983}.
At a fixed interface $x_{j+1/2}$, abbreviate $\bm U^\pm=\bm U_{j+1/2}^\pm$.
Given wave speeds $a^-<0<a^+$, the two-sided HLL flux is given by
\begin{equation}\label{eq:1d-hll-flux}
	\widehat{\bm F}_1
	=\frac{a^+\bm F_1(\bm U^-)-a^-\bm F_1(\bm U^+)
		+a^+a^-(\bm U^+-\bm U^-)}{a^+-a^-}.
\end{equation}
The associated HLL intermediate state
\begin{equation}\label{eq:1d-hll-state}
	\bm H
	=\frac{a^+\bm U^+-a^-\bm U^-
		-\bm F_1(\bm U^+)+\bm F_1(\bm U^-)}{a^+-a^-}
\end{equation}
satisfies the consistency identities
\begin{equation}\label{eq:1d-hll-identities}
	\widehat{\bm F}_1
	=\bm F_1(\bm U^-)+a^-(\bm H-\bm U^-)
	=\bm F_1(\bm U^+)+a^+(\bm H-\bm U^+).
\end{equation}
Since both interface traces satisfy $B_1=B_{\rm const}$,
the flux component $(\widehat{\bm F}_1)_{B_1}$ vanishes identically.
Hence update \eqref{eq:1d-mean-update} preserves the cell average
of the normal magnetic field: $\overline{B}_{1,j}^{n+1}=B_{\rm const}$.

\subsubsection{Positivity preservation}

To establish physical admissibility, we specialize the quadrature rule in
\eqref{eq:1d-cell-average} to an
$L$-point Gauss--Lobatto formula satisfying $L\ge2$, $2L-3\ge k$, and
boundary weights $\widehat\omega_1=\widehat\omega_L=1/[L(L-1)]$.
Denote the nodal states and the two cell-boundary traces by
\begin{equation}\label{eq:1d-quadrature}
	\begin{aligned}
		\widehat{\bm U}_j^{[\alpha]}
		&:=\bm U_j^n(\widehat x_j^{[\alpha]}),
		&\overline{\bm U}_j^n
		&=\sum_{\alpha=1}^L\widehat\omega_\alpha
		\widehat{\bm U}_j^{[\alpha]},\\
		\bm U_{j,L}
		&:=\widehat{\bm U}_j^{[1]}=\bm U_{j-1/2}^+,
		&\bm U_{j,R}
		&:=\widehat{\bm U}_j^{[L]}=\bm U_{j+1/2}^-.
	\end{aligned}
\end{equation}
When $k=0$, we set $L=2$ with identical nodal states, so that interior
quadrature contributions vanish.

Let $\alpha_l,\alpha_r,\alpha_\star$ denote the splitting bounds
of Wu and Shu \cite[Eqs.~(12)--(14)]{WuShu2019}, where directional argument $1$ denotes $\bm e_1$. Define
\begin{equation}\label{eq:1d-cell-alpha}
	\alpha_j
	:=\max\{\alpha_\star(\bm U_{j,R},\bm U_{j,L};1),
	\alpha_\star(\bm U_{j,L},\bm U_{j,R};1)\}.
\end{equation}
At each interface, we impose the strict wave-speed bounds
\begin{equation}\label{eq:1d-weak-speed-assumptions}
	a^-<\min\{0,\alpha_l(\bm U^-,\bm U^+;1)\},
	\qquad
	a^+>\max\{0,\alpha_r(\bm U^+,\bm U^-;1)\}.
\end{equation}
For admissible interface states, these bounds ensure
$\bm H\in\G_{B_{\rm const}}$
\cite[Corollary~1 and Theorem~2]{WuShu2019}.
The following result provides the corresponding
Courant--Friedrichs--Lewy (CFL) condition for cell-average positivity.

\begin{proposition}[{Positivity of the cell average \cite[Theorem~4]{WuShu2019}}]
	\label{prop:1d-pp}
	Suppose that all quadrature states in \eqref{eq:1d-quadrature}
	and all exterior boundary traces belong to $\G_{B_{\rm const}}$,
	and that \eqref{eq:1d-weak-speed-assumptions} holds at every
	interface. If
	\begin{equation}\label{eq:1d-pp-cfl}
		\lambda\max\{\alpha_j+a_{j-1/2}^+,\,
		\alpha_j-a_{j+1/2}^-\}
		<\widehat\omega_1,
		\qquad 1\le j\le N_x,
	\end{equation}
	then the forward Euler update \eqref{eq:1d-mean-update} satisfies
	$\overline{\bm U}_j^{n+1}\in\G_{B_{\rm const}}$ in every cell.
\end{proposition}

\subsubsection{The HLL entropy flux and reference time step}
\label{subsubsec:1d-entropy-bound}
Both weak multi-entropy stability criteria rely on the HLL numerical entropy flux
and the local entropy budget \eqref{eq:1d-weak-bound}.
Setting $\eta_r^\pm:=\eta_r(\bm U^\pm)$ and
$q_r^\pm:=q_r(\bm U^\pm)$, the HLL entropy flux and its
associated intermediate entropy value are given by
\begin{equation}\label{eq:1d-hll-entropy}
	\begin{aligned}
		\widehat q_r
		=\frac{a^+q_r^--a^-q_r^++a^+a^-(\eta_r^+-\eta_r^-)}{a^+-a^-},\qquad 
		\mathscr E_r^{\rm HLL}
		=\frac{a^+\eta_r^+-a^-\eta_r^--q_r^++q_r^-}{a^+-a^-},
	\end{aligned}
\end{equation}
which satisfy the algebraic identities
\begin{equation}\label{eq:1d-hll-entropy-identities}
	\widehat q_r
	=q_r^-+a^-(\mathscr E_r^{\rm HLL}-\eta_r^-)
	=q_r^++a^+(\mathscr E_r^{\rm HLL}-\eta_r^+).
\end{equation}
The numerical flux $\widehat q_r$ is consistent with $q_r$ and single-valued across cell interfaces.

At each interface $F$, fixing a positive margin $s_F>0$, we specify the baseline speeds
\begin{equation}\label{eq:1d-initial-speeds}
	\begin{aligned}
		a^{-,0}=\min\{0,\alpha_l(\bm U^-,\bm U^+;1)\}-s_F,\qquad
		a^{+,0}=\max\{0,\alpha_r(\bm U^+,\bm U^-;1)\}+s_F,
	\end{aligned}
\end{equation}
with initial wave fan width $C_F^0:=a^{+,0}-a^{-,0}>0$.

\begin{remark}[Strict admissibility of the baseline state]
	\label{rem:1d-strict-speed-margin}
	For $\bm U^\pm\in\G_{B_{\rm const}}$, consider the HLL state
	\eqref{eq:1d-hll-state} associated with zero-margin speeds:
	\[
	a^- =\min\{0,\alpha_l(\bm U^-,\bm U^+;1)\},\qquad
	a^+ =\max\{0,\alpha_r(\bm U^+,\bm U^-;1)\}.
	\]
	By \cite[Corollary~1]{WuShu2019}, this state possesses positive density
	and nonnegative internal energy.
	Owing to the margin $s_F>0$ in \eqref{eq:1d-initial-speeds}, 
	the baseline intermediate state $\bm H^0$ is a convex
	combination of this zero-margin state and the strictly admissible
	midpoint $(\bm U^-+\bm U^+)/2$, with the latter carrying positive weight.
	The strict concavity of the internal energy ensures strictly
	positive pressure, implying that $\bm H^0\in\G_{B_{\rm const}}$
	and that all associated entropies $\eta_r(\bm H^0)$ are well-defined.
\end{remark}

Both criteria adjust these baseline speeds while maintaining
the admissibility requirement \eqref{eq:1d-weak-speed-assumptions}. Once the speeds
are chosen, the physical and entropy numerical fluxes are fixed.
For a CFL safety factor $\sigma\in(0,1)$, we define the reference time-step bound
\begin{equation}\label{eq:1d-weak-reference-step}
	\lambda_j^0
	:=\frac{\sigma\widehat\omega_1}
	{\max\{\alpha_j+a_{j-1/2}^+,\,\alpha_j-a_{j+1/2}^-\}}>0.
\end{equation}
Any common step size satisfying $0<\lambda\le\min_j\lambda_j^0$ satisfies the positivity CFL condition \eqref{eq:1d-pp-cfl} strictly. Below, we write $[z]_+:=\max\{z,0\}$.
The first criterion establishes weak stability by bounding intermediate-state
entropies within a convex decomposition, whereas the second criterion
estimates the relative entropy defect in terms of the nodal states.

\subsection{A convex-decomposition criterion}
\label{subsec:1d-weak-convex}

Substituting the HLL flux identities \eqref{eq:1d-hll-identities} into the
cell-average update \eqref{eq:1d-mean-update} and regrouping the endpoint
contributions as in \cite[Theorem~4]{WuShu2019}, we introduce the auxiliary coefficients
\begin{equation}\label{eq:1d-beta}
	\begin{aligned}
		\beta_{j,L}:=\frac{\widehat\omega_1}{\lambda}-a_{j-1/2}^+,\qquad
		\beta_{j,R}:=\frac{\widehat\omega_1}{\lambda}+a_{j+1/2}^-,\qquad
		D_j:=\beta_{j,L}+\beta_{j,R}.
	\end{aligned}
\end{equation}
With the interior splitting state defined by
\begin{equation}\label{eq:1d-interior-state}
	\bm\Xi_j
	:=\frac{\beta_{j,R}\bm U_{j,R}+\beta_{j,L}\bm U_{j,L}
		-\bm F_1(\bm U_{j,R})+\bm F_1(\bm U_{j,L})}{D_j},
\end{equation}
the cell-average update can be recast as
\begin{equation}\label{eq:1d-pp-decomposition}
	\overline{\bm U}_j^{n+1}
	=\sum_{\alpha=2}^{L-1}\widehat\omega_\alpha
	\widehat{\bm U}_j^{[\alpha]}
	+\lambda D_j\bm\Xi_j-\lambda a_{j+1/2}^-\bm H_{j+1/2}
	+\lambda a_{j-1/2}^+\bm H_{j-1/2}.
\end{equation}
Under the wave-speed bounds \eqref{eq:1d-weak-speed-assumptions} and the
CFL condition \eqref{eq:1d-pp-cfl}, we have $\beta_{j,L},\beta_{j,R}>\alpha_j$.
Consequently, all intermediate states in \eqref{eq:1d-pp-decomposition} belong to
$\G_{B_{\rm const}}$ by the splitting results of \cite[Corollaries~1--2 and Theorem~2]{WuShu2019}.
Furthermore, the expansion coefficients are strictly positive and sum to unity, since
$\lambda D_j-\lambda a_{j+1/2}^-+\lambda a_{j-1/2}^+=2\widehat\omega_1$ and $\sum_{\alpha=1}^L\widehat\omega_\alpha=1$,
confirming that \eqref{eq:1d-pp-decomposition} is a convex decomposition.

To obtain a matching decomposition for the local entropy budget,
we associate with $\bm\Xi_j$ the entropy value
\begin{equation}\label{eq:1d-weak-interior-entropy}
	\mathscr E_{j,r}^{\Xi}
	:=\frac{\beta_{j,R}\eta_r(\bm U_{j,R})
		+\beta_{j,L}\eta_r(\bm U_{j,L})
		-q_r(\bm U_{j,R})+q_r(\bm U_{j,L})}{D_j}.
\end{equation}
Using the entropy flux identities \eqref{eq:1d-hll-entropy-identities},
the local entropy budget \eqref{eq:1d-weak-bound} admits the matching representation
\begin{equation}\label{eq:1d-weak-bound-decomposition}
	\mathcal B_{j,r}^n
	=\sum_{\alpha=2}^{L-1}\widehat\omega_\alpha
	\eta_r(\widehat{\bm U}_j^{[\alpha]})
	+\lambda D_j\mathscr E_{j,r}^{\Xi}-\lambda a_{j+1/2}^-\mathscr E_{j+1/2,r}^{\rm HLL}
	+\lambda a_{j-1/2}^+\mathscr E_{j-1/2,r}^{\rm HLL}.
\end{equation}
Because \eqref{eq:1d-pp-decomposition} and \eqref{eq:1d-weak-bound-decomposition}
share identical convex weights, establishing weak entropy stability reduces to
verifying entropy inequalities for the intermediate states.

\begin{theorem}[Weak stability by convex decomposition]
	\label{thm:1d-weak-convex}
	Suppose that all quadrature states at time level $t^n$ and all exterior boundary traces
	belong to $\G_{B_{\rm const}}$. Under \eqref{eq:1d-weak-speed-assumptions} and
	\eqref{eq:1d-pp-cfl}, assume further that
	\begin{equation}\label{eq:1d-weak-state-conditions}
		\eta_r(\bm H_{j+1/2})\le\mathscr E_{j+1/2,r}^{\rm HLL},
		\qquad
		\eta_r(\bm\Xi_j)\le\mathscr E_{j,r}^{\Xi},
		\qquad 1\le r\le M.
	\end{equation}
	Then the cell-average update \eqref{eq:1d-mean-update} is weakly multi-entropy stable in
	the sense of Definition~\ref{def:weak-strong-entropy} with respect to the budgets \eqref{eq:1d-weak-bound}.
\end{theorem}

\begin{proof}
	Admissibility of the updated cell average $\overline{\bm U}_j^{n+1}\in\G_{B_{\rm const}}$ follows from
	Proposition~\ref{prop:1d-pp}.
	Applying the convexity of $\eta_r$ to \eqref{eq:1d-pp-decomposition}
	and substituting the intermediate bounds \eqref{eq:1d-weak-state-conditions} into
	\eqref{eq:1d-weak-bound-decomposition} yields
	$\eta_r(\overline{\bm U}_j^{n+1})\le\mathcal B_{j,r}^n$
	simultaneously for all $1\le r\le M$.
\end{proof}

The intermediate state conditions \eqref{eq:1d-weak-state-conditions} can always be
satisfied with finite wave speeds and a strictly positive uniform time step.
The construction rests on an elementary convexity estimate: for any $\bm Z_0,\bm M\in\G$,
$E_{0,r},\overline\eta_r\in\R$, $C>0$, and $t\ge0$,
\begin{equation}\label{eq:1d-weak-completion-bound}
	CE_{0,r}+t\overline\eta_r
	-(C+t)\eta_r\left(\frac{C\bm Z_0+t\bm M}{C+t}\right)
	\ge C[E_{0,r}-\eta_r(\bm Z_0)]
	+t[\overline\eta_r-\eta_r(\bm M)].
\end{equation}

\begin{proposition}[Finite wave speeds and a positive common time step]
	\label{prop:1d-weak-convex-attainability}
	On a finite mesh, suppose that all quadrature states in
	\eqref{eq:1d-quadrature} and all exterior boundary traces belong
	to $\G_{B_{\rm const}}$.
	Let $\{f_r\}_{r=1}^M\subset C^2(\R)$ satisfy
	\eqref{eq:intro-harten-conditions}.
	Then there exist finite interface wave speeds
	$a_{j+1/2}^-<0<a_{j+1/2}^+$ and a threshold
	$\lambda_*^{\rm cvx}>0$ such that, for every
	$0<\lambda\le\lambda_*^{\rm cvx}$, all hypotheses of
	Theorem~\ref{thm:1d-weak-convex} hold simultaneously
	in every cell and for every $1\le r\le M$.
\end{proposition}

\begin{proof}
	\emph{Interface inequalities.} Let $\bm H^0$ and
	$\mathscr E_r^{{\rm HLL},0}$ be the intermediate state and entropy value associated with the
	baseline speeds \eqref{eq:1d-initial-speeds}. Choose a parameter $\theta_F\in(0,1)$ common to all entropies
	at interface $F$, such as $\theta_F=-a^{-,0}/C_F^0$ or $\theta_F=1/2$.
	Under reversal of the interface orientation, these choices transform as $\theta_F \mapsto 1-\theta_F$.
	Define
	\[
	\begin{aligned}
		\bm M_F&=\theta_F\bm U^-+(1-\theta_F)\bm U^+,\qquad 
		\overline\eta_{F,r} =\theta_F\eta_r^-+(1-\theta_F)\eta_r^+,\\
		J_{F,r}&=\overline\eta_{F,r}-\eta_r(\bm M_F),\qquad\qquad 
		g_{F,r}^0 =C_F^0[\mathscr E_r^{{\rm HLL},0}-\eta_r(\bm H^0)].
	\end{aligned}
	\]
	For distinct traces $\bm U^-\ne\bm U^+$, set
	\begin{equation}\label{eq:1d-weak-interface-completion}
		\zeta_F=\max_{1\le r\le M}\frac{[-g_{F,r}^0]_+}{J_{F,r}},
		\qquad a^-=a^{-,0}-\theta_F\zeta_F,
		\qquad a^+=a^{+,0}+(1-\theta_F)\zeta_F.
	\end{equation}
	Because $\zeta_F\ge0$, these speed enlargements preserve the strict PP bounds \eqref{eq:1d-weak-speed-assumptions}.
	The resulting intermediate state and entropy flux value satisfy the convex mixtures
	\begin{equation}\label{eq:1d-weak-interface-mixing}
		\bm H=\frac{C_F^0\bm H^0+\zeta_F\bm M_F}{C_F^0+\zeta_F},
		\qquad
		\mathscr E_r^{\rm HLL}
		=\frac{C_F^0\mathscr E_r^{{\rm HLL},0}
			+\zeta_F\overline\eta_{F,r}}{C_F^0+\zeta_F}.
	\end{equation}
	Applying the convexity estimate \eqref{eq:1d-weak-completion-bound} gives
	\[
	(C_F^0+\zeta_F)
	[\mathscr E_r^{\rm HLL}-\eta_r(\bm H)]
	\ge g_{F,r}^0+\zeta_FJ_{F,r}\ge0.
	\]
	Since $\eta_r$ is strictly convex on $\G_{B_{\rm const}}$ and $\theta_F\in(0,1)$,
	Jensen's inequality ensures $J_{F,r}>0$ whenever $\bm U^-\ne\bm U^+$, so
	$\zeta_F$ and the adjusted speeds $a^\pm$ remain finite.
	When the traces coincide ($\bm U^-=\bm U^+$), we set $\zeta_F=0$; by
consistency the entropy inequality reduces to an equality, and no division by zero occurs.
	
	\emph{Interior inequalities.} With the interface wave speeds and numerical
	fluxes fixed, let $\lambda_j^0$ be given by
	\eqref{eq:1d-weak-reference-step}, where superscript $0$ denotes
	quantities evaluated at $\lambda_j^0$. Define
	\begin{equation}\label{eq:1d-weak-interior-gaps}
		\begin{aligned}
			\bm M_j&=\tfrac12(\bm U_{j,L}+\bm U_{j,R}),\qquad
			\overline\eta_{j,r}
			=\tfrac12[\eta_r(\bm U_{j,L})+\eta_r(\bm U_{j,R})],\\
			J_{j,r}&=\overline\eta_{j,r}-\eta_r(\bm M_j),\qquad 
			g_{j,r}^0=D_j^0[\mathscr E_{j,r}^{\Xi,0}-\eta_r(\bm\Xi_j^0)].
		\end{aligned}
	\end{equation}
	For $0<\lambda\le\lambda_j^0$ and $\nu\in\{L,R\}$,
	decompose the coefficients in \eqref{eq:1d-beta} as
	$\beta_{j,\nu}(\lambda)=\beta_{j,\nu}^0+\delta_j(\lambda)$ with
	$\delta_j(\lambda):=\widehat\omega_1(1/\lambda-1/\lambda_j^0)\ge0$.
	Substituting these relations into \eqref{eq:1d-interior-state} and
	\eqref{eq:1d-weak-interior-entropy} yields the convex representations
	\begin{equation}\label{eq:1d-weak-interior-mixing}
		\begin{aligned}
			D_j(\lambda)=D_j^0+2\delta_j(\lambda),\quad
			D_j(\lambda)\bm\Xi_j(\lambda)
			&=D_j^0\bm\Xi_j^0+2\delta_j(\lambda)\bm M_j,\\
			D_j(\lambda)\mathscr E_{j,r}^{\Xi}(\lambda)
			&=D_j^0\mathscr E_{j,r}^{\Xi,0}
			+2\delta_j(\lambda)\overline\eta_{j,r}.
		\end{aligned}
	\end{equation}
	Applying \eqref{eq:1d-weak-completion-bound} with
	$C=D_j^0$ and $t=2\delta_j(\lambda)$ yields
	\begin{equation}\label{eq:1d-weak-interior-defect-bound}
		D_j(\lambda)[\mathscr E_{j,r}^{\Xi}(\lambda)
		-\eta_r(\bm\Xi_j(\lambda))]
		\ge g_{j,r}^0+2\delta_j(\lambda)J_{j,r}.
	\end{equation}
	For distinct endpoint states ($\bm U_{j,L}\ne\bm U_{j,R}$), strict convexity ensures $J_{j,r}>0$.
	We then set
	\begin{equation}\label{eq:1d-weak-interior-completion}
		\zeta_j:=\max_{1\le r\le M}\frac{[-g_{j,r}^0]_+}{J_{j,r}},\qquad
		\lambda_j^{\rm cvx}
		:=\frac{\widehat\omega_1}
		{\widehat\omega_1/\lambda_j^0+\zeta_j/2}>0.
	\end{equation}
	Consequently, for every $0<\lambda\le\lambda_j^{\rm cvx}$, we have
	$2\delta_j(\lambda)\ge\zeta_j$, so that
	\eqref{eq:1d-weak-interior-defect-bound} guarantees
	$\eta_r(\bm\Xi_j(\lambda))
	\le\mathscr E_{j,r}^{\Xi}(\lambda)$ simultaneously for all $1\le r\le M$.
	If the endpoint states coincide, the interior state reduces to $\bm\Xi_j=\bm U_{j,L}$ and
	$\mathscr E_{j,r}^{\Xi}=\eta_r(\bm U_{j,L})$; we then set
	$\zeta_j=0$ and $\lambda_j^{\rm cvx}=\lambda_j^0$.
	Because the mesh and entropy family are finite, the positive threshold
	\begin{equation}\label{eq:1d-weak-convex-cfl}
		\lambda_*^{\rm cvx}:=\min_{1\le j\le N_x}\lambda_j^{\rm cvx}>0
	\end{equation}
	provides a nonempty interval of uniform time steps $0<\lambda\le\lambda_*^{\rm cvx}$ satisfying both the
	positivity CFL condition and all interior entropy inequalities. The interface
	inequalities remain valid because the wave speeds are held fixed, which concludes the proof.
\end{proof}

\subsection{A criterion based on relative-entropy rates}
\label{subsec:1d-weak-relative}

Unlike the convex-decomposition approach, the second criterion directly estimates
the cell-average entropy defect without requiring individual entropy bounds on the
intermediate states.
Given admissible initial data and interface wave speeds satisfying
\eqref{eq:1d-weak-speed-assumptions}, we first establish an exact algebraic identity
for the entropy defect, and then construct wave speeds that rigorously control its
boundary contributions, keeping the numerical fluxes fixed as the time step varies.

\subsubsection{An exact identity for the entropy defect}
\label{subsubsec:1d-relative-defect}

For $\bm U,\bm W\in\G_{B_{\rm const}}$, we define the entropy
variables, entropy potential, and relative entropy (or Bregman divergence) by
\begin{equation}\label{eq:1d-weak-relative-quantities}
	\begin{aligned}
		\bm V_r(\bm U)&:=\nabla\eta_r(\bm U),\qquad 
		\psi_r(\bm U) :=\bm V_r(\bm U)^\top\bm F_1(\bm U)-q_r(\bm U),\\
		\Hrel_r(\bm U\mid\bm W)
		&:=\eta_r(\bm U)-\eta_r(\bm W)
		-\bm V_r(\bm W)^\top(\bm U-\bm W).
	\end{aligned}
\end{equation}
Here the gradient is evaluated with respect to all eight conserved
variables on $\G$. By the strict convexity of $\eta_r$ on
$\G_{B_{\rm const}}$, one has $\Hrel_r(\bm U\mid\bm W)\ge0$,
with equality holding if and only if $\bm U=\bm W$.
For any reference state $\bm W\in\G_{B_{\rm const}}$, we define the boundary defect terms
\begin{equation}\label{eq:1d-weak-side-defects}
	\begin{aligned}
		P_{j,L,r}(\bm W)
		&:=-\psi_r(\bm W)+\bm V_r(\bm W)^\top
		\widehat{\bm F}_{1,j-1/2}
		-\widehat q_{j-1/2,r},\\
		P_{j,R,r}(\bm W)
		&:=\psi_r(\bm W)-\bm V_r(\bm W)^\top
		\widehat{\bm F}_{1,j+1/2}
		+\widehat q_{j+1/2,r}.
	\end{aligned}
\end{equation}

For any time step satisfying the positivity CFL condition \eqref{eq:1d-pp-cfl},
the updated cell average satisfies $\overline{\bm U}_j^{n+1}\in\G_{B_{\rm const}}$.
Setting $\bm W=\overline{\bm U}_j^{n+1}$, the quadrature representation of $\overline{\bm U}_j^n$
and update \eqref{eq:1d-mean-update} yield the exact defect identity
\begin{equation}\label{eq:1d-weak-defect-identity}
	\mathcal B_{j,r}^n-\eta_r(\bm W)
	=\sum_{\alpha=1}^L\widehat\omega_\alpha
	\Hrel_r(\widehat{\bm U}_j^{[\alpha]}\mid\bm W)
	-\lambda\bigl[P_{j,L,r}(\bm W)+P_{j,R,r}(\bm W)\bigr].
\end{equation}
Indeed, expanding the relative entropies on the right-hand side of
\eqref{eq:1d-weak-defect-identity} and substituting
$\overline{\bm U}_j^n-\bm W
=\lambda(\widehat{\bm F}_{1,j+1/2}-\widehat{\bm F}_{1,j-1/2})$,
the entropy potential terms $\pm\psi_r(\bm W)$ cancel identically.
Verifying weak entropy stability therefore reduces to controlling the boundary
defects $P_{j,L,r}(\bm W)$ and $P_{j,R,r}(\bm W)$ via the relative entropies of the
corresponding cell boundary traces.

\subsubsection{Relative-entropy control of the boundary terms}
\label{subsubsec:1d-relative-control}

For fixed wave speeds, let $\lambda_j^0$ be the reference time-step bound
in \eqref{eq:1d-weak-reference-step}. Consider the parameterized update trajectory
and its compact image
\begin{equation}\label{eq:1d-weak-update-path}
	\begin{aligned}
		\bm W_j(\vartheta)
		:=\overline{\bm U}_j^n
		-\vartheta(\widehat{\bm F}_{1,j+1/2}
		-\widehat{\bm F}_{1,j-1/2}),\qquad 
		\mathscr W_j
		:=\{\bm W_j(\vartheta):0\le\vartheta\le\lambda_j^0\}.
	\end{aligned}
\end{equation}
By Proposition~\ref{prop:1d-pp} and the convexity of $\G_{B_{\rm const}}$, $\mathscr W_j$
is a compact line segment contained in $\G_{B_{\rm const}}$, encompassing every updated cell average
$\overline{\bm U}_j^{n+1}$ for $0<\lambda\le\lambda_j^0$.

Suppose there exist finite nonnegative rates $\mathfrak c_{j,\nu,r}$ such that
\begin{equation}\label{eq:1d-weak-rate-bound}
	P_{j,\nu,r}(\bm W)
	\le\mathfrak c_{j,\nu,r}\Hrel_r(\bm U_{j,\nu}\mid\bm W),
	\qquad \bm W\in\mathscr W_j,\quad \nu\in\{L,R\}.
\end{equation}
Set $\mathfrak c_{j,\nu}:=\max_{1\le r\le M}\mathfrak c_{j,\nu,r}$.
Since $\widehat\omega_L=\widehat\omega_1$, substituting \eqref{eq:1d-weak-rate-bound} into
\eqref{eq:1d-weak-defect-identity} with $\bm W=\overline{\bm U}_j^{n+1}$ yields
\begin{equation}\label{eq:1d-weak-controlled-defect}
	\mathcal B_{j,r}^n-\eta_r(\bm W)
	\ge\sum_{\alpha=2}^{L-1}\widehat\omega_\alpha
	\Hrel_r(\widehat{\bm U}_j^{[\alpha]}\mid\bm W)+\sum_{\nu\in\{L,R\}}
	(\widehat\omega_1-\lambda\mathfrak c_{j,\nu,r})
	\Hrel_r(\bm U_{j,\nu}\mid\bm W).
\end{equation}
The weak entropy inequalities $\eta_r(\overline{\bm U}_j^{n+1})\le\mathcal B_{j,r}^n$
follow immediately for all $1\le r\le M$ provided the time step satisfies
\begin{equation}\label{eq:1d-weak-relative-cfl}
	0<\lambda\le\min_j\lambda_j^0, \qquad 
	\lambda\mathfrak c_{j,L}\le\widehat\omega_1,
	\qquad
	\lambda\mathfrak c_{j,R}\le\widehat\omega_1
	\quad\text{for every }j.
\end{equation}

\subsubsection{Finite relative-entropy rates and wave-speed selection}
\label{subsubsec:1d-relative-finiteness}

We define the optimal rates by
\begin{equation}\label{eq:1d-weak-relative-rates}
	\mathfrak c_{j,\nu,r}
	:=\sup_{\substack{\bm W\in\mathscr W_j\\\bm W\ne\bm U_{j,\nu}}}
	\frac{[P_{j,\nu,r}(\bm W)]_+}
	{\Hrel_r(\bm U_{j,\nu}\mid\bm W)},
	\qquad \nu\in\{L,R\},
\end{equation}
with the supremum understood as zero if the set is empty.
Because $\mathscr W_j$ is compact and $\Hrel_r(\bm U_{j,\nu}\mid\bm W)>0$ for $\bm W\ne\bm U_{j,\nu}$,
the quotient in \eqref{eq:1d-weak-relative-rates} can become unbounded only as
$\bm W\to\bm U_{j,\nu}$. Furthermore, if the trace $\bm U_{j,\nu}$ belongs to $\mathscr W_j$,
bound \eqref{eq:1d-weak-rate-bound} also requires the sign condition $P_{j,\nu,r}(\bm U_{j,\nu})\le0$.

At each interface, we introduce the one-sided entropy dissipations
\begin{equation}\label{eq:1d-weak-one-sided-dissipation}
	\begin{aligned}
		d_r^-:=\bm V_r(\bm U^-)^\top\widehat{\bm F}_1
		-\widehat q_r-\psi_r(\bm U^-),\qquad 
		d_r^+ :=\psi_r(\bm U^+)
		-\bm V_r(\bm U^+)^\top\widehat{\bm F}_1+\widehat q_r.
	\end{aligned}
\end{equation}
The boundary defect terms at the cell-boundary traces satisfy
\begin{equation}\label{eq:1d-weak-trace-defects}
	P_{j,L,r}(\bm U_{j,L})=-d_{j-1/2,r}^+,
	\qquad
	P_{j,R,r}(\bm U_{j,R})=-d_{j+1/2,r}^-.
\end{equation}
When the one-sided dissipation is strictly positive, continuity ensures that
$[P_{j,\nu,r}(\bm W)]_+$ vanishes identically in a neighborhood of the trace.
If the dissipation vanishes, however, a first-order linear term in the Taylor expansion of $P_{j,\nu,r}$
could obstruct quadratic relative-entropy control. This linear term vanishes
whenever the numerical state and entropy fluxes coincide with the physical fluxes at that trace.
We therefore impose, at each interface and for each entropy, the alternative condition
\begin{equation}\label{eq:1d-weak-side-condition}
	\begin{gathered}
		d_r^->0\quad\text{or}\quad
		(\widehat{\bm F}_1,\widehat q_r)
		=(\bm F_1(\bm U^-),q_r(\bm U^-)),\\
		d_r^+>0\quad\text{or}\quad
		(\widehat{\bm F}_1,\widehat q_r)
		=(\bm F_1(\bm U^+),q_r(\bm U^+)).
	\end{gathered}
\end{equation}

To analyze the case where numerical fluxes equal physical fluxes, define the relative entropy flux
\begin{equation}\label{eq:1d-weak-relative-entropy-flux}
	\Qrel_r(\bm U\mid\bm W)
	:=q_r(\bm U)-q_r(\bm W)
	-\bm V_r(\bm W)^\top[\bm F_1(\bm U)-\bm F_1(\bm W)].
\end{equation}
On the constant-$B_1$ slice $\G_{B_{\rm const}}$, direct differentiation of
\eqref{eq:intro-harten-entropy} yields the compatibility identity
\begin{equation}\label{eq:1d-compatibility}
	Dq_r(\bm U)[\bm Z]
	=\bm V_r(\bm U)^\top D\bm F_1(\bm U)[\bm Z],
	\qquad z_{B_1}=0,
	\quad \bm U\in\G_{B_{\rm const}},
\end{equation}
for every tangent vector $\bm Z\in\R^8$ with vanishing magnetic component $z_{B_1}=0$, and for all $r$.
By Lemma~\ref{lem:harten-convexity}, $D^2\eta_r$ is positive definite on this tangent space.
Thus, on each compact convex subset of $\G_{B_{\rm const}}$, Taylor expansions ensure the existence
of constants $c_r,C_r>0$ such that
\begin{equation}\label{eq:1d-weak-quadratic-bounds}
	\Hrel_r(\bm U\mid\bm W)\ge c_r\norm{\bm U-\bm W}^2,
	\qquad
	|\Qrel_r(\bm U\mid\bm W)|\le C_r\norm{\bm U-\bm W}^2,
\end{equation}
where compatibility relation \eqref{eq:1d-compatibility} eliminates the first-order linear term in $\Qrel_r$.

\begin{lemma}[Finiteness of the relative-entropy rates]
	\label{lem:1d-weak-finite-rates}
	For fixed wave speeds satisfying
	\eqref{eq:1d-weak-speed-assumptions} and
	\eqref{eq:1d-weak-side-condition}, the rates defined in
	\eqref{eq:1d-weak-relative-rates} are finite and satisfy
	\eqref{eq:1d-weak-rate-bound} for every $\bm W\in\mathscr W_j$.
\end{lemma}

\begin{proof}
	If the corresponding one-sided dissipation is strictly positive,
	\eqref{eq:1d-weak-trace-defects} and continuity ensure that
	$[P_{j,\nu,r}(\bm W)]_+$ vanishes identically in a neighborhood of $\bm U_{j,\nu}$.
	Under the alternative flux equality in \eqref{eq:1d-weak-side-condition},
	substitution into \eqref{eq:1d-weak-side-defects} yields
	\[
	P_{j,L,r}(\bm W)=-\Qrel_r(\bm U_{j,L}\mid\bm W),
	\qquad\text{or}\qquad
	P_{j,R,r}(\bm W)=\Qrel_r(\bm U_{j,R}\mid\bm W).
	\]
	Applying the quadratic bounds \eqref{eq:1d-weak-quadratic-bounds} on the compact convex hull
	of $\mathscr W_j\cup\{\bm U_{j,\nu}\}$, we deduce that the quotient in \eqref{eq:1d-weak-relative-rates}
	is uniformly bounded by $C_r/c_r$ as $\bm W\to\bm U_{j,\nu}$.
	Away from the trace, boundedness is guaranteed by compactness.
	Hence the rates are finite and their definition implies \eqref{eq:1d-weak-rate-bound} for
	$\bm W\ne\bm U_{j,\nu}$. At the trace point itself, the estimate holds trivially since
	$P_{j,\nu,r}(\bm U_{j,\nu})\le0$ and $\Hrel_r(\bm U_{j,\nu}\mid\bm U_{j,\nu})=0$.
\end{proof}

We now construct interface wave speeds that satisfy
\eqref{eq:1d-weak-side-condition} while preserving the strict positivity bounds \eqref{eq:1d-weak-speed-assumptions}.
For distinct interface traces ($\bm U^-\ne\bm U^+$), define the directional relative wave speeds
\begin{equation}\label{eq:1d-weak-interface-ratios}
	\ell_r^+
	:=\frac{\Qrel_r(\bm U^+\mid\bm U^-)}
	{\Hrel_r(\bm U^+\mid\bm U^-)},
	\qquad
	\ell_r^-
	:=\frac{\Qrel_r(\bm U^-\mid\bm U^+)}
	{\Hrel_r(\bm U^-\mid\bm U^+)}.
\end{equation}
Substituting the HLL numerical state and entropy fluxes into \eqref{eq:1d-weak-one-sided-dissipation}
yields the factorization
\begin{equation}\label{eq:1d-weak-dissipation-factors}
	\begin{aligned}
		d_r^-=\frac{-a^-}{a^+-a^-}(a^+-\ell_r^+)
		\Hrel_r(\bm U^+\mid\bm U^-),\qquad 
		d_r^+=\frac{a^+}{a^+-a^-}(\ell_r^--a^-)
		\Hrel_r(\bm U^-\mid\bm U^+).
	\end{aligned}
\end{equation}
Enforcing $a^-<\ell_r^-$ and $a^+>\ell_r^+$ guarantees $d_r^\pm>0$ for every $1\le r\le M$.
Starting from the baseline speeds in \eqref{eq:1d-initial-speeds}, set
\begin{equation}\label{eq:1d-weak-speed-envelope-data}
	\Lambda_-:=\min_{1\le r\le M}\ell_r^-,
	\qquad
	\Lambda_+:=\max_{1\le r\le M}\ell_r^+,
\end{equation}
and define the final speeds via the hyperbolic envelopes
\begin{equation}\label{eq:1d-weak-relative-speeds}
	\begin{aligned}
		a^-:=\frac{a^{-,0}+\Lambda_-
			-\sqrt{(a^{-,0}-\Lambda_-)^2+(C_F^0)^2}}{2},\quad 
		a^+:=\frac{a^{+,0}+\Lambda_+
			+\sqrt{(a^{+,0}-\Lambda_+)^2+(C_F^0)^2}}{2}.
	\end{aligned}
\end{equation}
Because $\sqrt{z^2+(C_F^0)^2}>|z|$ for any $z\in\R$ and $C_F^0>0$, we immediately obtain
\[
a^-<\min\{a^{-,0},\Lambda_-\},
\qquad
a^+>\max\{a^{+,0},\Lambda_+\}.
\]
Consequently, the positivity bounds \eqref{eq:1d-weak-speed-assumptions} are preserved,
and \eqref{eq:1d-weak-dissipation-factors} yields $d_r^\pm>0$, with $C_F^0$ providing a strictly positive separation margin.
By \eqref{eq:1d-weak-quadratic-bounds}, the ratios $\ell_r^\pm$ remain uniformly bounded on
compact subsets of $\G_{B_{\rm const}}$ as $\bm U^-\to\bm U^+$. The selected speeds are therefore
finite for any fixed admissible data and finite entropy family.
When traces coincide ($\bm U^-=\bm U^+$), we set $a^-:=a^{-,0}$ and $a^+:=a^{+,0}$;
consistency then yields $d_r^\pm=0$ and the alternative in \eqref{eq:1d-weak-side-condition} holds.

With the interface speeds satisfying \eqref{eq:1d-weak-speed-assumptions} and
\eqref{eq:1d-weak-side-condition} held fixed across the mesh, define the critical CFL bound
\begin{equation}\label{eq:1d-weak-relative-step-bound}
	\lambda_*^{\rm rel}
	:=\min_{1\le j\le N_x}\left\{\lambda_j^0,
	\frac{\widehat\omega_1}{\mathfrak c_{j,L}},
	\frac{\widehat\omega_1}{\mathfrak c_{j,R}}\right\},
\end{equation}
with division by zero interpreted as $+\infty$.
Since $\lambda_j^0>0$ and the rates are finite by Lemma~\ref{lem:1d-weak-finite-rates},
the threshold $\lambda_*^{\rm rel}$ is strictly positive on any finite mesh.
Every time step satisfying $0<\lambda\le\lambda_*^{\rm rel}$ complies with \eqref{eq:1d-weak-relative-cfl}.
Computing this threshold directly requires certified upper bounds for the suprema in
\eqref{eq:1d-weak-relative-rates}. Alternatively, keeping the selected fluxes fixed and halving
a trial step until \eqref{eq:1d-weak-entropy} is satisfied gives a practical procedure that
terminates in exact arithmetic because $\lambda_*^{\rm rel}>0$.

\begin{theorem}[Weak stability by relative-entropy rates]
	\label{thm:1d-weak-relative}
	Suppose that all quadrature states at time level $t^n$ and all exterior boundary traces
	belong to $\G_{B_{\rm const}}$, and that the interface speeds satisfy
	\eqref{eq:1d-weak-speed-assumptions} and \eqref{eq:1d-weak-side-condition}.
	Then the cell-average update \eqref{eq:1d-mean-update} is weakly multi-entropy
	stable in the sense of Definition~\ref{def:weak-strong-entropy}
	for every $0<\lambda\le\lambda_*^{\rm rel}$, where
	$\lambda_*^{\rm rel}$ is defined by \eqref{eq:1d-weak-relative-step-bound}.
\end{theorem}

\begin{proof}
	By Proposition~\ref{prop:1d-pp} and the reference-step bound, the updated cell average satisfies
	$\overline{\bm U}_j^{n+1}=\bm W_j(\lambda)\in\mathscr W_j\subset\G_{B_{\rm const}}$.
	Lemma~\ref{lem:1d-weak-finite-rates} provides the boundary estimates \eqref{eq:1d-weak-rate-bound}.
	For every $\lambda\le\lambda_*^{\rm rel}$, all coefficients in \eqref{eq:1d-weak-controlled-defect}
	are nonnegative, establishing \eqref{eq:1d-weak-entropy} simultaneously for all $1\le r\le M$.
\end{proof}

\subsection{Comparison of the two criteria}
\label{subsec:1d-weak-comparison}

Although both criteria control the same cell-average entropy defect
$\mathcal B_{j,r}^n-\eta_r(\overline{\bm U}_j^{n+1})$, they do so by different mechanisms. The identities below relate their interface conditions and bring out the structural difference between their cellwise estimates.

At a given interface $F$, define the intermediate entropy gap
$e_{F,r}:=\mathscr E_r^{\rm HLL}-\eta_r(\bm H)$.
Combining the HLL state identities \eqref{eq:1d-hll-identities} with the
entropy flux identities \eqref{eq:1d-hll-entropy-identities}, direct algebraic manipulation
yields the interface bridge relations
\begin{equation}\label{eq:1d-weak-interface-bridge}
	\begin{aligned}
		d_r^-=(-a^-)\bigl[e_{F,r}+\Hrel_r(\bm H\mid\bm U^-)\bigr],\qquad 
		d_r^+=a^+\bigl[e_{F,r}+\Hrel_r(\bm H\mid\bm U^+)\bigr].
	\end{aligned}
\end{equation}
Because $a^-<0<a^+$ and $\Hrel_r\ge0$, the convex-decomposition condition $e_{F,r}\ge0$
guarantees the nonnegativity of the one-sided dissipations, $d_r^\pm\ge0$.
Furthermore, if either dissipation vanishes ($d_r^-=0$ or $d_r^+=0$), then necessarily
$e_{F,r}=0$ and $\Hrel_r(\bm H\mid\bm U^\pm)=0$. The strict convexity of $\eta_r$ then implies
$\bm H=\bm U^-$ (resp.~$\bm H=\bm U^+$), whence the HLL identities \eqref{eq:1d-hll-identities}
and \eqref{eq:1d-hll-entropy-identities} reduce to the physical state and entropy fluxes on that side.
Consequently, the convex-decomposition interface conditions strictly imply the alternative
condition \eqref{eq:1d-weak-side-condition}, and the wave speeds constructed in
Proposition~\ref{prop:1d-weak-convex-attainability} also satisfy the interface hypotheses of
Theorem~\ref{thm:1d-weak-relative}.
Conversely, the dissipation condition $d_r^\pm\ge0$ alone does not yield $e_{F,r}\ge0$ in
\eqref{eq:1d-weak-interface-bridge}, because the positive relative-entropy terms
$\Hrel_r(\bm H\mid\bm U^\pm)$ can compensate for a negative value of $e_{F,r}$.
In this sense, the interface requirement of the convex-decomposition criterion is strictly stronger.

At the cell level, under the CFL condition \eqref{eq:1d-pp-cfl}, we define
$e_{j,r}^{\Xi}:=\mathscr E_{j,r}^{\Xi}-\eta_r(\bm\Xi_j)$
and set $\bm W:=\overline{\bm U}_j^{n+1}\in\G_{B_{\rm const}}$.
For any convex combination $\bm W=\sum_k\theta_k\bm Z_k$ with $\sum_k\theta_k=1$ and $\theta_k\ge0$,
the definition of relative entropy gives the general identity
\[
\sum_k\theta_k\eta_r(\bm Z_k)-\eta_r(\bm W)
=\sum_k\theta_k\Hrel_r(\bm Z_k\mid\bm W).
\]
Applying this relation to the convex decomposition \eqref{eq:1d-pp-decomposition} and
subtracting $\eta_r(\bm W)$ from \eqref{eq:1d-weak-bound-decomposition}, we obtain the unified
cell-defect representation
\begin{equation}\label{eq:1d-weak-cell-bridge}
	\begin{aligned}
		\mathcal B_{j,r}^n-\eta_r(\bm W)
		={}&\sum_{\alpha=2}^{L-1}\widehat\omega_\alpha
		\Hrel_r(\widehat{\bm U}_j^{[\alpha]}\mid\bm W)
		+\lambda D_j
		\bigl[\Hrel_r(\bm\Xi_j\mid\bm W)+e_{j,r}^{\Xi}\bigr]\\
		&-\lambda a_{j+1/2}^-
		\bigl[\Hrel_r(\bm H_{j+1/2}\mid\bm W)+e_{j+1/2,r}\bigr]\\
		&+\lambda a_{j-1/2}^+
		\bigl[\Hrel_r(\bm H_{j-1/2}\mid\bm W)+e_{j-1/2,r}\bigr],
	\end{aligned}
\end{equation}
where $e_{j\pm1/2,r}$ denotes the intermediate gap $e_{F,r}$ at the corresponding cell interface $x_{j\pm1/2}$.

Identity \eqref{eq:1d-weak-cell-bridge} reveals the structural distinction between the two approaches.
The convex-decomposition criterion directly enforces $e_{j,r}^{\Xi}\ge0$ and $e_{j\pm1/2,r}\ge0$,
rendering every individual term on the right-hand side of \eqref{eq:1d-weak-cell-bridge} nonnegative.
In contrast, the relative-entropy-rate criterion starts from the exact defect identity
\eqref{eq:1d-weak-defect-identity} and controls the boundary terms $P_{j,\nu,r}(\bm W)$ by bounding them
against the boundary relative entropies $\Hrel_r(\bm U_{j,\nu}\mid\bm W)$ along the update trajectory $\mathscr W_j$.
In summary, the convex-decomposition interface condition implies the relative-entropy interface condition,
but the two proofs rest on different cellwise mechanisms, and the implication does not establish a general ordering between their respective
time-step restrictions $\lambda_*^{\rm cvx}$ and $\lambda_*^{\rm rel}$.

\section{Weak multi-entropy stability in multiple dimensions}
\label{sec:md}

In multiple dimensions, locally divergence-free (LDF) magnetic fields admit discontinuous normal traces across element interfaces. The Godunov--Powell discretization of Wu and Shu \cite{WuShu2019} treats these jumps through an interface source, while the volume source vanishes identically by the cellwise LDF property. In this section, we extend the two weak multi-entropy stability criteria of Section~\ref{sec:1d} to cell-average updates on general polygonal and polyhedral meshes, incorporating the interface source into both the positivity decomposition and the entropy defect estimates.

\subsection{Multidimensional discretization and positivity}
\label{subsec:md-scheme-pp}

\subsubsection{Quadrature and the LDF condition}
Let $\Th$ be a finite conforming mesh of bounded polytopal cells $K$ with nonempty interiors and planar faces in $d\in\{2,3\}$ spatial dimensions. On each cell $K\in\Th$, let $\bm U_K^n\in[\mathbb P^k(K)]^8$ denote a DG polynomial solution or a high-order finite-volume reconstruction with the given cell average, where $\mathbb P^k(K)$ is the space of polynomials of total degree at most $k$. The magnetic field is required to belong to the locally divergence-free polynomial space
\begin{equation}\label{eq:md-ldf}
	\mathbb V_{\rm div}^k(K)
	:=\left\{\bm b\in[\mathbb P^k(K)]^d:
	\sum_{i=1}^d\partial_{x_i}b_i=0
	\text{ in }K\right\},\qquad
	(B_{1,K}^n,\ldots,B_{d,K}^n)^\top
	\in\mathbb V_{\rm div}^k(K).
\end{equation}

For each face $e\subset\partial K$, let $\bm n_{K,e}\in\R^d$ denote its outward unit normal. When $d=2$, $\bm n_{K,e}$ is identified with $(n_1,n_2,0)^\top$ when taking inner or cross products with three-component vectors. For any unit normal $\bm n$, we define the directional quantities
\[
\bm F_{\bm n}(\bm U):=\sum_{i=1}^d n_i\bm F_i(\bm U),\qquad
B_{\bm n}(\bm U):=\bm B\cdot\bm n,\qquad
q_{r,\bm n}(\bm U):=\eta_r(\bm U)\bm v\cdot\bm n.
\]
Let $|K|$ denote the $d$-dimensional measure of $K$ and $|e|$ the $(d-1)$-dimensional measure of a face $e$. On each face, we choose a quadrature rule $\{(\bm x_{e,q},\varpi_{e,q})\}_{q=1}^{Q_e}$ exact for face averages of polynomials in $\mathbb P^k(e)$, with positive weights $\varpi_{e,q}>0$ normalized such that $\sum_q\varpi_{e,q}=1$. On shared internal faces, adjacent cells use identical quadrature nodes and weights. We define the face-quadrature index set
\[
\mathcal I_K^\partial
:=\{(e,q):e\subset\partial K,\ 1\le q\le Q_e\}.
\]
For each index $\nu=(e,q)\in\mathcal I_K^\partial$, we set
\begin{equation}\label{eq:md-face-data}
	\begin{aligned}
		c_{K,\nu}:=|e|\varpi_{e,q}, \quad 
		\bm n_{K,\nu}:=\bm n_{K,e},\quad 
		\bm U_{K,\nu}:=\bm U_K^n(\bm x_{e,q}), \quad 
		\bm U_{K,\nu}^{\ext}:=\bm U_{K_e}^n(\bm x_{e,q}),
	\end{aligned}
\end{equation}
where $K_e$ is the neighboring cell sharing face $e$. At a domain boundary, $\bm U_{K,\nu}^{\ext}$ represents a prescribed exterior boundary trace.

Following \cite{ZhangShu2010,WuShu2019}, we assume that each cell admits a volume quadrature with strictly positive weights that is exact for cell averages of polynomials in $\mathbb P^k(K)$ and whose node set includes all face quadrature points. The cell average then admits the representation
\begin{equation}\label{eq:md-cell-average}
	\begin{aligned}
		\overline{\bm U}_K^n
		&:=\frac1{|K|}\int_K\bm U_K^n\,\dd\bm x
		=\sum_{\alpha\in\mathcal I_K^\circ}\omega_{K,\alpha}^\circ
		\widehat{\bm U}_K^{[\alpha]}
		+\sum_{\nu\in\mathcal I_K^\partial}\omega_{K,\nu}\bm U_{K,\nu},\\
		&\sum_\alpha\omega_{K,\alpha}^\circ
		+\sum_\nu\omega_{K,\nu}=1,
		\qquad \omega_{K,\alpha}^\circ,\omega_{K,\nu}>0.
	\end{aligned}
\end{equation}
Here $\widehat{\bm U}_K^{[\alpha]}:=\bm U_K^n(\bm x_K^{[\alpha]})$ denotes an interior quadrature state, with the interior index set $\mathcal I_K^\circ$ permitted to be empty. The face integration weights $c_{K,\nu}$ and the cell-average weights $\omega_{K,\nu}$ need not be proportional. We retain these quadrature assumptions throughout this section; see \cite[Section~4.2.2]{WuShu2019} for explicit constructions on simplices, tensor-product cells, and general polygonal subdivisions.

All quadrature states and exterior boundary traces are assumed to belong to the admissible set $\G$. Denote the density, velocity, and magnetic field of $\bm U_{K,\nu}$ by $\rho_{K,\nu}$, $\bm v_{K,\nu}$, and $\bm B_{K,\nu}$, respectively. Unless otherwise specified, sums over $\nu$ and $\alpha$ run over $\mathcal I_K^\partial$ and $\mathcal I_K^\circ$.

The divergence theorem and the exactness of face quadratures yield the discrete Gauss identities
\begin{equation}\label{eq:md-discrete-gauss}
	\sum_\nu c_{K,\nu}\bm n_{K,\nu}=\bm0,\qquad
	\sum_\nu c_{K,\nu}B_{\bm n_{K,\nu}}(\bm U_{K,\nu})
	=\int_K\nabla\cdot\bm B_K^n\,\dd\bm x=0.
\end{equation}
The second identity in \eqref{eq:md-discrete-gauss} follows from the LDF property \eqref{eq:md-ldf} and involves only the interior traces of $K$; global continuity of the normal magnetic field across element faces is neither assumed nor required. Furthermore, no exactness assumption is made for quadratures involving nonlinear physical fluxes.

\subsubsection{The cell-average update}
At a given face quadrature point, let $\bm U^-$ and $\bm U^+$ denote the traces on the two sides, with the unit normal $\bm n$ pointing from the minus side to the plus side. We denote the normal magnetic jump by $\delta B_{\bm n}:=B_{\bm n}(\bm U^+)-B_{\bm n}(\bm U^-)$. Given wave speeds $a^-<0<a^+$, let $C_F:=a^+-a^-$, where $F$ identifies the interface quadrature point. The standard HLL flux is
\begin{equation}\label{eq:md-hll-flux}
	\widehat{\bm F}_{\bm n}
	:=\frac{a^+\bm F_{\bm n}(\bm U^-)
		-a^-\bm F_{\bm n}(\bm U^+)
		+a^+a^-(\bm U^+-\bm U^-)}{C_F}.
\end{equation}
For cell $K$ and face index $\nu$, we evaluate the interface quantities by setting $(\bm U^-,\bm U^+,\bm n) =(\bm U_{K,\nu},\bm U_{K,\nu}^{\ext},\bm n_{K,\nu})$. We denote the resulting HLL flux and wave speeds by $\widehat{\bm F}_{K,\nu}$ and $a_{K,\nu}^\pm$, and set $\delta B_{K,\nu}:= B_{\bm n_{K,\nu}}(\bm U_{K,\nu}^{\ext}) - B_{\bm n_{K,\nu}}(\bm U_{K,\nu})$.
Using the same wave speeds for the discretization of the Godunov--Powell source term, we define the modified outward numerical flux operator
\begin{equation}\label{eq:md-powell-operators}
	\bm G_{K,\nu}:=\widehat{\bm F}_{K,\nu}
	-\frac{a_{K,\nu}^-}{a_{K,\nu}^+-a_{K,\nu}^-}
	\delta B_{K,\nu}\bm S(\bm U_{K,\nu}).
\end{equation}
For a forward Euler step $\tau>0$, let $\lambda_K:=\tau/|K|$. The cell-average update takes the form
\begin{equation}\label{eq:md-mean-update}
	\overline{\bm U}_K^{n+1}
	=\overline{\bm U}_K^n
	-\lambda_K\sum_\nu c_{K,\nu}\bm G_{K,\nu}.
\end{equation}
Equation \eqref{eq:md-mean-update} is the common cell-average update for the LDF DG scheme and high-order finite-volume methods developed in \cite{WuShu2019}. Adjacent cells share a common pair of wave speeds at each interface quadrature point. Reversing the interface orientation exchanges the interior and exterior traces, mapping $\bm n \mapsto -\bm n$ and $(a^-,a^+) \mapsto (-a^+,-a^-)$. Under this orientation reversal, the HLL state flux and the numerical entropy fluxes defined below change sign, while the normal magnetic jump $\delta B_{\bm n}$ remains invariant.

\subsubsection{A sufficient positivity condition}
Building on the estimates of Wu and Shu \cite{WuShu2019}, we reorganize the positivity analysis by absorbing the Godunov--Powell interface source into corrected HLL intermediate states and strengthening the wave-speed bounds to guarantee their physical admissibility, thereby obtaining a multidimensional convex decomposition suited to the entropy analysis that follows.

Recall the standard HLL intermediate state
\begin{equation}\label{eq:md-hll-state}
	\bm H:=\frac{a^+\bm U^+-a^-\bm U^-
		-\bm F_{\bm n}(\bm U^+)+\bm F_{\bm n}(\bm U^-)}{C_F}
\end{equation}
and define the two source-corrected intermediate states
\begin{equation}\label{eq:md-corrected-hll-states}
	\bm T^\pm:=\bm H-\frac{\delta B_{\bm n}}{C_F}\bm S(\bm U^\pm).
\end{equation}
With respect to the outward orientation of cell $K$, we write $\bm T_{K,\nu}:=\bm T^-$. A direct calculation using \eqref{eq:md-powell-operators} yields the modified consistency identity
\begin{equation}\label{eq:md-hll-powell-identities}
	\bm G_{K,\nu}
	=\bm F_{\bm n_{K,\nu}}(\bm U_{K,\nu})
	+a_{K,\nu}^-(\bm T_{K,\nu}-\bm U_{K,\nu}).
\end{equation}
Under orientation reversal, $\bm T^-$ and $\bm T^+$ are interchanged, ensuring that \eqref{eq:md-hll-powell-identities} holds symmetrically for both cells sharing the interface.

\begin{remark}[Interface source correction and the one-dimensional reduction]
	\label{rem:md-interface-correction}
	Because the Godunov--Powell source vector $\bm S$ is evaluated separately at the left and right interface traces, the interface source naturally induces two distinct corrected states $\bm T^\pm$. In the absence of a normal magnetic jump ($\delta B_{\bm n}=0$), both states coincide with $\bm H$, recovering the single HLL intermediate state employed in the one-dimensional decomposition.
\end{remark}

Let $\alpha_l$ and $\alpha_r$ denote the directional splitting bounds from \cite[Eqs.~(12)--(13)]{WuShu2019}, and define the upper bound on the fast magnetoacoustic speed
\[
s_F:=\max_{\varsigma\in\{-,+\}}
\sqrt{\frac{\gamma p(\bm U^\varsigma)+|\bm B(\bm U^\varsigma)|^2}
	{\rho(\bm U^\varsigma)}}>0.
\]
Setting $\rho^\pm:=\rho(\bm U^\pm)$, we define the baseline interface wave speeds
\begin{equation}\label{eq:md-initial-speeds}
	\begin{aligned}
		a^{-,0}&:=\min\{0,\alpha_l(\bm U^-,\bm U^+;\bm n)\}
		-\frac{|\delta B_{\bm n}|}{\sqrt{\rho^-}}-s_F,\\
		a^{+,0}&:=\max\{0,\alpha_r(\bm U^+,\bm U^-;\bm n)\}
		+\frac{|\delta B_{\bm n}|}{\sqrt{\rho^+}}+s_F.
	\end{aligned}
\end{equation}
Quantities evaluated with these baseline speeds are designated with a superscript $0$; in particular, $C_F^0:=a^{+,0}-a^{-,0}>0$ denotes the baseline fan width.

\begin{lemma}[Admissibility of the corrected HLL states]
	\label{lem:md-corrected-hll-admissibility}
	Let $\bm U^\pm\in\G$ and let $a^{-,0},a^{+,0}$ be defined by \eqref{eq:md-initial-speeds}. If the interface wave speeds satisfy
	\begin{equation}\label{eq:md-pp-speed-condition}
		a^-\le a^{-,0},\qquad a^+\ge a^{+,0},
	\end{equation}
	then both corrected states in \eqref{eq:md-corrected-hll-states} belong to the admissible set $\G$.
\end{lemma}
\begin{proof}
	For arbitrary constant vectors $\bm v_*,\bm B_*\in\R^3$, consider the quadratic auxiliary functional
	\[
	\Phi_*(\bm U):=E-\bm m\cdot\bm v_*+\tfrac12\rho|\bm v_*|^2
	-\bm B\cdot\bm B_*+\tfrac12|\bm B_*|^2.
	\]
	Because the density component of the source vector $\bm S$ vanishes identically, the HLL estimate in \cite[Theorem~2]{WuShu2019} guarantees $\rho(\bm T^{\pm,0})=\rho(\bm H^0)>0$. Combining that estimate with the speed increments in \eqref{eq:md-initial-speeds} and the nonconservative source bound from \cite[Lemma~7]{WuShu2019} yields
	\begin{equation}\label{eq:md-corrected-state-positivity}
		\begin{aligned}
			C_F^0\Phi_*(\bm T^{-,0})
			&\ge s_F\Phi_*(\bm U^-)
			+\left(s_F+\frac{|\delta B_{\bm n}|}{\sqrt{\rho^+}}\right)
			\Phi_*(\bm U^+)>0,\\
			C_F^0\Phi_*(\bm T^{+,0})
			&\ge\left(s_F+\frac{|\delta B_{\bm n}|}{\sqrt{\rho^-}}\right)
			\Phi_*(\bm U^-)+s_F\Phi_*(\bm U^+)>0.
		\end{aligned}
	\end{equation}
	By the convex dual characterization of $\G$ in \cite[Lemma~1]{WuShu2019}, this implies $\bm T^{\pm,0}\in\G$. For general wave speeds satisfying \eqref{eq:md-pp-speed-condition}, defining $h_-:=a^{-,0}-a^-\ge0$ and $h_+:=a^+-a^{+,0}\ge0$, we express $\bm T^\pm$ as the convex combination
	\[
	\bm T^\pm
	=\frac{C_F^0\bm T^{\pm,0}+h_-\bm U^-+h_+\bm U^+}
	{C_F^0+h_-+h_+}\in\G,
	\]
	which completes the proof by the convexity of $\G$.
\end{proof}

For the cell-interior estimate, we introduce the Roe-type velocity average
\[
\widetilde{\bm v}_{K,\nu\mu}
:=\frac{\sqrt{\rho_{K,\nu}}\bm v_{K,\nu}
	+\sqrt{\rho_{K,\mu}}\bm v_{K,\mu}}
{\sqrt{\rho_{K,\nu}}+\sqrt{\rho_{K,\mu}}}.
\]
Using the total boundary measure $|\partial K|:=\sum_\nu c_{K,\nu}$, the multistate splitting coefficients of \cite[Theorem~1]{WuShu2019} take the form
\begin{equation}\label{eq:md-multistate-coefficient}
	\begin{aligned}
		\alpha_{K,\nu}^{\rm PP}:={}&
		\max\left\{\bm v_{K,\nu}\cdot\bm n_{K,\nu},
		\frac1{|\partial K|}\sum_\mu c_{K,\mu}
		(\bm n_{K,\nu}-\bm n_{K,\mu})\cdot
		\widetilde{\bm v}_{K,\nu\mu}\right\}\\
		&+\mathcal C(\bm U_{K,\nu};\bm n_{K,\nu})
		+\frac2{|\partial K|}\sum_\mu c_{K,\mu}
		\frac{|\bm B_{K,\nu}-\bm B_{K,\mu}|}
		{\sqrt{\rho_{K,\nu}}+\sqrt{\rho_{K,\mu}}},
	\end{aligned}
\end{equation}
where $\mathcal C(\bm U;\bm n)$ is the directional acoustic splitting speed defined in \cite[Section~2.2.1]{WuShu2019}.

We impose the multidimensional positivity CFL restriction
\begin{equation}\label{eq:md-pp-cfl}
	\lambda_Kc_{K,\nu}
	(\alpha_{K,\nu}^{\rm PP}-a_{K,\nu}^-)
	<\omega_{K,\nu},\qquad K\in\Th,\quad\nu\in\mathcal I_K^\partial.
\end{equation}
Analogously to the one-dimensional case, we define the auxiliary boundary coefficients and the cell-interior splitting state by
\begin{equation}\label{eq:md-interior-state}
	\begin{aligned}
		\beta_{K,\nu}:=\frac{\omega_{K,\nu}}{\lambda_Kc_{K,\nu}}
		+a_{K,\nu}^-, \quad 
		D_K:=\sum_\nu c_{K,\nu}\beta_{K,\nu},\quad 
		\bm\Xi_K
		:=\frac{\sum_\nu c_{K,\nu}
			[\beta_{K,\nu}\bm U_{K,\nu}
			-\bm F_{\bm n_{K,\nu}}(\bm U_{K,\nu})]}{D_K}.
	\end{aligned}
\end{equation}
Condition \eqref{eq:md-pp-cfl} ensures that $\varepsilon_{K,\nu}:=\beta_{K,\nu}-\alpha_{K,\nu}^{\rm PP}>0$. Invoking the discrete divergence identity \eqref{eq:md-discrete-gauss} and \cite[Theorem~1 and Remark~1]{WuShu2019}, the baseline sum $\widehat D_K:=\sum_\nu c_{K,\nu}\alpha_{K,\nu}^{\rm PP}$ is strictly positive, and the baseline state $\widehat{\bm\Xi}_K$ obtained from \eqref{eq:md-interior-state} by replacing $\beta_{K,\nu}$ with $\alpha_{K,\nu}^{\rm PP}$ satisfies $\widehat{\bm\Xi}_K\in\overline\G$. Consequently,
\[
D_K=\widehat D_K+\sum_\nu c_{K,\nu}\varepsilon_{K,\nu}>0,
\qquad
\bm\Xi_K=
\frac{\widehat D_K\widehat{\bm\Xi}_K
	+\sum_\nu c_{K,\nu}\varepsilon_{K,\nu}\bm U_{K,\nu}}
{D_K}\in\G,
\]
because the admissible boundary states carry strictly positive weights.

Substituting the modified flux identity \eqref{eq:md-hll-powell-identities} into \eqref{eq:md-mean-update}, we obtain the multidimensional convex decomposition
\begin{equation}\label{eq:md-pp-decomposition}
	\overline{\bm U}_K^{n+1}
	=\sum_\alpha\omega_{K,\alpha}^\circ
	\widehat{\bm U}_K^{[\alpha]}
	+\lambda_KD_K\bm\Xi_K
	-\lambda_K\sum_\nu c_{K,\nu}a_{K,\nu}^-\bm T_{K,\nu}.
\end{equation}
Because $\lambda_KD_K-\lambda_K\sum_\nu c_{K,\nu}a_{K,\nu}^- =\sum_\nu\omega_{K,\nu}$ and the quadrature weights sum to unity in \eqref{eq:md-cell-average}, all expansion coefficients in \eqref{eq:md-pp-decomposition} are strictly positive and sum to one.

\begin{proposition}[Positivity of the multidimensional cell average]
	\label{prop:md-pp}
	Under the preceding quadrature, admissibility, and LDF assumptions, suppose that \eqref{eq:md-pp-speed-condition} holds at every interface point. If the CFL condition \eqref{eq:md-pp-cfl} holds, then $\overline{\bm U}_K^{n+1}\in\G$ for every cell $K\in\Th$.
\end{proposition}
\begin{proof}
	Lemma~\ref{lem:md-corrected-hll-admissibility} and the interior estimate ensure that every state in decomposition \eqref{eq:md-pp-decomposition} belongs to $\G$. The assertion then follows from the convexity of $\G$.
\end{proof}

For both multidimensional entropy criteria below, interface wave speeds are selected first and held fixed. With these speeds and numerical fluxes fixed, we choose a CFL safety factor $0<\sigma<1$ and define the cell reference time-step bound
\begin{equation}\label{eq:md-reference-step}
	\tau_K^0:=\sigma|K|\min_{\nu\in\mathcal I_K^\partial}
	\frac{\omega_{K,\nu}}
	{c_{K,\nu}(\alpha_{K,\nu}^{\rm PP}-a_{K,\nu}^-)}>0.
\end{equation}
Indeed, \eqref{eq:md-multistate-coefficient} implies $\alpha_{K,\nu}^{\rm PP}\ge\bm v_{K,\nu}\cdot\bm n_{K,\nu} +\mathcal C(\bm U_{K,\nu};\bm n_{K,\nu})$, whereas \eqref{eq:md-initial-speeds}--\eqref{eq:md-pp-speed-condition} and \cite[Eq.~(13)]{WuShu2019} give $a_{K,\nu}^-\le\bm v_{K,\nu}\cdot\bm n_{K,\nu} -\mathcal C(\bm U_{K,\nu};\bm n_{K,\nu})$. Subtracting these inequalities yields
\[
\alpha_{K,\nu}^{\rm PP}-a_{K,\nu}^-
\ge 2\mathcal C(\bm U_{K,\nu};\bm n_{K,\nu})>0.
\]
Thus every denominator in \eqref{eq:md-reference-step} is strictly positive and bounded away from zero. Any uniform time step satisfying $0<\tau\le\min_K\tau_K^0$ therefore satisfies the positivity CFL condition \eqref{eq:md-pp-cfl} strictly.

\subsection{Entropy bounds and stability definitions}
\label{subsec:md-weak}
Combining the cell quadrature with the outward numerical entropy fluxes yields the multidimensional counterpart of Definition~\ref{def:weak-strong-entropy}. We present the formulation for planar polygonal meshes ($d=2$); the definitions and bounds carry over directly to polyhedral meshes in three dimensions ($d=3$).

At each oriented interface point, employing the same wave speeds as in \eqref{eq:md-hll-flux} and~\eqref{eq:md-powell-operators}, we define the numerical entropy fluxes by
\begin{equation}\label{eq:md-hll-entropy-flux}
	\widehat q_r
	:=\frac{a^+q_{r,\bm n}(\bm U^-)
		-a^-q_{r,\bm n}(\bm U^+)
		+a^+a^-[\eta_r(\bm U^+)-\eta_r(\bm U^-)]}{C_F}.
\end{equation}
Each flux is consistent with the directional entropy flux $q_{r,\bm n}$ and is conservative across cell interfaces, changing sign under reversal of the normal. We denote by $\widehat q_{K,\nu,r}$ its outward value at face node $\nu$ of cell $K$.

In terms of the cell quadrature \eqref{eq:md-cell-average}, the discrete cell entropy at time level $t^n$ is given by
\begin{equation}\label{eq:md-quadrature-entropy}
	\mathcal S_{K,r}(\bm U_K^n)
	:=\sum_\alpha\omega_{K,\alpha}^\circ
	\eta_r(\widehat{\bm U}_K^{[\alpha]})
	+\sum_\nu\omega_{K,\nu}\eta_r(\bm U_{K,\nu}),
\end{equation}
and the corresponding forward Euler entropy bound is
\begin{equation}\label{eq:md-weak-bound}
	\mathcal B_{K,r}^n
	:=\mathcal S_{K,r}(\bm U_K^n)
	-\lambda_K\sum_\nu c_{K,\nu}\widehat q_{K,\nu,r}.
\end{equation}
Because the quadrature nodes and weights are fixed and the numerical entropy fluxes depend solely on $\bm U^n$, the bound $\mathcal B_{K,r}^n$ is fixed by the cell-average update.

\begin{definition}[Local weak and strong multi-entropy stability in multiple dimensions]
	\label{def:md-weak-strong-entropy}
	Retain the preceding quadrature, admissibility, and LDF assumptions on $\Th$ for $d\in\{2,3\}$. The cell-average update is \emph{locally weakly multi-entropy stable} if
	\begin{equation}\label{eq:md-weak-entropy}
		\overline{\bm U}_K^{n+1}\in\G,
		\qquad
		\eta_r(\overline{\bm U}_K^{n+1})\le\mathcal B_{K,r}^n,
		\quad K\in\Th,\quad 1\le r\le M,
	\end{equation}
	with the local entropy bounds \eqref{eq:md-weak-bound}.
	A polynomial update $\bm U_K^{n+1}\in[\mathbb P^k(K)]^8$ whose magnetic components satisfy $(B_{1,K}^{n+1},\ldots,B_{d,K}^{n+1})^\top\in\mathbb V_{\rm div}^k(K)$, having these prescribed cell averages, is \emph{locally strongly multi-entropy stable} if its states at every node in \eqref{eq:md-cell-average} belong to $\G$ and
	\begin{equation}\label{eq:md-strong-entropy}
		\mathcal S_{K,r}(\bm U_K^{n+1})\le\mathcal B_{K,r}^n,
		\qquad K\in\Th,\quad 1\le r\le M.
	\end{equation}
	The inequalities hold simultaneously for the prescribed entropy family, with one updated solution, one common time step, and one shared speed pair at each interface point. When $M=1$, we refer to these properties as weak and strong entropy stability. A scheme possesses the corresponding property if the inequalities hold for every admissible input satisfying the standing assumptions and for every time step permitted by its stability restriction.
\end{definition}

\begin{remark}[Global entropy balance on a polygonal mesh]
	\label{rem:md-global-entropy}
	Let $\Omega$ denote the computational domain partitioned by $\Th$. For a domain with physical boundary, we define the total outward entropy flux across $\partial\Omega$ as
	\begin{equation}\label{eq:md-boundary-entropy-flux}
		\mathcal F_{\partial\Omega,r}^n
		:=\sum_{K\in\Th}
		\sum_{\substack{\nu=(e,q)\in\mathcal I_K^\partial\\
				e\subset\partial\Omega}}
		c_{K,\nu}\widehat q_{K,\nu,r}.
	\end{equation}
	Under the respective admissibility and solenoidal requirements, the \emph{global weak} and \emph{global strong} multi-entropy inequalities read
	\begin{equation}\label{eq:md-global-weak-boundary}
		\sum_{K\in\Th}|K|\eta_r(\overline{\bm U}_K^{n+1})
		\le\sum_{K\in\Th}|K|\mathcal S_{K,r}(\bm U_K^n)
		-\tau\mathcal F_{\partial\Omega,r}^n,
	\end{equation}
	\begin{equation}\label{eq:md-global-strong-boundary}
		\sum_{K\in\Th}|K|\mathcal S_{K,r}(\bm U_K^{n+1})
		\le\sum_{K\in\Th}|K|\mathcal S_{K,r}(\bm U_K^n)
		-\tau\mathcal F_{\partial\Omega,r}^n,
	\end{equation}
	respectively, for each $1\le r\le M$. Multiplying the local inequalities \eqref{eq:md-weak-entropy} and \eqref{eq:md-strong-entropy} by $|K|$ and summing over all $K\in\Th$, the interior numerical entropy fluxes cancel pairwise across interelement faces, while the identity $|K|\lambda_K=\tau$ yields the boundary flux term $\tau\mathcal F_{\partial\Omega,r}^n$. Although the modified fluxes $\bm G_{K,\nu}$ are nonconservative across cell interfaces due to normal magnetic jumps, the numerical entropy fluxes $\widehat q_{K,\nu,r}$ remain conservative and anti-symmetric, ensuring exact interior cancellation upon summation. For periodic boundary conditions, the boundary flux vanishes identically, whereby \eqref{eq:md-global-strong-boundary} guarantees that the total discrete entropy is non-increasing across successive strongly stable forward Euler steps. Furthermore, the entropy hierarchy and quadrature interpretation established in Remark~\ref{rem:weak-strong-distinction} hold verbatim in two and three dimensions.
\end{remark}

No magnetic divergence source term appears in the entropy budget \eqref{eq:md-weak-bound}. As shown in \eqref{eq:md-magnetic-cancellation} below, the Godunov--Powell interface correction and the LDF discrete Gauss identity \eqref{eq:md-discrete-gauss} jointly cancel the magnetic divergence contribution in the cell entropy estimate. Both stability criteria below establish weak multi-entropy stability in the sense of Definition~\ref{def:md-weak-strong-entropy}: the first directly bounds the entropy of each constituent state in the convex decomposition \eqref{eq:md-pp-decomposition}, whereas the second bounds the net cell entropy defect via relative entropies evaluated at the nodal states at time level $t^n$.

\subsection{A convex-decomposition criterion}
\label{subsec:md-weak-convex}

\subsubsection{Entropy bounds for the intermediate states}
To derive the interface entropy bounds, we define the intermediate entropy value
\begin{equation}\label{eq:md-hll-entropy-value}
	\mathscr E_r^{\rm HLL}
	:=\frac{a^+\eta_r(\bm U^+)-a^-\eta_r(\bm U^-)
		-q_{r,\bm n}(\bm U^+)+q_{r,\bm n}(\bm U^-)}{C_F}.
\end{equation}
This scalar quantity is invariant under orientation reversal and satisfies the identity
\begin{equation}\label{eq:md-hll-entropy-identities}
	\begin{aligned}
		\widehat q_r
		=q_{r,\bm n}(\bm U^-)
		+a^-[\mathscr E_r^{\rm HLL}-\eta_r(\bm U^-)]
		=q_{r,\bm n}(\bm U^+)
		+a^+[\mathscr E_r^{\rm HLL}-\eta_r(\bm U^+)].
	\end{aligned}
\end{equation}
Let $\mathscr E_{K,\nu,r}^{\rm HLL}$ denote its value at face node $\nu$ of cell $K$. Corresponding to the interior state $\bm\Xi_K$ in \eqref{eq:md-interior-state}, we define the scalar entropy value
\begin{equation}\label{eq:md-interior-entropy}
	\mathscr E_{K,r}^{\Xi}
	:=\frac{\sum_\nu c_{K,\nu}
		[\beta_{K,\nu}\eta_r(\bm U_{K,\nu})
		-q_{r,\bm n_{K,\nu}}(\bm U_{K,\nu})]}{D_K}.
\end{equation}
Substituting the entropy flux identities \eqref{eq:md-hll-entropy-identities} into \eqref{eq:md-weak-bound} yields the matching decomposition of the local entropy bound:
\begin{equation}\label{eq:md-bound-decomposition}
	\mathcal B_{K,r}^n
	=\sum_\alpha\omega_{K,\alpha}^\circ
	\eta_r(\widehat{\bm U}_K^{[\alpha]})
	+\lambda_KD_K\mathscr E_{K,r}^{\Xi}
	-\lambda_K\sum_\nu c_{K,\nu}a_{K,\nu}^-
	\mathscr E_{K,\nu,r}^{\rm HLL}.
\end{equation}

\begin{theorem}[Weak multi-entropy stability by convex decomposition]
	\label{thm:md-weak-convex}
	Under the hypotheses of Proposition~\ref{prop:md-pp}, suppose that the intermediate states satisfy
	\begin{equation}\label{eq:md-convex-entropy-conditions}
		\begin{aligned}
			\eta_r(\bm T^\pm)&\le\mathscr E_r^{\rm HLL}
			&&\text{at every interface point},\\
			\eta_r(\bm\Xi_K)&\le\mathscr E_{K,r}^{\Xi}
			&&\text{in every cell } K\in\Th,
		\end{aligned}
	\end{equation}
	for all $1\le r\le M$. Then the forward Euler cell-average update \eqref{eq:md-mean-update} is weakly multi-entropy stable in the sense of Definition~\ref{def:md-weak-strong-entropy}.
\end{theorem}
\begin{proof}
	By Proposition~\ref{prop:md-pp}, all constituent states in decomposition \eqref{eq:md-pp-decomposition} belong to $\G$, which ensures $\overline{\bm U}_K^{n+1}\in\G$. Applying the convexity of $\eta_r$ to \eqref{eq:md-pp-decomposition}, invoking the statewise bounds \eqref{eq:md-convex-entropy-conditions}, and recognizing the matching decomposition \eqref{eq:md-bound-decomposition} of $\mathcal B_{K,r}^n$, we obtain $\eta_r(\overline{\bm U}_K^{n+1})\le\mathcal B_{K,r}^n$ simultaneously for all $1\le r\le M$.
\end{proof}

\subsubsection{Finite wave speeds and a positive common step}
The interface entropy bounds are satisfied by suitably enlarging the wave speeds, whereas the interior bounds are enforced by restricting the time step. Both constructions rely on the convex mixing inequality \eqref{eq:1d-weak-completion-bound}, which holds on $\G$ for every entropy $\eta_r$ in the family.

At each interface point, let the initial wave speeds satisfy \eqref{eq:md-initial-speeds}, so that Lemma~\ref{lem:md-corrected-hll-admissibility} guarantees $\bm T^{\pm,0}\in\G$. As in the one-dimensional setting, select a partition parameter $0<\theta_F<1$ independent of $r$, satisfying $\theta_F^{\rm rev}=1-\theta_F$ under orientation reversal (such as $\theta_F=-a^{-,0}/C_F^0$). We define the Jensen gap quantities and initial defects by
\begin{equation}\label{eq:md-interface-gaps}
	\begin{aligned}
		\bm M_F&:=\theta_F\bm U^-+(1-\theta_F)\bm U^+,\qquad 
		\overline\eta_{F,r}
		:=\theta_F\eta_r(\bm U^-)+(1-\theta_F)\eta_r(\bm U^+),\\
		J_{F,r}&:=\overline\eta_{F,r}-\eta_r(\bm M_F),\qquad \qquad 
		g_{F,\pm,r}^0
		:=C_F^0[\mathscr E_r^{{\rm HLL},0}-\eta_r(\bm T^{\pm,0})].
	\end{aligned}
\end{equation}
Whenever $\bm U^-\neq\bm U^+$, the strict convexity of $\eta_r$ implies $J_{F,r}>0$ for all $1\le r\le M$. We then set
\begin{equation}\label{eq:md-convex-interface-speeds}
	\zeta_F:=\max_{\substack{1\le r\le M\\\varsigma\in\{-,+\}}}
	\frac{[-g_{F,\varsigma,r}^0]_+}{J_{F,r}}, \qquad 
	a^-:=a^{-,0}-\theta_F\zeta_F,
	\qquad a^+:=a^{+,0}+(1-\theta_F)\zeta_F.
\end{equation}
Because the source numerator $\delta B_{\bm n}\bm S(\bm U^\pm)$ is determined solely by the interface traces, the speed dilation mixes both corrected states $\bm T^\pm$ toward the common admissible state $\bm M_F$:
\begin{equation}\label{eq:md-interface-mixing}
	\bm T^\pm
	=\frac{C_F^0\bm T^{\pm,0}+\zeta_F\bm M_F}{C_F^0+\zeta_F}, \qquad 
	\mathscr E_r^{\rm HLL}
	=\frac{C_F^0\mathscr E_r^{{\rm HLL},0}
		+\zeta_F\overline\eta_{F,r}}
	{C_F^0+\zeta_F}.
\end{equation}
Applying the convexity estimate \eqref{eq:1d-weak-completion-bound} gives
\begin{equation*}
	(C_F^0+\zeta_F)[\mathscr E_r^{\rm HLL}-\eta_r(\bm T^\pm)]
	\ge g_{F,\pm,r}^0+\zeta_FJ_{F,r}\ge0,
\end{equation*}
which establishes the interface condition in \eqref{eq:md-convex-entropy-conditions}. If $\bm U^-=\bm U^+$, we set $\zeta_F=0$; then $\bm T^\pm=\bm U^\pm$, and the interface entropy bounds hold with equality.

With the interface wave speeds and numerical fluxes fixed, we evaluate the reference time steps $\tau_K^0$ from \eqref{eq:md-reference-step}. In each cell $K\in\Th$, define
\begin{equation}\label{eq:md-interior-gaps}
	\begin{aligned}
		\Omega_K&:=\sum_{\nu\in\mathcal I_K^\partial}\omega_{K,\nu},
		&\bm M_K&:=\frac1{\Omega_K}\sum_\nu\omega_{K,\nu}\bm U_{K,\nu},\\
		\overline\eta_{K,r}
		&:=\frac1{\Omega_K}\sum_\nu\omega_{K,\nu}\eta_r(\bm U_{K,\nu}),
		&J_{K,r}&:=\overline\eta_{K,r}-\eta_r(\bm M_K),\quad 
		g_{K,r}^0
		:=D_K^0[\mathscr E_{K,r}^{\Xi,0}-\eta_r(\bm\Xi_K^0)],
	\end{aligned}
\end{equation}
where superscript $0$ denotes evaluation at $\tau=\tau_K^0$. For any $0<\tau\le\tau_K^0$, let $t_K(\tau):=|K|\Omega_K(1/\tau-1/\tau_K^0)\ge0$. Recalling the definition of $\beta_{K,\nu}$ in \eqref{eq:md-interior-state}, we obtain
\begin{equation}\label{eq:md-interior-mixing}
	\begin{aligned}
		D_K(\tau)=D_K^0+t_K(\tau),\quad
		\bm\Xi_K(\tau)
		&=\frac{D_K^0\bm\Xi_K^0+t_K(\tau)\bm M_K}{D_K^0+t_K(\tau)},\\
		\mathscr E_{K,r}^{\Xi}(\tau)
		&=\frac{D_K^0\mathscr E_{K,r}^{\Xi,0}
			+t_K(\tau)\overline\eta_{K,r}}{D_K^0+t_K(\tau)}.
	\end{aligned}
\end{equation}
Decreasing the time step $\tau$ mixes the interior state $\bm\Xi_K(\tau)$ toward the fixed boundary average $\bm M_K$. Applying the convexity inequality \eqref{eq:1d-weak-completion-bound} yields
\begin{equation}\label{eq:md-interior-defect-bound}
	D_K(\tau)[\mathscr E_{K,r}^{\Xi}(\tau)-\eta_r(\bm\Xi_K(\tau))]
	\ge g_{K,r}^0+t_K(\tau)J_{K,r}.
\end{equation}
If the boundary states $\{\bm U_{K,\nu}\}_\nu$ are not all identical, the strict convexity of $\eta_r$ ensures $J_{K,r}>0$ for every $1\le r\le M$. Setting
\begin{equation}\label{eq:md-convex-step}
	\zeta_K:=\max_{1\le r\le M}\frac{[-g_{K,r}^0]_+}{J_{K,r}},
	\qquad
	\tau_K^{\rm cvx}
	:=\left(\frac1{\tau_K^0}+\frac{\zeta_K}{|K|\Omega_K}\right)^{-1},
\end{equation}
the right-hand side of \eqref{eq:md-interior-defect-bound} remains nonnegative for all $0<\tau\le\tau_K^{\rm cvx}$, establishing $\eta_r(\bm\Xi_K)\le\mathscr E_{K,r}^{\Xi}$. If all boundary states coincide, $\bm U_{K,\nu}\equiv\bm U_0$, the geometric closure identity $\sum_\nu c_{K,\nu}\bm n_{K,\nu}=\bm 0$ forces the weighted sums of the physical fluxes and directional entropy fluxes to vanish identically. In this degenerate case, $\bm\Xi_K=\bm U_0$ and $\mathscr E_{K,r}^{\Xi}=\eta_r(\bm U_0)$, so we may set $\zeta_K=0$ and $\tau_K^{\rm cvx}=\tau_K^0$.

\begin{proposition}[Realizability of the convex-decomposition criterion]
	\label{prop:md-convex-attainability}
	Assume that the high-order polynomial data at time $t^n$ satisfy the admissibility and LDF conditions with nodal values in $\G$. Then the speed modification \eqref{eq:md-convex-interface-speeds} yields finite shared interface speeds $a^-<0<a^+$. Under the conventions adopted for identical states, the global time-step bound satisfies
	\begin{equation}
		\tau_*^{\rm cvx}:=\min_{K\in\Th}\tau_K^{\rm cvx}>0,
	\end{equation}
	and every time step $0<\tau\le\tau_*^{\rm cvx}$ fulfills the hypotheses of Theorem~\ref{thm:md-weak-convex} simultaneously across all mesh cells $K\in\Th$ and all entropies $1\le r\le M$.
\end{proposition}
\begin{proof}
	By Lemma~\ref{lem:md-corrected-hll-admissibility} and Proposition~\ref{prop:md-pp}, the initial states $\bm T^{\pm,0}$ and $\bm\Xi_K^0$ are strictly admissible. For distinct interface traces $\bm U^-\neq\bm U^+$ and non-identical boundary quadrature states, the strict convexity of the Harten entropies guarantees strictly positive Jensen gaps $J_{F,r}>0$ and $J_{K,r}>0$. Consequently, the dilation parameter $\zeta_F$ in \eqref{eq:md-convex-interface-speeds} and the cell parameter $\zeta_K$ in \eqref{eq:md-convex-step} are finite, yielding finite wave speeds and strictly positive local time-step bounds $\tau_K^{\rm cvx}>0$. The degenerate cases with identical states satisfy the inequalities trivially with $\zeta_F=0$ and $\zeta_K=0$. Because the mesh $\Th$ contains a finite number of elements and the entropy family $\{(\eta_r, q_{r,\bm n})\}_{r=1}^M$ is finite, the minimum over all cells $\tau_*^{\rm cvx} = \min_{K\in\Th}\tau_K^{\rm cvx}$ is strictly positive, completing the proof.
\end{proof}

\subsection{A criterion based on relative-entropy rates}
\label{subsec:md-weak-relative}

In multiple dimensions, spatial variations in the normal magnetic field introduce a first-order linear term into the uncorrected relative entropy flux. The Godunov--Powell formulation eliminates this linear growth, while the discrete LDF condition ensures that the corresponding magnetic correction cancels identically in the weighted boundary sum. These two complementary properties enable the extension of the one-dimensional relative-entropy rate framework of Subsection~\ref{subsec:1d-weak-relative} to polygonal and polyhedral meshes.

\subsubsection{Magnetic corrections and the entropy defect}
For any state $\bm U\in\G$, recall the entropy variables $\bm V_r(\bm U)=\nabla\eta_r(\bm U)$ and relative entropy $\Hrel_r(\bm U\mid\bm W)$ from \eqref{eq:1d-weak-relative-quantities}. We define the directional entropy potential $\psi_{r,\bm n}$ and the magnetic compatibility coefficient $\phi_r$ by
\begin{equation}\label{eq:md-entropy-potentials}
	\begin{aligned}
		\psi_{r,\bm n}(\bm U)
		:=\bm V_r(\bm U)^\top\bm F_{\bm n}(\bm U)-q_{r,\bm n}(\bm U),\qquad
		\phi_r(\bm U)
		:=\bm V_r(\bm U)^\top\bm S(\bm U)
		=\frac{(\gamma-1)\rho f_r'(s)}{p}(\bm v\cdot\bm B).
	\end{aligned}
\end{equation}
As established in Appendix~\ref{app:harten}, the Godunov--Powell compatibility identity on $\G$ reads
\begin{equation}\label{eq:md-powell-compatibility}
	Dq_{r,\bm n}(\bm U)
	=\bm V_r(\bm U)^\top D\bm F_{\bm n}(\bm U)
	+\phi_r(\bm U)DB_{\bm n}.
\end{equation}
Here $D$ denotes the Jacobian with respect to the conservative variables $\bm U$, so that $Dq_{r,\bm n}$ is a row vector, $\bm V_r$ is a column vector, and $DB_{\bm n}=(0,\bm0_3^\top,\bm n^\top,0)$. We define the Powell-corrected relative entropy flux by
\begin{equation}\label{eq:md-powell-relative-flux}
	\begin{aligned}
		\Qrel_{r,\bm n}(\bm U\mid\bm W)
		:={}&q_{r,\bm n}(\bm U)-q_{r,\bm n}(\bm W)
		-\bm V_r(\bm W)^\top
		[\bm F_{\bm n}(\bm U)-\bm F_{\bm n}(\bm W)]\\
		&-\phi_r(\bm W)[B_{\bm n}(\bm U)-B_{\bm n}(\bm W)].
	\end{aligned}
\end{equation}
Differentiating \eqref{eq:md-powell-relative-flux} with respect to $\bm U$ and invoking compatibility identity \eqref{eq:md-powell-compatibility}, we find that the first Fr\'echet derivative $D_{\bm U}\Qrel_{r,\bm n}(\bm U\mid\bm W)$ vanishes at $\bm U=\bm W$. Taylor's theorem together with the positive definiteness of the Hessian $D^2\eta_r$ (Lemma~\ref{lem:harten-convexity}) then guarantees the existence of constants $c_r,C_r>0$ on any compact convex subset of $\G$ such that
\begin{equation}\label{eq:md-quadratic-bounds}
	\Hrel_r(\bm U\mid\bm W)\ge c_r\norm{\bm U-\bm W}^2,
	\qquad
	|\Qrel_{r,\bm n}(\bm U\mid\bm W)|
	\le C_r\norm{\bm U-\bm W}^2,
\end{equation}
uniformly for all unit normal vectors $\bm n$.

\begin{remark}[Magnetic correction and bounded relative-entropy ratios]
	\label{rem:md-relative-flux}
	In one dimension, the normal magnetic field is strictly constant ($B_1\equiv B_{\rm const}$), and compatibility relation \eqref{eq:1d-compatibility} ensures that the relative entropy flux is quadratic in $\bm U-\bm W$. In multiple dimensions, evaluating the uncorrected relative entropy flux for arbitrary states in $\G$ yields the Taylor expansion
	\begin{equation}\label{eq:md-uncorrected-relative-expansion}
		\begin{aligned}
			&q_{r,\bm n}(\bm U)-q_{r,\bm n}(\bm W)
			-\bm V_r(\bm W)^\top
			[\bm F_{\bm n}(\bm U)-\bm F_{\bm n}(\bm W)]\\
			&\qquad=\phi_r(\bm W)[B_{\bm n}(\bm U)-B_{\bm n}(\bm W)]
			+O(\norm{\bm U-\bm W}^2).
		\end{aligned}
	\end{equation}
	Along any perturbation path $\bm U=\bm W+\varepsilon\bm Z$ with $\phi_r(\bm W)DB_{\bm n}[\bm Z]\ne0$, the ratio of this uncorrected flux to the relative entropy $\Hrel_r(\bm U\mid\bm W)\sim O(\varepsilon^2)$ diverges at the rate $O(\varepsilon^{-1})$ as $\varepsilon\to0$. The Powell correction in \eqref{eq:md-powell-relative-flux} eliminates this leading-order linear defect, ensuring via \eqref{eq:md-quadratic-bounds} that the ratio $|\Qrel_{r,\bm n}(\bm U\mid\bm W)|/\Hrel_r(\bm U\mid\bm W)$ remains uniformly bounded by $C_r/c_r$ on any compact convex subset of $\G$.
\end{remark}

For any given numerical flux pair $(\bm G_{K,\nu},\widehat q_{K,\nu,r})$ and arbitrary reference state $\bm W\in\G$, we define the corrected boundary defect term
\begin{equation}\label{eq:md-side-defect}
	P_{K,\nu,r}(\bm W):=\psi_{r,\bm n_{K,\nu}}(\bm W)
	-\bm V_r(\bm W)^\top\bm G_{K,\nu}+\widehat q_{K,\nu,r}
	+\phi_r(\bm W)[B_{\bm n_{K,\nu}}(\bm W)
	-B_{\bm n_{K,\nu}}(\bm U_{K,\nu})].
\end{equation}
This serves as the multidimensional analogue of \eqref{eq:1d-weak-side-defects}, incorporating outward normal vectors and the Powell magnetic correction from \eqref{eq:md-powell-relative-flux}. When the numerical fluxes coincide with the physical fluxes evaluated at the face trace,
\begin{equation}\label{eq:physical-fluxes}
	(\bm G_{K,\nu},\widehat q_{K,\nu,r})
	=\bigl(\bm F_{\bm n_{K,\nu}}(\bm U_{K,\nu}),
	q_{r,\bm n_{K,\nu}}(\bm U_{K,\nu})\bigr),
\end{equation}
the boundary defect reduces identically to the Powell-corrected relative entropy flux:
\begin{equation*}
	P_{K,\nu,r}(\bm W)
	=\Qrel_{r,\bm n_{K,\nu}}(\bm U_{K,\nu}\mid\bm W).
\end{equation*}
Summing \eqref{eq:md-side-defect} over all boundary nodes $\nu$ with weights $c_{K,\nu}$ and invoking the discrete Gauss identities \eqref{eq:md-discrete-gauss}, the magnetic correction terms cancel identically, yielding the exact summation identity
\begin{equation}\label{eq:md-magnetic-cancellation}
	\sum_\nu c_{K,\nu}P_{K,\nu,r}(\bm W)
	=-\bm V_r(\bm W)^\top\sum_\nu c_{K,\nu}\bm G_{K,\nu}
	+\sum_\nu c_{K,\nu}\widehat q_{K,\nu,r}.
\end{equation}
Specializing to the updated cell average $\bm W=\overline{\bm U}_K^{n+1}\in\G$, we expand the local entropy defect $\mathcal B_{K,r}^n-\eta_r(\overline{\bm U}_K^{n+1})$. Using cell quadrature \eqref{eq:md-cell-average} and the definition of relative entropy,
\begin{equation*}
	\begin{aligned}
		\mathcal B_{K,r}^n-\eta_r(\bm W)
		={}&\mathcal S_{K,r}(\bm U_K^n)-\eta_r(\bm W)
		-\lambda_K\sum_\nu c_{K,\nu}\widehat q_{K,\nu,r}\\
		={}&\sum_\alpha\omega_{K,\alpha}^\circ
		\Hrel_r(\widehat{\bm U}_K^{[\alpha]}\mid\bm W)
		+\sum_\nu\omega_{K,\nu}
		\Hrel_r(\bm U_{K,\nu}\mid\bm W)\\
		&+\bm V_r(\bm W)^\top(\overline{\bm U}_K^n-\bm W)
		-\lambda_K\sum_\nu c_{K,\nu}\widehat q_{K,\nu,r}.
	\end{aligned}
\end{equation*}
Since the forward Euler update \eqref{eq:md-mean-update} implies $\overline{\bm U}_K^n-\bm W = \lambda_K\sum_\nu c_{K,\nu}\bm G_{K,\nu}$, substituting this into the expansion and employing cancellation identity \eqref{eq:md-magnetic-cancellation} yields the exact multidimensional entropy defect identity
\begin{equation}\label{eq:md-defect-identity}
	\mathcal B_{K,r}^n-\eta_r(\bm W)
	=\sum_\alpha\omega_{K,\alpha}^\circ
	\Hrel_r(\widehat{\bm U}_K^{[\alpha]}\mid\bm W)
	+\sum_\nu\omega_{K,\nu}
	\Hrel_r(\bm U_{K,\nu}\mid\bm W)
	-\lambda_K\sum_\nu c_{K,\nu}P_{K,\nu,r}(\bm W).
\end{equation}

\begin{remark}[Purpose of the boundary correction]
	\label{rem:md-reference-correction}
	The magnetic correction in \eqref{eq:md-side-defect} vanishes under boundary quadrature weights because the geometric closure and LDF identities \eqref{eq:md-discrete-gauss} give
	\begin{equation*}
		\phi_r(\bm W)\sum_\nu c_{K,\nu}
		[B_{\bm n_{K,\nu}}(\bm W)-B_{\bm n_{K,\nu}}(\bm U_{K,\nu})] = 0.
	\end{equation*}
	Consequently, this correction preserves the exact defect balance \eqref{eq:md-defect-identity} globally on cell $K$ while rendering each individual boundary term $P_{K,\nu,r}(\bm W)$ locally quadratic whenever \eqref{eq:physical-fluxes} is satisfied. Under \eqref{eq:physical-fluxes}, the quadratic bounds \eqref{eq:md-quadratic-bounds} imply
	\begin{equation*}
		|P_{K,\nu,r}(\bm W)|
		\le\frac{C_r}{c_r}\Hrel_r(\bm U_{K,\nu}\mid\bm W)
	\end{equation*}
	on any compact convex subset of $\G$, which controls the boundary defect rates as $\bm W\to\bm U_{K,\nu}$.
\end{remark}

\subsubsection{Boundary control and finite rates}
For wave speeds satisfying the strict positivity condition \eqref{eq:md-pp-speed-condition}, we employ the reference step \eqref{eq:md-reference-step} to define the compact segment of candidate cell averages:
\begin{equation}\label{eq:md-reference-path}
	\mathscr W_K
	:=\left\{\overline{\bm U}_K^n
	-\frac{\vartheta}{|K|}\sum_\nu c_{K,\nu}\bm G_{K,\nu}:
	0\le\vartheta\le\tau_K^0\right\}\subset\G.
\end{equation}
By Proposition~\ref{prop:md-pp} and the convexity of $\G$, $\mathscr W_K\subset\G$. We define the one-sided relative-entropy defect rates by
\begin{equation}\label{eq:md-relative-rates}
	\mathfrak c_{K,\nu,r}
	:=\sup_{\substack{\bm W\in\mathscr W_K\\\bm W\ne\bm U_{K,\nu}}}
	\frac{[P_{K,\nu,r}(\bm W)]_+}
	{\Hrel_r(\bm U_{K,\nu}\mid\bm W)},
	\qquad
	\mathfrak c_{K,\nu}:=\max_{1\le r\le M}\mathfrak c_{K,\nu,r},
\end{equation}
with the supremum over an empty set taken as zero. Whenever these rates are finite and $P_{K,\nu,r}(\bm U_{K,\nu})\le0$ if $\bm U_{K,\nu}\in\mathscr W_K$, we obtain the uniform bound
\begin{equation}\label{eq:md-rate-bound}
	P_{K,\nu,r}(\bm W)
	\le\mathfrak c_{K,\nu,r}
	\Hrel_r(\bm U_{K,\nu}\mid\bm W),
	\qquad \forall\,\bm W\in\mathscr W_K.
\end{equation}
Substituting \eqref{eq:md-rate-bound} with $\bm W=\overline{\bm U}_K^{n+1}$ into \eqref{eq:md-defect-identity} yields
\begin{equation}\label{eq:md-controlled-defect}
	\mathcal B_{K,r}^n-\eta_r(\overline{\bm U}_K^{n+1})
	\ge\sum_\alpha\omega_{K,\alpha}^\circ
	\Hrel_r(\widehat{\bm U}_K^{[\alpha]}\mid\overline{\bm U}_K^{n+1})
	+\sum_\nu
	(\omega_{K,\nu}-\lambda_Kc_{K,\nu}\mathfrak c_{K,\nu,r})
	\Hrel_r(\bm U_{K,\nu}\mid\overline{\bm U}_K^{n+1}).
\end{equation}
Because the relative entropy is nonnegative, establishing weak multi-entropy stability reduces to selecting interface wave speeds that ensure the finiteness of the rates $\mathfrak c_{K,\nu,r}$ and to restricting the time step so that each coefficient $\omega_{K,\nu}-\lambda_Kc_{K,\nu}\mathfrak c_{K,\nu,r}$ is nonnegative.

At each boundary quadrature node $\nu$, we define the one-sided entropy dissipation rate by
\begin{equation}\label{eq:md-one-sided-dissipation}
	d_{K,\nu,r}:=\bm V_r(\bm U_{K,\nu})^\top\bm G_{K,\nu}
	-\widehat q_{K,\nu,r}
	-\psi_{r,\bm n_{K,\nu}}(\bm U_{K,\nu}).
\end{equation}
Across an oriented interface, let $d_r^-$ and $d_r^+$ denote the values evaluated on the two sides using the respective outward unit normal vectors. Comparing \eqref{eq:md-one-sided-dissipation} with \eqref{eq:md-side-defect} shows that $P_{K,\nu,r}(\bm U_{K,\nu})=-d_{K,\nu,r}$. To control the quotient in \eqref{eq:md-relative-rates}, we impose at every boundary node and for each $1\le r\le M$ the dissipation dichotomy:
\begin{equation}\label{eq:md-side-condition}
	d_{K,\nu,r}>0\quad\text{or}\quad
	(\bm G_{K,\nu},\widehat q_{K,\nu,r})
	=(\bm F_{\bm n_{K,\nu}}(\bm U_{K,\nu}),
	q_{r,\bm n_{K,\nu}}(\bm U_{K,\nu})).
\end{equation}

\begin{lemma}[Finiteness of the relative-entropy rates]
	\label{lem:md-finite-rates}
	For fixed wave speeds satisfying \eqref{eq:md-pp-speed-condition} and \eqref{eq:md-side-condition}, the rates in \eqref{eq:md-relative-rates} are finite and satisfy \eqref{eq:md-rate-bound} throughout $\mathscr W_K$.
\end{lemma}
\begin{proof}
	Under $d_{K,\nu,r}>0$, continuity implies that $P_{K,\nu,r}(\bm W)<0$, and hence $[P_{K,\nu,r}(\bm W)]_+=0$, in an open neighborhood of $\bm U_{K,\nu}$. When the flux equalities in \eqref{eq:md-side-condition} hold, the boundary defect coincides with the Powell-corrected relative flux, $P_{K,\nu,r}(\bm W)=\Qrel_{r,\bm n_{K,\nu}}(\bm U_{K,\nu}\mid\bm W)$, so quadratic estimates \eqref{eq:md-quadratic-bounds} bound the quotient near $\bm U_{K,\nu}$. Outside that neighborhood, compactness of $\mathscr W_K$ and strict positivity $\Hrel_r(\bm U_{K,\nu}\mid\bm W)>0$ bound the continuous quotient. Finally, at the trace point itself, $P_{K,\nu,r}(\bm U_{K,\nu})=-d_{K,\nu,r}\le0$ while $\Hrel_r(\bm U_{K,\nu}\mid\bm U_{K,\nu})=0$, confirming that \eqref{eq:md-rate-bound} holds throughout $\mathscr W_K$.
\end{proof}

\subsubsection{Explicit speeds and the common time-step bound}
For distinct interface traces $\bm U^-\neq\bm U^+$, we define the relative-entropy flux velocities
\begin{equation}\label{eq:md-relative-speeds}
	\ell_r^-:=\frac{\Qrel_{r,\bm n}(\bm U^-\mid\bm U^+)}
	{\Hrel_r(\bm U^-\mid\bm U^+)},
	\qquad
	\ell_r^+:=\frac{\Qrel_{r,\bm n}(\bm U^+\mid\bm U^-)}
	{\Hrel_r(\bm U^+\mid\bm U^-)}.
\end{equation}
Substituting the HLL flux \eqref{eq:md-hll-flux}, the numerical entropy flux \eqref{eq:md-hll-entropy-flux}, and the Powell-modified operator \eqref{eq:md-powell-operators} into \eqref{eq:md-one-sided-dissipation} yields the factorization
\begin{equation}\label{eq:md-dissipation-factors}
	\begin{aligned}
		d_r^-=\frac{-a^-}{C_F}(a^+-\ell_r^+)
		\Hrel_r(\bm U^+\mid\bm U^-),\qquad 
		d_r^+=\frac{a^+}{C_F}(\ell_r^--a^-)
		\Hrel_r(\bm U^-\mid\bm U^+).
	\end{aligned}
\end{equation}
The magnetic term in \eqref{eq:md-powell-relative-flux} accounts directly for the Godunov--Powell source term in \eqref{eq:md-powell-operators}, thereby preserving this product structure. Using the initial wave-speed spread $C_F^0$ from \eqref{eq:md-initial-speeds}, we define the extremal directional envelopes
\begin{equation*}
	\Lambda_-:=\min_{1\le r\le M}\ell_r^-,
	\qquad \Lambda_+:=\max_{1\le r\le M}\ell_r^+,
\end{equation*}
and set the augmented interface wave speeds as
\begin{equation}\label{eq:md-joint-speeds}
	\begin{aligned}
		a^-:=\frac{a^{-,0}+\Lambda_-
			-\sqrt{(a^{-,0}-\Lambda_-)^2+(C_F^0)^2}}2,\qquad 
		a^+:=\frac{a^{+,0}+\Lambda_+
			+\sqrt{(a^{+,0}-\Lambda_+)^2+(C_F^0)^2}}2.
	\end{aligned}
\end{equation}
By construction, these explicit speeds satisfy $a^-<\min\{a^{-,0},\Lambda_-\}$ and $a^+>\max\{a^{+,0},\Lambda_+\}$, preserving positivity \eqref{eq:md-pp-speed-condition} and ensuring $d_r^\pm>0$ via \eqref{eq:md-dissipation-factors} for all $1\le r\le M$. By \eqref{eq:md-quadratic-bounds}, the ratios $\ell_r^\pm$ remain bounded on compact subsets of $\G$, so the speeds remain finite. If $\bm U^-=\bm U^+$, we set $(a^-,a^+)=(a^{-,0},a^{+,0})$ without evaluating the indeterminate ratios; flux consistency then satisfies the equality branch of \eqref{eq:md-side-condition}. Both cases respect orientation invariance and yield finite interface speeds fulfilling \eqref{eq:md-pp-speed-condition} and \eqref{eq:md-side-condition}.

With the interface wave speeds and numerical fluxes determined, we evaluate the cell reference steps $\tau_K^0$, segments $\mathscr W_K$, and defect rates $\mathfrak c_{K,\nu}$, and define the global time-step threshold
\begin{equation}\label{eq:md-relative-step}
	\tau_*^{\rm rel}:=\min_{K\in\Th}\left\{\tau_K^0,
	\min_{\nu\in\mathcal I_K^\partial}\frac{|K|\omega_{K,\nu}}
	{c_{K,\nu}\mathfrak c_{K,\nu}}\right\}>0,
\end{equation}
where division by zero is interpreted as $+\infty$. By Lemma~\ref{lem:md-finite-rates} and $\tau_K^0>0$, the threshold $\tau_*^{\rm rel}$ is strictly positive on any finite mesh $\Th$. We emphasize that numerical fluxes and rates remain fixed while the time step is restricted. Evaluating \eqref{eq:md-relative-step} directly requires upper bounds on the suprema in \eqref{eq:md-relative-rates}. Alternatively, one may initialize with a PP-admissible trial step and bisect until \eqref{eq:md-weak-entropy} holds; the strict positivity of $\tau_*^{\rm rel}$ guarantees finite termination in exact arithmetic. In contrast, the convex-decomposition criterion in Subsection~\ref{subsec:md-weak-convex} yields an explicit closed-form time-step bound.

\begin{theorem}[Weak multi-entropy stability by relative-entropy rates]
	\label{thm:md-weak-relative}
	Under the quadrature, admissibility, and LDF assumptions above, suppose that the wave speeds satisfy \eqref{eq:md-pp-speed-condition} and \eqref{eq:md-side-condition}. Then the forward Euler cell-average update \eqref{eq:md-mean-update} is weakly multi-entropy stable in the sense of Definition~\ref{def:md-weak-strong-entropy} for every $0<\tau\le\tau_*^{\rm rel}$.
\end{theorem}
\begin{proof}
	By Proposition~\ref{prop:md-pp}, any time step $0<\tau\le\tau_*^{\rm rel}\le\tau_K^0$ ensures that $\overline{\bm U}_K^{n+1}\in\mathscr W_K\subset\G$. Applying Lemma~\ref{lem:md-finite-rates} and the definition of $\tau_*^{\rm rel}$ in \eqref{eq:md-relative-step}, all boundary coefficients in \eqref{eq:md-controlled-defect} satisfy $\omega_{K,\nu}-\lambda_Kc_{K,\nu}\mathfrak c_{K,\nu,r}\ge0$. Because the relative entropy is nonnegative, \eqref{eq:md-controlled-defect} implies $\eta_r(\overline{\bm U}_K^{n+1})\le\mathcal B_{K,r}^n$ simultaneously for all $1\le r\le M$.
\end{proof}

\subsection{Comparison and the role of discrete divergence control}
\label{subsec:md-comparison}
The interface bridge identities established in Subsection~\ref{subsec:1d-weak-comparison} extend directly to the Powell-corrected states $\bm T^\pm$. Combining the modified flux identities \eqref{eq:md-hll-powell-identities} and the entropy flux relations \eqref{eq:md-hll-entropy-identities} with the outward normal of each adjacent cell, we obtain
\begin{equation}\label{eq:md-interface-bridge}
	\begin{aligned}
		d_r^-=(-a^-)\bigl[\mathscr E_r^{\rm HLL}-\eta_r(\bm T^-)
		+\Hrel_r(\bm T^-\mid\bm U^-)\bigr],\qquad
		d_r^+=a^+\bigl[\mathscr E_r^{\rm HLL}-\eta_r(\bm T^+)
		+\Hrel_r(\bm T^+\mid\bm U^+)\bigr].
	\end{aligned}
\end{equation}
Under the convex-decomposition conditions \eqref{eq:md-convex-entropy-conditions}, both the entropy difference $\mathscr E_r^{\rm HLL}-\eta_r(\bm T^\pm)$ and the relative entropy $\Hrel_r(\bm T^\pm\mid\bm U^\pm)$ are nonnegative. Because $-a^->0$ and $a^+>0$, we have $d_r^\pm\ge0$. If either dissipation vanishes, say $d_r^-=0$, then both nonnegative terms in the bracket must vanish: $\mathscr E_r^{\rm HLL}=\eta_r(\bm T^-)$ and $\Hrel_r(\bm T^-\mid\bm U^-)=0$. By the strict convexity of $\eta_r$, this forces $\bm T^-=\bm U^-$, and the numerical flux reduces to the physical flux. Consequently, the interface conditions of the convex-decomposition criterion strictly imply the dissipation dichotomy \eqref{eq:md-side-condition}, so wave speeds satisfying the hypotheses of Theorem~\ref{thm:md-weak-convex} also satisfy the interface conditions of Theorem~\ref{thm:md-weak-relative}. Although both criteria target the same local entropy budget $\mathcal B_{K,r}^n$, this interface implication does not establish an ordering between their time-step bounds $\tau_*^{\rm cvx}$ and $\tau_*^{\rm rel}$.

\begin{remark}[The role of discrete divergence control]
	\label{rem:md-divergence-control}
	The LDF property plays two essential roles in our multidimensional analysis. First, it fulfills the discrete divergence hypothesis required for interior state positivity in Proposition~\ref{prop:md-pp}. Second, via the discrete Gauss formula \eqref{eq:md-discrete-gauss}, it enables the exact algebraic cancellation in \eqref{eq:md-magnetic-cancellation}. If the discrete normal magnetic constraint were violated, leaving a nonzero local divergence residual $\delta_K^B:=\sum_\nu c_{K,\nu} B_{\bm n_{K,\nu}}(\bm U_{K,\nu})\ne0$, the cell entropy defect identity \eqref{eq:md-defect-identity} would acquire the spurious defect term
	\begin{equation*}
		-\lambda_K\phi_r(\bm W)\delta_K^B,
	\end{equation*}
	which is generically sign-indefinite because the sign of $\phi_r(\bm W) = (\gamma-1)\rho f_r'(s)(\bm v\cdot\bm B)/p$ depends on the local flow alignment. Such a sign-indefinite term precludes provable cell-average multi-entropy inequalities. Thus, the LDF property combined with the Godunov--Powell interface discretization exactly eliminates this divergence error at the fully discrete level.
\end{remark}

The corresponding global entropy stability results under either criterion follow directly from Remark~\ref{rem:md-global-entropy}.

\section{Weak-to-strong lifting}
\label{sec:lifting}

Sections~\ref{sec:1d} and~\ref{sec:md} established weak entropy stability for cell averages. To achieve strong multi-entropy stability (Definitions~\ref{def:weak-strong-entropy} and~\ref{def:md-weak-strong-entropy}), intra-element polynomial variations must also be controlled. Extending the framework of \cite{WuEPO2026} to compressible MHD, we construct a weak-to-strong (W2S) lifting operator: we first enforce physical admissibility ($\rho>0$, $p>0$) via positivity-preserving scaling, and then contract the polynomial variation about its cell average to satisfy all prescribed discrete entropy inequalities simultaneously. Both operations strictly preserve the cell average and the magnetic solenoidal involution.

\subsection{The cell average and its entropy bounds}
\label{subsec:lifting-data}

Let $\Th:=\{I_j\}_{j=1}^{N_x}$ for $d=1$, and let $\Th$ be the conforming polytopal mesh of Section~\ref{sec:md} for $d\in\{2,3\}$. On each cell $K\in\Th$, denote the quadrature nodes and weights of \eqref{eq:1d-cell-average} and \eqref{eq:md-cell-average} by $\{\bm x_{K,a}\}$ and $\{\omega_{K,a}\}$. The weights are strictly positive, satisfy $\sum_a\omega_{K,a}=1$, and are exact for polynomials in $\mathbb P^k(K)$, so that
\begin{equation}\label{eq:lift-quadrature}
	\overline{\bm U}_K
	=\sum_a\omega_{K,a}\bm U_K(\bm x_{K,a}),
	\qquad \forall\,\bm U_K\in[\mathbb P^k(K)]^8.
\end{equation}
When all quadrature states belong to $\G$, the discrete cell entropy is
\begin{equation}\label{eq:lift-quadrature-entropy}
	\mathcal S_{K,r}(\bm U_K)
	:=\sum_a\omega_{K,a}\eta_r(\bm U_K(\bm x_{K,a})),
	\qquad 1\le r\le M,
\end{equation}
unifying \eqref{eq:1d-quadrature-entropy} and \eqref{eq:md-quadrature-entropy}.

Let $\bm U_K^\star\in[\mathbb P^k(K)]^8$ be a provisional polynomial (from an unconstrained DG update or finite-volume reconstruction) with cell average $\bm W_K:=\overline{\bm U}_K^\star$. Assume $\bm U_K^\star$ respects the magnetic constraint ($B_{1,K}^\star\equiv B_{\rm const}$ for $d=1$ and $(B_{1,K}^\star,\ldots,B_{d,K}^\star)^\top\in\mathbb V_{\rm div}^k(K)$ for $d\in\{2,3\}$), although its point values need not belong to $\G$. The time discretization provides local entropy bounds $\mathcal B_{K,r}$ satisfying weak multi-entropy stability:
\begin{equation}\label{eq:lift-feasibility}
	\overline{\bm U}_K^\star=\bm W_K\in\G,
	\qquad
	\eta_r(\bm W_K)\le\mathcal B_{K,r},
	\quad 1\le r\le M.
\end{equation}
For forward Euler steps, $\bm W_K=\overline{\bm U}_K^{n+1}$ and $\mathcal B_{K,r}=\mathcal B_{K,r}^n$ are given by \eqref{eq:1d-weak-bound} or \eqref{eq:md-weak-bound}; Section~\ref{sec:high-order-time} derives the bounds for SSP multistep methods.

Let $\mathcal X_K\subset K$ be a finite set containing all entropy quadrature nodes $\{\bm x_{K,a}\}$ and any additional points needed to evaluate spatial residuals at the next time level. The lifting guarantees admissibility on $\mathcal X_K$, which suffices to advance the scheme. Exterior boundary traces are prescribed in $\G$.

\begin{proposition}[Weak stability as feasibility of strong entropy bounds]
	\label{prop:lift-entropy-feasibility}
	Let $\bm W_K\in\G$ (with $B_1(\bm W_K)=B_{\rm const}$ for $d=1$), and let $\mathcal B_{K,r}$ ($1\le r\le M$) be prescribed entropy bounds. Within the polynomial space respecting the magnetic constraint, there exists a polynomial $\bm U_K\in[\mathbb P^k(K)]^8$ with $\overline{\bm U}_K=\bm W_K$ that is admissible on $\mathcal X_K$ and satisfies $\mathcal S_{K,r}(\bm U_K)\le\mathcal B_{K,r}$ for all $1\le r\le M$ if and only if
	\begin{equation}\label{eq:lift-feasibility-equivalence}
		\eta_r(\bm W_K)\le\mathcal B_{K,r},
		\qquad 1\le r\le M.
	\end{equation}
	More precisely, for any polynomial $\bm U_K$ with cell average $\bm W_K$ admissible on $\mathcal X_K$, the discrete Jensen gap
	\begin{equation}\label{eq:lift-gap-and-margin}
		\mathcal J_{K,r}(\bm U_K)
		:=\mathcal S_{K,r}(\bm U_K)-\eta_r(\bm W_K)\ge0
	\end{equation}
	satisfies the strong bounds if and only if
	\begin{equation}\label{eq:lift-gap-constraint}
		\mathcal J_{K,r}(\bm U_K)
		\le\mathcal B_{K,r}-\eta_r(\bm W_K),
		\qquad 1\le r\le M.
	\end{equation}
\end{proposition}
\begin{proof}
	Necessity follows from Jensen's inequality and \eqref{eq:lift-quadrature}:
	\begin{equation*}
		\eta_r(\bm W_K) = \eta_r\left(\sum_a\omega_{K,a}\bm U_K(\bm x_{K,a})\right) \le \sum_a\omega_{K,a}\eta_r(\bm U_K(\bm x_{K,a})) = \mathcal S_{K,r}(\bm U_K) \le \mathcal B_{K,r}.
	\end{equation*}
	Sufficiency is demonstrated by the constant polynomial $\bm U_K\equiv\bm W_K$, which satisfies the magnetic constraint, belongs to $\G$ on $\mathcal X_K$, and attains $\mathcal S_{K,r}(\bm W_K)=\eta_r(\bm W_K)\le\mathcal B_{K,r}$. Subtracting $\eta_r(\bm W_K)$ yields \eqref{eq:lift-gap-constraint}. Expanding relative entropy \eqref{eq:1d-weak-relative-quantities} and applying \eqref{eq:lift-quadrature} gives
	\begin{equation}\label{eq:lift-jensen-relative-entropy}
		\mathcal J_{K,r}(\bm U_K)
		=\sum_a\omega_{K,a}
		\Hrel_r(\bm U_K(\bm x_{K,a})\mid\bm W_K),
	\end{equation}
	completing the proof.
\end{proof}

Identity \eqref{eq:lift-jensen-relative-entropy} shows that $\mathcal J_{K,r}(\bm U_K)$ quantifies the entropy fluctuation induced by polynomial variations within $K$. By strict convexity of $\eta_r$, this gap vanishes if and only if all quadrature values coincide with $\bm W_K$. On compact convex subsets of $\G$, Lemma~\ref{lem:harten-convexity} ensures that $\mathcal J_{K,r}(\bm U_K)$ is equivalent, up to constants, to the weighted $L^2$-deviation $\sum_a\omega_{K,a}\norm{\bm U_K(\bm x_{K,a})-\bm W_K}^2$.

\subsection{Simultaneous entropy stability lifting after positivity limiting}
\label{subsec:lifting-construction}

For positivity thresholds $0<\epsilon_{\rho,K}<\rho(\bm W_K)$ and $0<\epsilon_{p,K}<p(\bm W_K)$, define the truncated admissible set
\begin{equation}\label{eq:lift-positive-set}
	\G_{K,\epsilon}
	:=\{\bm U\in\G:
	\rho(\bm U)\ge\epsilon_{\rho,K},\quad
	p(\bm U)\ge\epsilon_{p,K}\}.
\end{equation}
Since pressure $p(\bm U)$ is concave on $\{\bm U\in\G:\rho>0\}$, $\G_{K,\epsilon}$ is closed and convex, with $\bm W_K$ in its interior. Along the contraction ray
\begin{equation}\label{eq:lift-ray}
	\bm U_K(\bm x;\theta)
	:=\bm W_K+\theta\bigl(\bm U_K^\star(\bm x)-\bm W_K\bigr),
	\qquad \theta\in[0,1],
\end{equation}
we first apply the positivity-preserving (PP) limiter \cite{ZhangShu2010,WuShu2019}:
\begin{equation}\label{eq:lift-pp}
	\theta_K^{\rm PP}
	:=\max\{\theta\in[0,1]:
	\bm U_K(\bm x;\theta)\in\G_{K,\epsilon}
	\text{ for every }\bm x\in\mathcal X_K\},
	\qquad
	\bm U_K^{\rm PP}(\bm x):=\bm U_K(\bm x;\theta_K^{\rm PP}).
\end{equation}
Since $\mathcal X_K$ is finite and $\bm W_K$ satisfies both thresholds strictly, $\theta_K^{\rm PP}>0$. Admissibility is essential here because $\eta_r$ and its derivatives are defined only on $\G$.

For each entropy $\eta_r$ ($1\le r\le M$), define the one-dimensional entropy profile
\begin{equation}\label{eq:lift-profile}
	\Phi_{K,r}(\vartheta)
	:=\mathcal S_{K,r}\bigl(
	\bm W_K+\vartheta(\bm U_K^{\rm PP}-\bm W_K)\bigr),
	\qquad 0\le\vartheta\le1.
\end{equation}
Strict convexity of $\eta_r$ ensures that $\Phi_{K,r}$ is continuous and convex on $[0,1]$. Since $\overline{\bm U}_K^{\rm PP}=\bm W_K$, quadrature exactness \eqref{eq:lift-quadrature} yields
\begin{equation}\label{eq:lift-profile-origin}
	\Phi_{K,r}(0)=\eta_r(\bm W_K),
	\qquad
	\Phi_{K,r}'(0)
	=\bm V_r(\bm W_K)^\top
	\sum_a\omega_{K,a}
	\bigl(\bm U_K^{\rm PP}(\bm x_{K,a})-\bm W_K\bigr)=0.
\end{equation}
Hence $\Phi_{K,r}(\vartheta)$ is nondecreasing on $[0,1]$, attaining its minimum at $\vartheta=0$.

The entropy factor $\vartheta_{K,r}\in[0,1]$ is defined by
\begin{equation}\label{eq:lift-entropy-factor}
	\begin{cases}
		\vartheta_{K,r}=1,
		&\Phi_{K,r}(1)\le\mathcal B_{K,r},\\[2mm]
		\Phi_{K,r}(\vartheta_{K,r})=\mathcal B_{K,r},\quad
		\vartheta_{K,r}\in[0,1),
		&\Phi_{K,r}(1)>\mathcal B_{K,r}.
	\end{cases}
\end{equation}
In the second case, strict convexity of $\eta_r$ and $\Phi'_{K,r}(0)=0$ imply that $\Phi_{K,r}$ is strictly increasing; weak feasibility \eqref{eq:lift-feasibility} guarantees a unique root in $[0,1)$, computed via bisection or a safeguarded Newton iteration.
To satisfy all $M$ entropy inequalities simultaneously, we define
\begin{equation}\label{eq:lift-final}
	\vartheta_K:=\min_{1\le r\le M}\vartheta_{K,r},
	\qquad \theta_K:=\theta_K^{\rm PP}\vartheta_K,
	\qquad
	\bm U_K^{\rm new}
	:=\bm W_K+\theta_K(\bm U_K^\star-\bm W_K).
\end{equation}

\begin{theorem}[Weak-to-strong multi-entropy lifting]
	\label{thm:weak-strong-lifting}
	Under the preceding hypotheses, $\bm U_K^{\rm new}$ in \eqref{eq:lift-final} preserves the cell average $\overline{\bm U}_K^{\rm new}=\bm W_K$, belongs to $\G_{K,\epsilon}$ on $\mathcal X_K$, and satisfies the strong multi-entropy bounds
	\begin{equation}\label{eq:lift-strong}
		\mathcal S_{K,r}(\bm U_K^{\rm new})\le\mathcal B_{K,r},
		\qquad 1\le r\le M.
	\end{equation}
	Furthermore, it preserves the solenoidal constraint ($B_{1,K}^{\rm new}\equiv B_{\rm const}$ for $d=1$ and $(B_{1,K}^{\rm new},\allowbreak\ldots,B_{d,K}^{\rm new})^\top\in\mathbb V_{\rm div}^k(K)$ for $d\in\{2,3\}$). These conclusions remain valid if $\theta_K$ is replaced by any smaller nonnegative scalar.
\end{theorem}
\begin{proof}
	Linearity of the mean ensures $\overline{\bm U}_K^{\rm new}=\bm W_K$. Since $\bm U_K^{\rm PP}(\bm x),\bm W_K\in\G_{K,\epsilon}$, convexity of $\G_{K,\epsilon}$ ensures $\bm U_K^{\rm new}(\bm x)\in\G_{K,\epsilon}$ for all $\bm x\in\mathcal X_K$. For each $1\le r\le M$, monotonicity of $\Phi_{K,r}$ and \eqref{eq:lift-entropy-factor} yield
	\begin{equation*}
		\mathcal S_{K,r}(\bm U_K^{\rm new})
		=\Phi_{K,r}(\vartheta_K)
		\le\Phi_{K,r}(\vartheta_{K,r})
		\le\mathcal B_{K,r}.
	\end{equation*}
	Because the scalar factor $\theta_K$ multiplies all spatial components of the polynomial variation equally,
	\begin{equation*}
		\nabla\cdot\bm B_K^{\rm new}
		=\nabla\cdot\bigl(\bm B(\bm W_K)+\theta_K(\bm B_K^\star-\bm B(\bm W_K))\bigr)
		=\theta_K\nabla\cdot\bm B_K^\star=0
	\end{equation*}
	for $d\in\{2,3\}$, and $B_{1,K}^{\rm new}\equiv B_{\rm const}$ for $d=1$. Finally, replacing $\theta_K$ by any $0\le\theta\le\theta_K$ corresponds to $\vartheta\le\vartheta_K$, preserving admissibility by convexity and $\mathcal S_{K,r}\le\mathcal B_{K,r}$ by monotonicity.
\end{proof}

\begin{remark}[Maximal entropy factors and compatibility with limiters]
	\label{rem:lift-exact-radius}
	By monotonicity, $\vartheta_{K,r}$ defined in \eqref{eq:lift-entropy-factor} is the maximal admissible contraction parameter:
	\begin{equation}\label{eq:lift-exact-radius}
		\vartheta_{K,r}
		=\max\{\vartheta\in[0,1]:
		\Phi_{K,r}(\vartheta)\le\mathcal B_{K,r}\}.
	\end{equation}
	Convexity of $\Phi_{K,r}$ implies that the secant ratio $[\mathcal B_{K,r}-\eta_r(\bm W_K)]/[\Phi_{K,r}(1)-\eta_r(\bm W_K)]$ provides a lower bound to $\vartheta_{K,r}$, but using it contracts more than necessary and can induce order reduction near smooth extrema; we therefore use the root defined in \eqref{eq:lift-entropy-factor}. By Theorem~\ref{thm:weak-strong-lifting}, any additional limiter contracting further along the same ray with $0\le\theta\le\theta_K$ preserves both admissibility and strong multi-entropy stability.
\end{remark}

\section{Multi-entropy stability under high-order time discretization}
\label{sec:high-order-time}

The forward Euler weak and strong multi-entropy stability results extend to high-order strong-stability-preserving (SSP) multistep methods via their canonical decomposition into forward Euler blocks \cite{GottliebShuTadmor2001,HadjimichaelKetchesonLocziNemeth2016}. Unlike SSP Runge--Kutta schemes, which require substage limiting and can suffer order reduction, SSP multistep methods evaluate spatial residuals on the solution history and invoke the W2S lifting operator only once per time step, at $t^{n+1}$. The convex decomposition yields a local multi-entropy budget that guarantees fully discrete strong multi-entropy stability.

Let $\bm U_h=(\bm U_K)_{K\in\Th}$ denote a global solution with $\bm U_K\in[\mathbb P^k(K)]^8$ satisfying $\bm U_K(\bm x)\in\G$ for all $\bm x\in\mathcal X_K$ and the solenoidal constraint $B_{1,K}\equiv B_{\rm const}$ ($d=1$) or $(B_{1,K},\ldots,B_{d,K})^\top\in\mathbb V_{\rm div}^k(K)$ ($d\in\{2,3\}$). Prescribed exterior boundary traces also belong to $\G$. Let $\mathcal L_h$ denote the spatial residual operator of Sections~\ref{sec:1d} and~\ref{sec:md}, whose cell averages satisfy \eqref{eq:1d-mean-update} or \eqref{eq:md-mean-update}. The operator $\mathcal L_h$ respects the magnetic involution: its $B_1$ component vanishes in one dimension, and its magnetic components lie in $\mathbb V_{\rm div}^k(K)$ in multiple dimensions. This preservation holds by construction for LDF DG schemes, and via divergence-free reconstruction for high-order finite-volume methods.

Define the forward Euler sub-step operator $\bm Y_h$ and its local entropy bound by
\begin{equation}\label{eq:time-fe-block}
	\bm Y_h(\bm U_h;\tau):=\bm U_h+\tau\mathcal L_h(\bm U_h),
\end{equation}
\begin{equation}\label{eq:time-fe-bound}
	\mathcal B_{K,r}^{\rm FE}(\bm U_h;\tau)
	:=\mathcal S_{K,r}(\bm U_K)
	-\frac{\tau}{|K|}\mathcal F_{K,r}(\bm U_h),
\end{equation}
where the outward numerical entropy flux is
\begin{equation}\label{eq:time-entropy-flux-functional}
	\mathcal F_{K,r}(\bm U_h):=
	\begin{cases}
		\widehat q_{j+1/2,r}-\widehat q_{j-1/2,r},&d=1,\quad K=I_j,\\
		\displaystyle\sum_{\nu\in\mathcal I_K^\partial}
		c_{K,\nu}\widehat q_{K,\nu,r},&d\in\{2,3\}.
	\end{cases}
\end{equation}
For any admissible $\bm U_h$, Sections~\ref{sec:1d} and~\ref{sec:md} determine wave speeds and a time-step bound $\tau_{\rm FE}(\bm U_h)>0$ such that
\begin{equation}\label{eq:time-fe-property}
	\overline{\bm Y}_K(\bm U_h;\tau)\in\G,
	\qquad
	\eta_r(\overline{\bm Y}_K(\bm U_h;\tau))
	\le\mathcal B_{K,r}^{\rm FE}(\bm U_h;\tau),
\end{equation}
for all $0\le\tau\le\tau_{\rm FE}(\bm U_h)$ and $1\le r\le M$. For $\tau=0$, \eqref{eq:time-fe-property} reduces to $\eta_r(\overline{\bm U}_K)\le\mathcal S_{K,r}(\bm U_K)$ via Jensen's inequality.

Let $\Delta t_n=t^{n+1}-t^n>0$, and consider an explicit $m$-step SSP multistep method with nonnegative coefficients satisfying \cite{GottliebShuTadmor2001}
\begin{equation}\label{eq:time-ssp-coefficients}
	c_{\ell,n},d_{\ell,n}\ge0,\qquad
	\sum_{\ell=0}^{m-1}c_{\ell,n}=1,\qquad
	c_{\ell,n}=0\ \Longrightarrow\ d_{\ell,n}=0.
\end{equation}
Omitting terms with $c_{\ell,n}=0$, the provisional multistep update is a convex combination of forward Euler steps:
\begin{equation}\label{eq:time-ssp-multistep}
	\bm U_h^{n+1,\star}
	=\sum_\ell c_{\ell,n}
	\bm Y_h(\bm U_h^{n-\ell};\tau_{\ell,n}),
	\qquad
	\tau_{\ell,n}:=\frac{d_{\ell,n}}{c_{\ell,n}}\Delta t_n.
\end{equation}
To ensure stability of each constituent block, we impose the step restriction
\begin{equation}\label{eq:time-ssp-step}
	\Delta t_n\le
	\min_{\ell:d_{\ell,n}>0}
	\frac{c_{\ell,n}}{d_{\ell,n}}
	\tau_{\rm FE}(\bm U_h^{n-\ell}).
\end{equation}
Combining the forward Euler bounds \eqref{eq:time-fe-bound} with weights $c_{\ell,n}$ defines the multistep entropy budget:
\begin{equation}\label{eq:time-multistep-bound}
	\begin{aligned}
		\mathcal B_{K,r}^{n,\rm MS}
		&:=\sum_\ell c_{\ell,n}
		\mathcal B_{K,r}^{\rm FE}(\bm U_h^{n-\ell};\tau_{\ell,n})
		=\sum_\ell c_{\ell,n}\mathcal S_{K,r}(\bm U_K^{n-\ell})
		-\frac{\Delta t_n}{|K|}\sum_\ell d_{\ell,n}
		\mathcal F_{K,r}(\bm U_h^{n-\ell}).
	\end{aligned}
\end{equation}
The W2S lifting of Section~\ref{sec:lifting} is then applied to $\bm U_K^{n+1,\star}$ with $\bm W_K=\overline{\bm U}_K^{n+1,\star}$ and $\mathcal B_{K,r}=\mathcal B_{K,r}^{n,\rm MS}$.

\begin{theorem}[Multi-entropy stability of the lifted SSP update]
	\label{thm:time-ssp-multistep}
	Suppose each previous state $\bm U_h^{n-\ell}$ with $c_{\ell,n}>0$ is nodal-admissible on $\mathcal X_K$ and solenoidal, with exterior boundary traces in $\G$. Under \eqref{eq:time-ssp-coefficients} and \eqref{eq:time-ssp-step}, the updated cell average satisfies
	\begin{equation}\label{eq:time-ms-weak}
		\overline{\bm U}_K^{n+1,\star}\in\G,
		\qquad
		\eta_r(\overline{\bm U}_K^{n+1,\star})
		\le\mathcal B_{K,r}^{n,\rm MS},\qquad 1\le r\le M,
	\end{equation}
	for every $K\in\Th$. Furthermore, the lifted polynomial $\bm U_K^{n+1}$ preserves this cell average ($\overline{\bm U}_K^{n+1}=\overline{\bm U}_K^{n+1,\star}$), satisfies the solenoidal constraint ($B_{1,K}^{n+1}\equiv B_{\rm const}$ or $(B_{1,K}^{n+1},\ldots,B_{d,K}^{n+1})^\top\in\mathbb V_{\rm div}^k(K)$), belongs to $\G$ on $\mathcal X_K$, and obeys the strong multi-entropy bounds
	\begin{equation}\label{eq:time-ms-strong}
		\mathcal S_{K,r}(\bm U_K^{n+1})
		\le\mathcal B_{K,r}^{n,\rm MS},\qquad 1\le r\le M.
	\end{equation}
\end{theorem}
\begin{proof}
	By linearity, $\overline{\bm U}_K^{n+1,\star} = \sum_\ell c_{\ell,n}\overline{\bm Y}_K(\bm U_h^{n-\ell};\tau_{\ell,n})$ is a convex combination of forward Euler cell averages. Under \eqref{eq:time-ssp-step}, each constituent step satisfies $\tau_{\ell,n}\le\tau_{\rm FE}(\bm U_h^{n-\ell})$, so \eqref{eq:time-fe-property} guarantees $\overline{\bm Y}_K(\bm U_h^{n-\ell};\tau_{\ell,n})\in\G$ and $\eta_r(\overline{\bm Y}_K)\le\mathcal B_{K,r}^{\rm FE}$. Convexity of $\G$ yields $\overline{\bm U}_K^{n+1,\star}\in\G$, while Jensen's inequality implies
	\begin{equation*}
		\eta_r(\overline{\bm U}_K^{n+1,\star})
		\le\sum_\ell c_{\ell,n}\eta_r\bigl(\overline{\bm Y}_K(\bm U_h^{n-\ell};\tau_{\ell,n})\bigr)
		\le\sum_\ell c_{\ell,n}\mathcal B_{K,r}^{\rm FE}(\bm U_h^{n-\ell};\tau_{\ell,n})
		=\mathcal B_{K,r}^{n,\rm MS},
	\end{equation*}
	proving \eqref{eq:time-ms-weak}. Because update \eqref{eq:time-ssp-multistep} is an affine combination of solenoidal states, $\bm U_h^{n+1,\star}$ satisfies the linear magnetic constraints. Theorem~\ref{thm:weak-strong-lifting} then directly yields nodal admissibility on $\mathcal X_K$, magnetic preservation, and strong multi-entropy stability \eqref{eq:time-ms-strong}.
\end{proof}

Define the total discrete entropy by
\begin{equation}\label{eq:time-global-entropy}
	\mathcal E_{h,r}(\bm U_h)
	:=\sum_{K\in\Th}|K|\mathcal S_{K,r}(\bm U_K).
\end{equation}
Multiplying \eqref{eq:time-ms-strong} by $|K|$ and summing over $K\in\Th$, interior numerical entropy fluxes cancel by conservation and anti-symmetry. Under periodic boundary conditions, boundary terms vanish, yielding the stability recursion
\begin{equation}\label{eq:time-ms-global}
	\mathcal E_{h,r}(\bm U_h^{n+1})
	\le\sum_\ell c_{\ell,n}\mathcal E_{h,r}(\bm U_h^{n-\ell})
	\le\max_{\ell:c_{\ell,n}>0}\mathcal E_{h,r}(\bm U_h^{n-\ell}).
\end{equation}
Inequality \eqref{eq:time-ms-global} establishes that the discrete entropy at $t^{n+1}$ is bounded by a convex combination of the entropies of retained history levels, precluding unbounded growth. As is standard for multistep schemes, the total entropy need not decrease monotonically between consecutive steps $(t^n, t^{n+1})$, but remains bounded by the startup values $\max_{0\le\ell\le m-1}\mathcal E_{h,r}(\bm U_h^\ell)$. For domains with physical boundaries, boundary fluxes accumulate telescopically in \eqref{eq:time-ms-global} via \eqref{eq:md-boundary-entropy-flux}.

\section{Numerical tests}
\label{sec:numer}

In this section, we validate the accuracy, multi-entropy stability, and physical admissibility of the proposed fully discrete schemes. Subsection~\ref{subsec:num-setup} summarizes the computational setup. Smooth test cases in Subsection~\ref{subsec:num-smooth} confirm the theoretical convergence rates, while challenging Riemann and two-dimensional problems in Subsections~\ref{subsec:num-riemann} and~\ref{subsec:num-2d} demonstrate robustness in near-vacuum, high-Mach, and strongly magnetized regimes.

\subsection{Numerical setup}\label{subsec:num-setup}

\subsubsection{Entropy family}\label{subsubsec:num-entropy}
We test the Harten entropy pairs \eqref{eq:intro-harten-entropy} generated by the parametric family
\begin{equation}\label{eq:entropy-parameterized-family}
	f_r(s):=
	\begin{cases}
		\displaystyle\frac{\exp\bigl(\tau_r(s-s_0)\bigr)-1}{\tau_r},
		&\tau_r\ne0,\\[2mm]
		s-s_0,&\tau_r=0,
	\end{cases}
	\qquad
	\tau_r:=\frac{\xi_r}{\gamma},\quad \xi_r<1,
	\quad r=1,\ldots,M,
\end{equation}
with the value at $\tau_r=0$ defined by continuous extension. Since
\begin{equation}\label{eq:entropy-parameterized-convexity}
	f_r'(s)=\exp\bigl(\tau_r(s-s_0)\bigr)>0,
	\qquad
	f_r'(s)-\gamma f_r''(s)=(1-\xi_r)\exp\bigl(\tau_r(s-s_0)\bigr)>0,
\end{equation}
these functions satisfy Harten's convexity conditions \eqref{eq:intro-harten-conditions}. For $\xi_r=0$, we recover the classical physical entropy $\eta_r=-\rho s$. Unless specified otherwise, we set $s_0=0$ and examine up to three simultaneous entropies with parameters
\begin{equation}\label{eq:num-entropy-parameters}
	\{\xi_r\}_{r=1}^M=
	\begin{cases}
		\{0\},&M=1,\\
		\{0,-0.5\},&M=2,\\
		\{0,0.9,-0.5\},&M=3.
	\end{cases}
\end{equation}
While our theoretical analysis applies to any finite family, this selection suffices to evaluate multi-entropy control.

\subsubsection{Spatial discretization and limiting}\label{subsubsec:num-spatial}
All simulations employ discontinuous Galerkin (DG) approximations of polynomial degree $k=2$ ($\mathbb P^2$). In multidimensional tests, we adopt Cartesian meshes with the locally divergence-free (LDF) magnetic spaces and quadrature rules from \cite{WuShu2019}, together with the Godunov--Powell interface discretization \eqref{eq:md-powell-operators}. The adiabatic index is set to $\gamma=5/3$ unless noted otherwise.

The W2S lifting limiter of Section~\ref{sec:lifting} is applied in all computations: positivity-preserving (PP) scaling is applied first, followed by a single cellwise scaling to satisfy the multi-entropy bounds. For discontinuous solutions, a standard WENO limiter precedes these scalings to suppress spurious oscillations, but is omitted in smooth accuracy tests.

We compare two variants differing solely in their interface wave speeds:
\begin{itemize}[leftmargin=2em]
	\item \textbf{Method 1}: uses the convex-decomposition wave speeds from Subsections~\ref{subsec:1d-weak-convex} and~\ref{subsec:md-weak-convex};
	\item \textbf{Method 2}: uses the relative-entropy-rate wave speeds from Subsections~\ref{subsec:1d-weak-relative} and~\ref{subsec:md-weak-relative}.
\end{itemize}
Because both methods produce visually indistinguishable profiles in most tests, solution plots for Method~2 are shown only when noticeable differences arise.

\subsubsection{Time discretization}\label{subsubsec:num-time}
To initialize the multistep scheme, the first three steps are advanced with the third-order SSP Runge--Kutta method \cite{RanochaSayyariDalcinParsaniKetcheson2020}. Subsequent integration uses the four-step, third-order variable-step formula SSPMSV43 \cite{HadjimichaelKetchesonLocziNemeth2016}. For $n\ge3$, let $H_n:=t^n-t^{n-3}$ and $\Omega_n:=H_n/\Delta t_n>2$. In \eqref{eq:time-ssp-multistep}, the only nonzero coefficients are
\begin{equation}\label{eq:time-ms43-coefficients}
	c_{0,n}=\frac{(\Omega_n+1)^2(\Omega_n-2)}{\Omega_n^3},
	\quad
	c_{3,n}=\frac{3\Omega_n+2}{\Omega_n^3},
	\quad
	d_{0,n}=\frac{(\Omega_n+1)^2}{\Omega_n^2},
	\quad
	d_{3,n}=\frac{\Omega_n+1}{\Omega_n^2},
\end{equation}
which satisfy the third-order accuracy and SSP conditions \eqref{eq:time-ssp-coefficients}. The corresponding effective forward Euler steps are
\begin{equation}\label{eq:time-ms43-effective}
	\tau_{0,n}=\frac{d_{0,n}}{c_{0,n}}\Delta t_n,
	\qquad
	\tau_{3,n}=\frac{d_{3,n}}{c_{3,n}}\Delta t_n.
\end{equation}
Writing $\mathcal S_{K,r}^n:=\mathcal S_{K,r}(\bm U_K^n)$, the multistep entropy bound \eqref{eq:time-multistep-bound} specializes to
\begin{equation}\label{eq:num-sspms-budget}
	\mathcal B_{K,r}^{n,\mathrm{MS}}
	=
	c_{0,n}\mathcal S_{K,r}^{n}
	+c_{3,n}\mathcal S_{K,r}^{n-3}
	-
	\frac{\Delta t_n}{|K|}
	\bigl[
	d_{0,n}\mathcal F_{K,r}(\bm U_h^{n})+d_{3,n}\mathcal F_{K,r}(\bm U_h^{n-3})
	\bigr],
\end{equation}
where $\mathcal F_{K,r}$ is the outward numerical entropy flux sum \eqref{eq:time-entropy-flux-functional} evaluated with the selected interface speeds.

\subsubsection{Step-size selection}\label{subsubsec:num-step}
At time level $t^n$, we define the directional wave speeds
\begin{equation}\label{eq:num-directional-cfl-speed}
	\Lambda_i^n
	:=
	\max_{K\in\Th}
	\bigl\{
	|v_i(\overline{\bm U}_K^n)|
	+\mathcal C(\overline{\bm U}_K^n;\bm e_i)
	\bigr\}
	+
	\max_{K\in\Th}
	\left\{
	|v_i(\overline{\bm U}_K^n)|
	+\frac{|\bm B(\overline{\bm U}_K^n)|}
	{\sqrt{\rho(\overline{\bm U}_K^n)}}
	\right\},
	\qquad i=1,\ldots,d,
\end{equation}
where $\mathcal C(\bm U;\bm n)$ is the auxiliary splitting speed from \cite{WuShu2019} (also used in \eqref{eq:md-multistate-coefficient}). The nominal time step is given by
\begin{equation}\label{eq:num-time-step}
	\Delta t_n^{\rm CFL}
	:=
	\begin{cases}
		\displaystyle
		\frac{\mathrm{CFL}\,\Delta x_1}{\Lambda_1^n},
		& d=1,\\[3mm]
		\displaystyle
		\frac{\mathrm{CFL}}
		{\Lambda_1^n/\Delta x_1+\Lambda_2^n/\Delta x_2},
		& d=2,
	\end{cases}
	\qquad
	\mathrm{CFL}=0.1.
\end{equation}
To guarantee physical admissibility and multi-entropy stability under SSPMSV43, the accepted step $\Delta t_n$ must satisfy the SSP restriction \eqref{eq:time-ssp-step}, which is equivalent to requiring
\begin{equation*}
	\tau_{\ell,n}\le\tau_{\rm FE}(\bm U_h^{n-\ell}),
	\qquad \ell\in\{0,3\},
\end{equation*}
for the effective forward Euler steps in \eqref{eq:time-ms43-effective}.

\subsubsection{Entropy and dissipation diagnostics}\label{subsubsec:num-diagnostics}

We use three diagnostics to compare the \cmmAdd{two }wave-speed constructions.
The index \mbox{$r=1,\ldots,M$} identifies the entropy pair generated by
\mbox{$f_r$} in \mbox{\eqref{eq:entropy-parameterized-family}};
the label \mbox{$\xi$} in each plot denotes the corresponding parameter
\mbox{$\xi_r$} in \mbox{\eqref{eq:num-entropy-parameters}}.

\begin{enumerate}[itemsep=0.6\baselineskip]
	
	\item \emph{Total entropy increment.}
	\begin{equation}\label{eq:num-global-entropy}
		\Delta\mathcal E_{h,r}^n
		:=
		\mathcal E_{h,r}^n-\mathcal E_{h,r}^0,\qquad
		\mathcal E_{h,r}^n
		:=\mathcal E_{h,r}(\bm U_h^n)
		=\sum_{K\in\Th}|K|\,\mathcal S_{K,r}^n.
	\end{equation}
	For periodic boundaries, \eqref{eq:time-ms-global} bounds the
	total entropy by its values at earlier time levels, without implying
	monotonicity between consecutive steps.
	By induction, this yields $\Delta\mathcal E_{h,r}^n\le0$
	for all $n\ge0$, provided the startup values satisfy
	$\mathcal E_{h,r}^j\le\mathcal E_{h,r}^0$ for $j=1,2,3$.
	For nonperiodic boundaries,
	boundary entropy transport may change the sign of the increment.
	
	\item \emph{Global entropy residual.}
	For nonperiodic problems, we report
	\begin{equation}\label{eq:num-entropy-residual}
		\mathcal R_{h,r}^{n+1}
		:=
		\sum_{K\in\Th}
		|K|
		\left[
		\mathcal S_{K,r}^{n+1}
		-
		\mathcal B_{K,r}^{n,\mathrm{MS}}
		\right].
	\end{equation}
	The entropy bound \eqref{eq:num-sspms-budget} includes boundary entropy
	fluxes, so the cellwise entropy inequalities imply
	$\mathcal R_{h,r}^{n+1}\le0$.
	
	\item \emph{HLL jump dissipation.}
	\begin{equation}\label{eq:num-hll-jump}
		\mathcal D_{\mathrm{HLL}}^n
		:=
		\max_{K,\nu}
		\frac{-a_{K,\nu}^{-,n}a_{K,\nu}^{+,n}}
		{a_{K,\nu}^{+,n}-a_{K,\nu}^{-,n}}
		\left\|
		\bm U_{K,\nu}^{\mathrm{ext},n}
		-
		\bm U_{K,\nu}^{n}
		\right\|_2 .
	\end{equation}
	The HLL flux \eqref{eq:md-hll-flux} contains a dissipative term proportional to
	the state jump. This diagnostic reports the maximum magnitude of that term over all face quadrature points.
	
\end{enumerate}

\subsection{Accuracy for smooth solutions}\label{subsec:num-smooth}

\begin{numexample}[One-dimensional low-density advection]
	\label{ex:smooth1d}
	We first verify the accuracy of the one-dimensional schemes equipped with the W2S lifting operator of Section~\ref{sec:lifting}. The low-density advection problem \cite{WuShu2019} is solved on the periodic domain $[0,2\pi]$ with $\gamma=1.4$. The smooth exact solution is
	\begin{equation}\label{eq:num-smooth1d}
		(\rho,\bm v^\top,p,\bm B^\top)(x_1,t)
		=
		\bigl(
		1+0.99\sin(x_1-t),1,0,0,1,0.1,0,0
		\bigr).
	\end{equation}
	Computations are performed on $N$ uniform cells with WENO limiting disabled.
	
	Table~\ref{tab:rho_errors} lists the $L^1$ and $L^2$ density errors and numerical orders at $T=0.1$. Both methods achieve the optimal third-order convergence rate, confirming that the W2S lifting procedure does not degrade the designed accuracy in smooth regions.
	
	\begin{table}[!htbp]
		\centering
		\caption{Example~\ref{ex:smooth1d}: $L^1$ and $L^2$ density errors and convergence orders at $T=0.1$, with $N$ cells.}
		\label{tab:rho_errors}
		\small
		\setlength{\tabcolsep}{3pt}
		\begin{tabular}{cccccccccc}
			\toprule
			\multirow{2}{*}{$N$} & \multicolumn{2}{c}{Method 1 ($L^1$)} & \multicolumn{2}{c}{Method 1 ($L^2$)} & & \multicolumn{2}{c}{Method 2 ($L^1$)} & \multicolumn{2}{c}{Method 2 ($L^2$)} \\
			\cmidrule(lr){2-3} \cmidrule(lr){4-5} \cmidrule(lr){7-8} \cmidrule(lr){9-10}
			& Error & Order & Error & Order & & Error & Order & Error & Order \\
			\midrule
			40  & $7.74 \times 10^{-5}$ & ---  & $4.95 \times 10^{-5}$ & ---  & & $8.22 \times 10^{-5}$ & ---  & $4.61 \times 10^{-5}$ & ---  \\
			80  & $8.29 \times 10^{-6}$ & 3.22 & $5.85 \times 10^{-6}$ & 3.08 & & $1.31 \times 10^{-5}$ & 2.64 & $6.95 \times 10^{-6}$ & 2.73 \\
			160 & $1.08 \times 10^{-6}$ & 2.94 & $6.57 \times 10^{-7}$ & 3.15 & & $1.97 \times 10^{-6}$ & 2.74 & $1.05 \times 10^{-6}$ & 2.73 \\
			320 & $1.54 \times 10^{-7}$ & 2.82 & $8.12 \times 10^{-8}$ & 3.02 & & $2.00 \times 10^{-7}$ & 3.30 & $1.16 \times 10^{-7}$ & 3.18 \\
			640 & $2.25 \times 10^{-8}$ & 2.77 & $1.22 \times 10^{-8}$ & 2.74 & & $2.39 \times 10^{-8}$ & 3.07 & $1.18 \times 10^{-8}$ & 3.29 \\
			\bottomrule
		\end{tabular}
	\end{table}
\end{numexample}

\begin{numexample}[Two-dimensional smooth vortex]
	\label{ex:smooth2d}
	We next simulate the low-pressure MHD vortex \cite{WuShu2018,WuShu2019} on the periodic domain $\Omega=[-10,10]^2$. The vortex strength parameter $\mu=5.389489439$ yields an extreme minimum pressure of $p\approx 5.3\times10^{-12}$ at the vortex core. The exact solution is a rigid translation of the initial vortex with velocity $(1,1)^\top$. With $R^2=x_1^2+x_2^2$ and $\gamma=5/3$, the initial state is
	\begin{equation}\label{eq:num-vortex-initial}
		\begin{aligned}
			\rho&=1,\qquad
			\bm v=(1+\delta v_1,1+\delta v_2,0)^\top,\qquad
			\bm B=(\delta B_1,\delta B_2,0)^\top,\\
			(\delta v_1,\delta v_2)
			&=\frac{\mu}{\sqrt{2}\pi}e^{(1-R^2)/2}(-x_2,x_1),\\
			(\delta B_1,\delta B_2)
			&=\frac{\mu}{2\pi}e^{(1-R^2)/2}(-x_2,x_1),\qquad
			p=1-\frac{\mu^2(1+R^2)}{8\pi^2}e^{1-R^2}.
		\end{aligned}
	\end{equation}
	
	Table~\ref{tab:num-vortex} displays the $L^1$ errors in $v_1, v_2, B_1$, and $B_2$ for Method~1 at $T=0.05$ across uniform $N\times N$ meshes. All four components achieve optimal third-order accuracy. Method~2 produces nearly identical errors and orders and is omitted for brevity. Together with Example~\ref{ex:smooth1d}, these results confirm that the W2S lifting preserves the design order in both low-density and low-pressure regimes.
	
	\begin{table}[!htbp]
		\caption{Example~\ref{ex:smooth2d}: $L^1$ errors and convergence orders for Method~1 at $T=0.05$ on uniform $N\times N$ meshes.}
		\label{tab:num-vortex}
		\centering
		\small
		\setlength{\tabcolsep}{4pt}
		\begin{tabular}{ccccccccc}
			\toprule
			\multirow{2}{*}{$N$} & \multicolumn{2}{c}{$v_1$}
			& \multicolumn{2}{c}{$v_2$}
			& \multicolumn{2}{c}{$B_1$}
			& \multicolumn{2}{c}{$B_2$}\\
			\cmidrule(lr){2-3}
			\cmidrule(lr){4-5}
			\cmidrule(lr){6-7}
			\cmidrule(lr){8-9}
			& Error & Order
			& Error & Order
			& Error & Order
			& Error & Order\\
			\midrule
			40
			& $6.23\times10^{-2}$ & ---
			& $6.34\times10^{-2}$ & ---
			& $4.41\times10^{-2}$ & ---
			& $4.43\times10^{-2}$ & ---\\
			80
			& $5.95\times10^{-3}$ & 3.39
			& $6.04\times10^{-3}$ & 3.39
			& $4.13\times10^{-3}$ & 3.42
			& $4.15\times10^{-3}$ & 3.42\\
			160
			& $6.97\times10^{-4}$ & 3.09
			& $7.01\times10^{-4}$ & 3.11
			& $4.78\times10^{-4}$ & 3.11
			& $4.85\times10^{-4}$ & 3.10\\
			320
			& $8.32\times10^{-5}$ & 3.07
			& $8.34\times10^{-5}$ & 3.07
			& $5.51\times10^{-5}$ & 3.12
			& $5.55\times10^{-5}$ & 3.13\\
			\bottomrule
		\end{tabular}
	\end{table}
\end{numexample}

\subsection{One-dimensional Riemann problems}
\label{subsec:num-riemann}

We now investigate one-dimensional Riemann problems in extreme physical regimes. Computational domains, stopping times, and boundary conditions follow \cite{WuShu2019}. Numerical solutions are benchmarked against high-resolution reference solutions, and discrete entropy dissipation is monitored via the global entropy increment \eqref{eq:num-global-entropy}.

\begin{numexample}[Near-vacuum shock tube]
	\label{ex:vacuum}
	We consider the near-vacuum shock tube problem \cite{ChristliebEtAl2015,WuShu2019} on $[-0.5,0.5]$:
	\begin{equation}\label{eq:num-vacuum}
		(\rho,\bm v^\top,p,\bm B^\top)(x_1,0)=
		\begin{cases}
			(10^{-12},0,0,0,10^{-12},0,0,0), & x_1<0,\\
			(1,0,0,0,0.5,0,1,0), & x_1>0.
		\end{cases}
	\end{equation}
	The left state possesses an extremely low density and pressure of $10^{-12}$, which places a stringent demand on numerical admissibility. We compute up to $T=0.1$ on $200$ uniform cells and benchmark against a $5000$-cell numerical reference.
	
	Figure~\ref{fig:num-vacuum} displays the density and pressure profiles computed by Method~1, which agree closely with the reference. The global entropy increments remain strictly nonpositive throughout the simulation for both methods, confirming fully discrete entropy stability.
	
	\begin{figure}[tbp]
		\centering
		\includegraphics[width=0.32\linewidth]{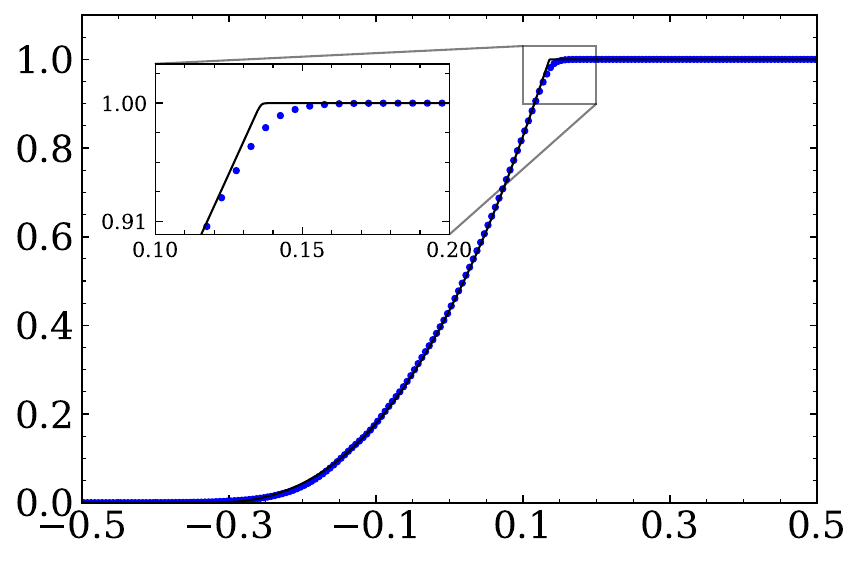}
		\includegraphics[width=0.32\linewidth]{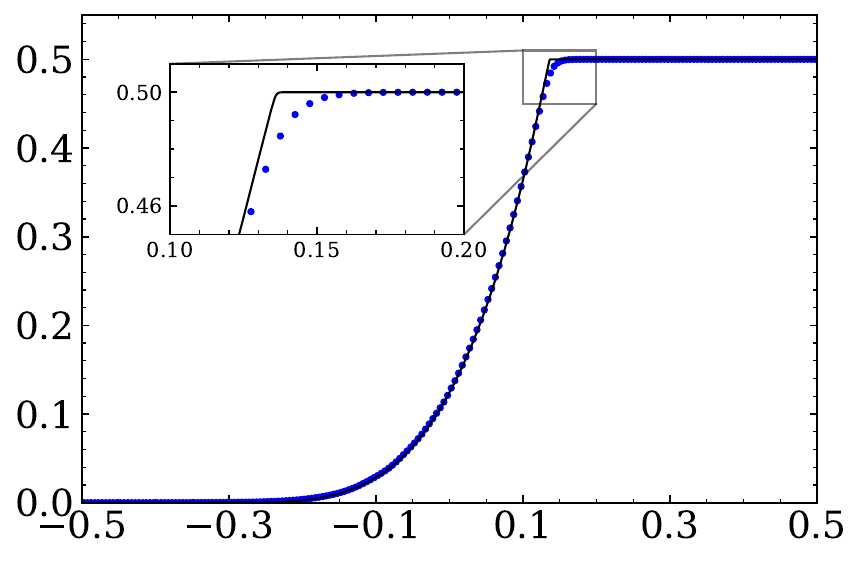}
		\includegraphics[width=0.30\linewidth]{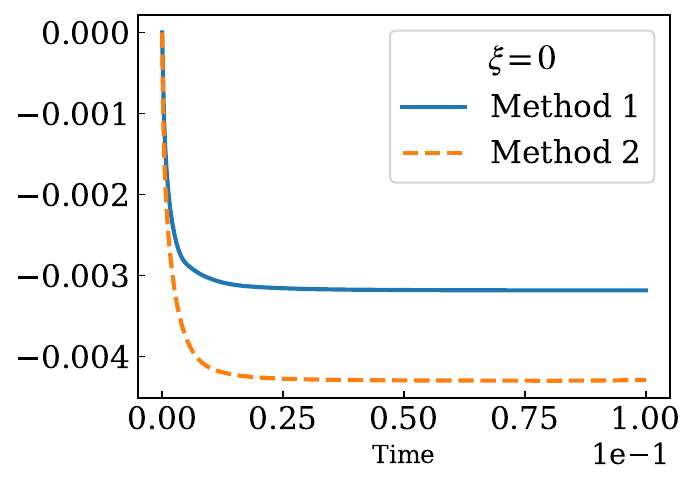}
		\caption{Example~\ref{ex:vacuum}: density $\rho$ (left) and gas pressure $p$ (center) computed by Method~1 at $T=0.1$; entropy increments \eqref{eq:num-global-entropy} for Methods~1 and~2 with $\xi=0$ (right).}
		\label{fig:num-vacuum}
	\end{figure}
\end{numexample}

\begin{numexample}[Strong-field LeBlanc problem]
	\label{ex:leblanc}
	We simulate the strongly magnetized LeBlanc problem \cite{WuShu2019} on $[-10,10]$:
	\begin{equation}\label{eq:num-leblanc}
		(\rho,\bm v^\top,p,\bm B^\top)(x_1,0)=
		\begin{cases}
			(2,0,0,0,10^9,0,5000,5000), & x_1<0,\\
			(10^{-3},0,0,0,1,0,5000,5000), & x_1>0,
		\end{cases}
	\end{equation}
	with $\gamma=1.4$. This setup features an initial pressure ratio of $10^9$ and a plasma beta as low as $\beta=2p/|\bm B|^2=4\times10^{-8}$ in the right state, with vanishing normal magnetic field $B_1\equiv 0$. We compute up to $T=3\times10^{-5}$ on $2000$ uniform cells against a $10{,}000$-cell reference solution.
	
	Figure~\ref{fig:num-leblanc} shows the computed density and magnetic pressure profiles for Method~1 alongside the fine-grid reference. The global entropy increments for both methods are virtually indistinguishable and remain strictly nonpositive, demonstrating robust entropy dissipation in the presence of very large pressure jumps and dominant magnetic fields.
	
	\begin{figure}[!htbp]
		\centering
		\includegraphics[width=0.33\linewidth]{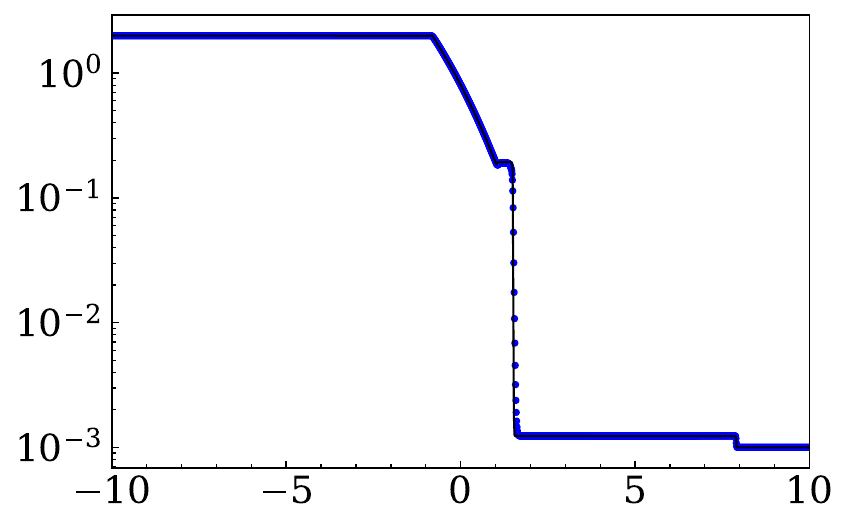}
		\includegraphics[width=0.31\linewidth]{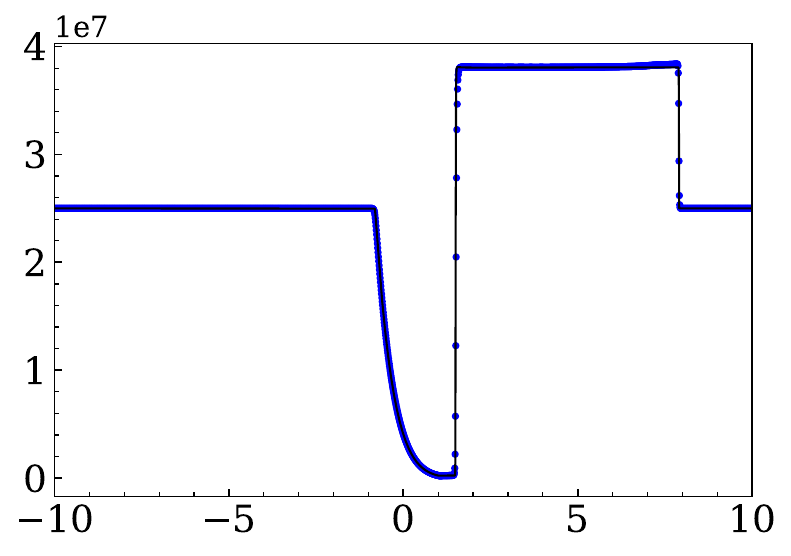}
		\includegraphics[width=0.29\linewidth]{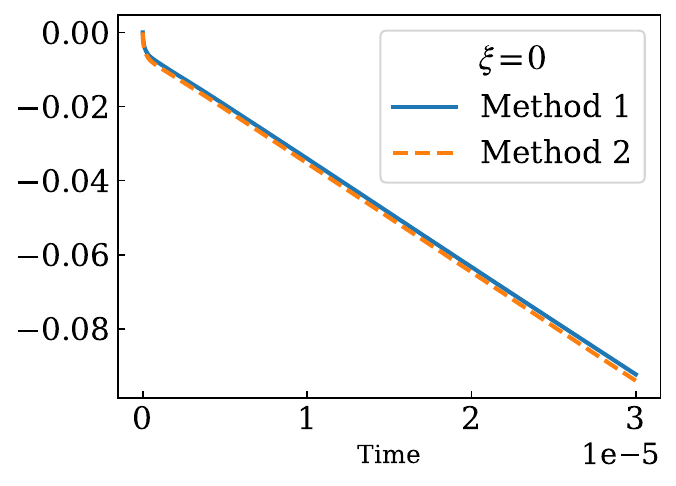}
		\caption{Example~\ref{ex:leblanc}: density $\rho$ (left) and magnetic pressure $|\bm B|^2/2$ (center) computed by Method~1 at $T=3\times10^{-5}$; entropy increments \eqref{eq:num-global-entropy} for Methods~1 and~2 with $\xi=0$ (right).}
		\label{fig:num-leblanc}
	\end{figure}
\end{numexample}

\subsection{Two-dimensional discontinuous solutions}\label{subsec:num-2d}

\begin{numexample}[Magnetized blast waves]
	\label{ex:blast}
	We compare the two wave-speed selection strategies on the magnetized blast-wave problem \cite{BalsaraSpicer1999,WuShu2019} on $\Omega=[-0.5,0.5]^2$. The initial state is
	\[
	\rho=1,\qquad
	\bm v=\bm 0,\qquad
	\bm B=(B_1,0,0)^\top,
	\]
	with the gas pressure initialized as
	\[
	p=
	\begin{cases}
		p_{\rm in}, & x_1^2+x_2^2\le 0.1^2,\\
		0.1,        & x_1^2+x_2^2>0.1^2.
	\end{cases}
	\]
	We set $\gamma=1.4$, enforce outflow boundary conditions following \cite{WuShu2019}, and discretize the domain with a uniform $320\times320$ rectangular mesh. Table~\ref{tab:num-blast} summarizes the parameters for two configurations, where $\beta_{\rm amb}=2p/|\bm B|^2$ denotes the ambient plasma beta.
	
	\begin{table}[!htbp]
		\caption{Example~\ref{ex:blast}: parameters of the two configurations.}
		\label{tab:num-blast}
		\centering
		\small
		\begin{tabular}{lcccc}
			\toprule
			Case & $B_1$ & $p_{\rm in}$ & $\beta_{\rm amb}$ & $T$\\
			\midrule
			I& $100/\sqrt{4\pi}$
			& $10^3$
			& $2.51\times10^{-4}$
			& $0.01$\\
			II
			& $1000/\sqrt{4\pi}$
			& $10^4$
			& $2.51\times10^{-6}$
			& $0.001$\\
			\bottomrule
		\end{tabular}
	\end{table}
	
	For Case~I ($B_1=100/\sqrt{4\pi}$), Figure~\ref{fig:num-blast100} shows that both methods produce nearly identical density profiles with nonpositive entropy increments, while Method~1 yields smaller HLL jump dissipation \eqref{eq:num-hll-jump} over most of the simulation. For Case~II ($B_1=1000/\sqrt{4\pi}$, $\beta_{\rm amb}=2.51\times10^{-6}$), Figure~\ref{fig:num-blast1000} displays the density and magnetic pressure for Method~1; the entropy increments again remain strictly nonpositive for both methods.
	
	\begin{figure}[!htbp]
		\centering
		\includegraphics[width=0.47\linewidth]{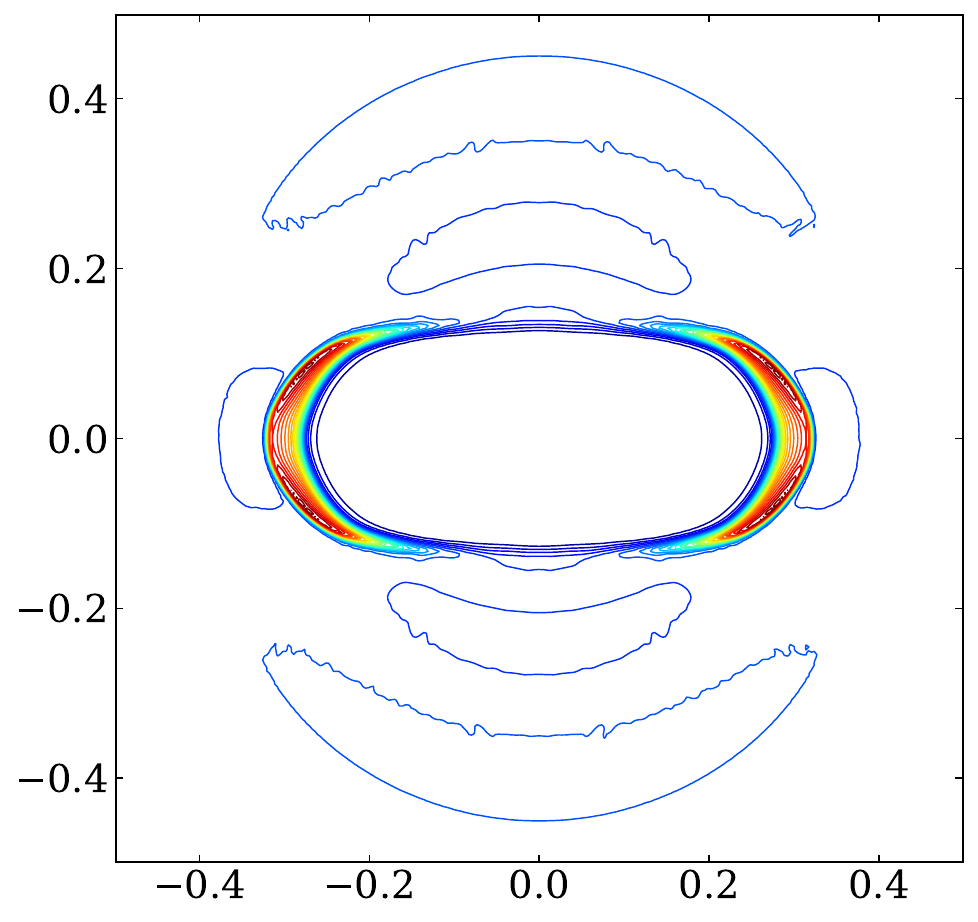}\hfill
		\includegraphics[width=0.47\linewidth]{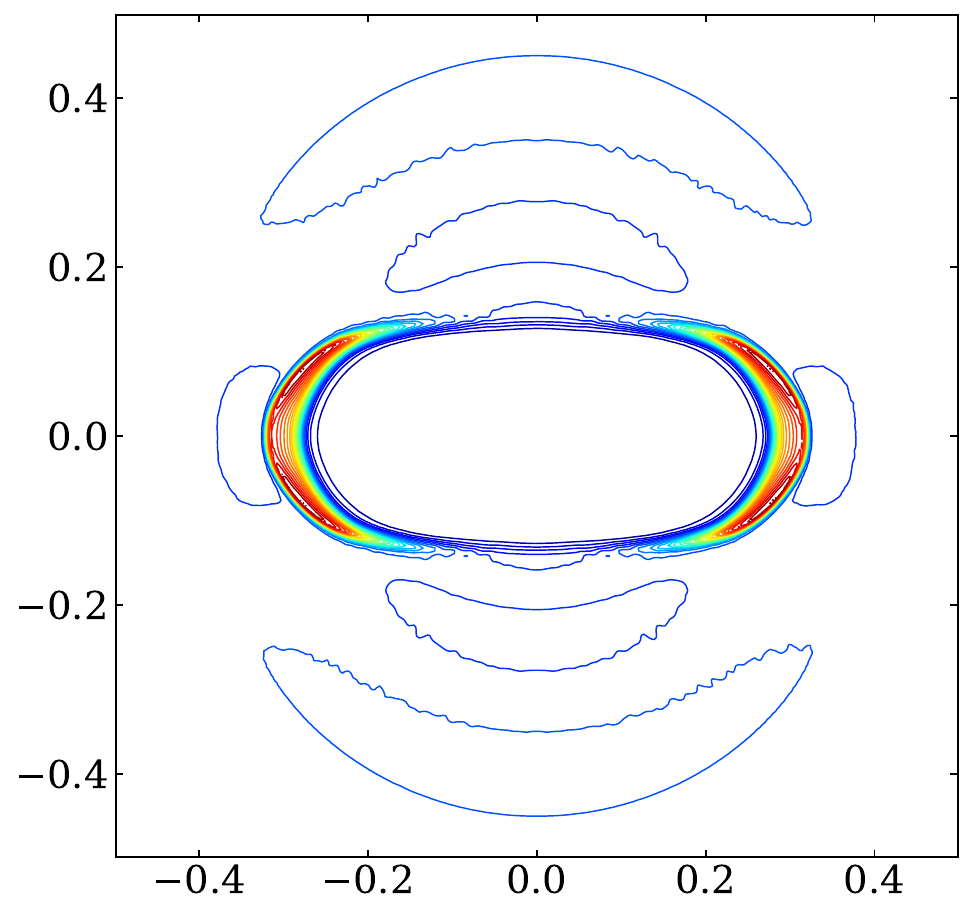}
		
		\medskip
		
		\includegraphics[width=0.47\linewidth]{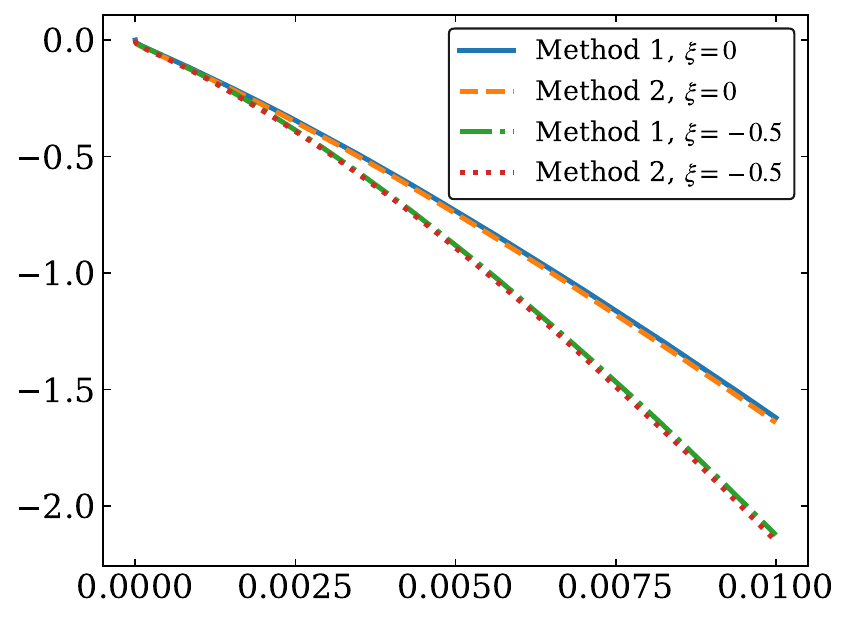}\hfill
		\includegraphics[width=0.47\linewidth]{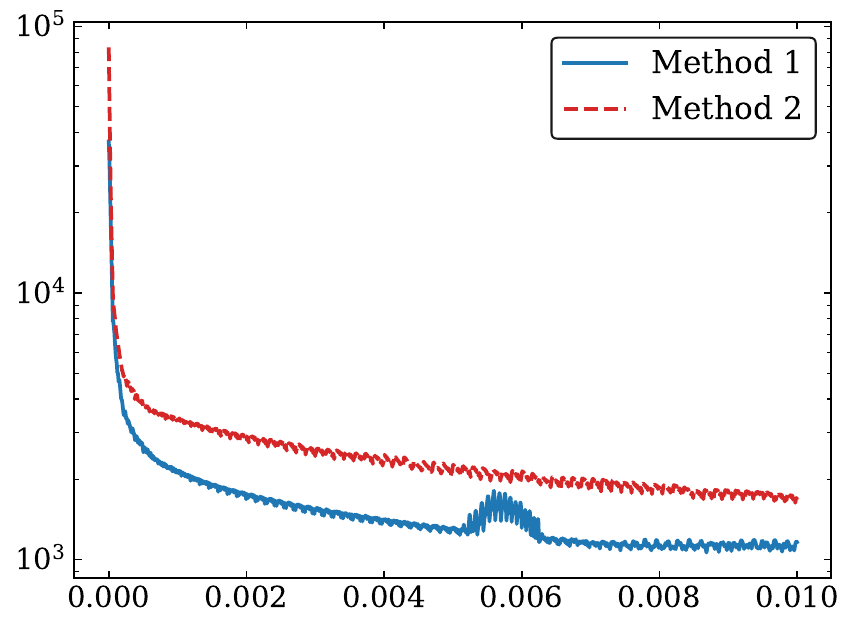}
		
		\caption{Example~\ref{ex:blast}, Case~I ($B_1=100/\sqrt{4\pi}$): density $\rho$ at $T=0.01$ computed by Method~1 (top left) and Method~2 (top right); entropy increments \eqref{eq:num-global-entropy} with $\xi=0,-0.5$ (bottom left) and maximum HLL jump dissipation \eqref{eq:num-hll-jump} (bottom right), for Methods~1 and~2.}
		\label{fig:num-blast100}
	\end{figure}
	
	\begin{figure}[!htbp]
		\centering
		\includegraphics[width=0.31\linewidth]{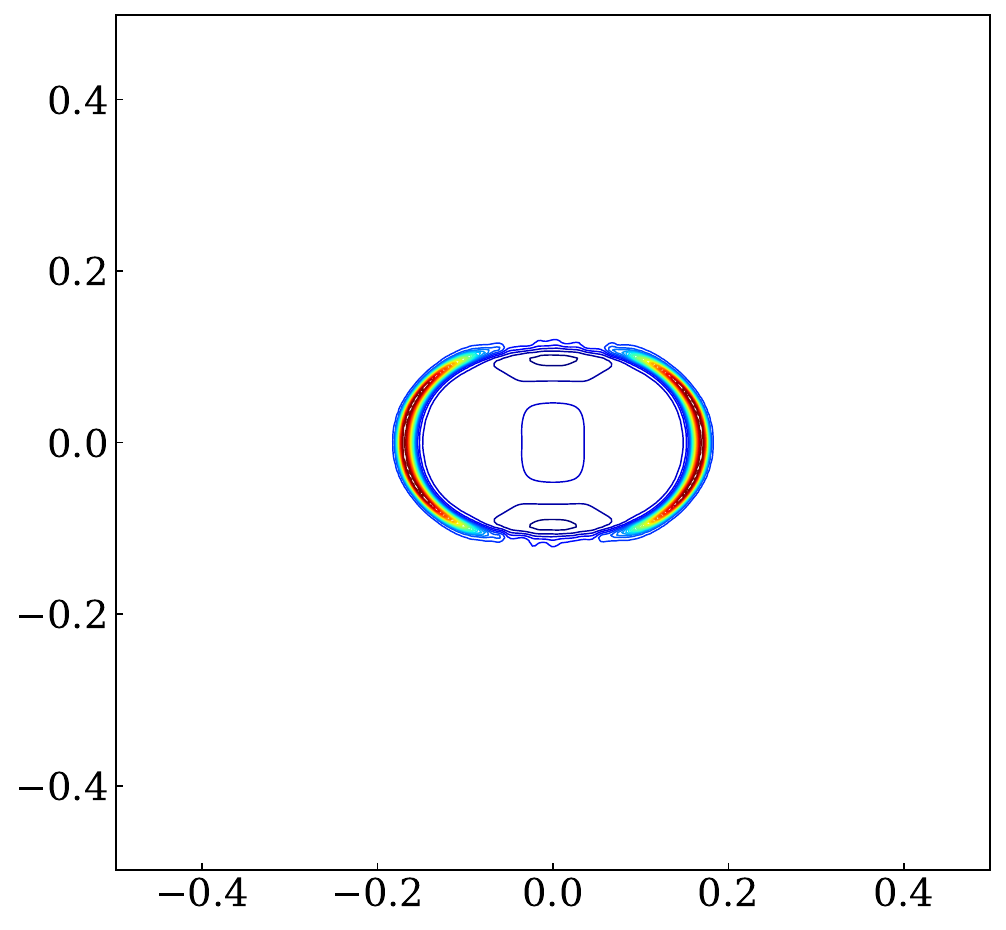}\hfill
		\includegraphics[width=0.31\linewidth]{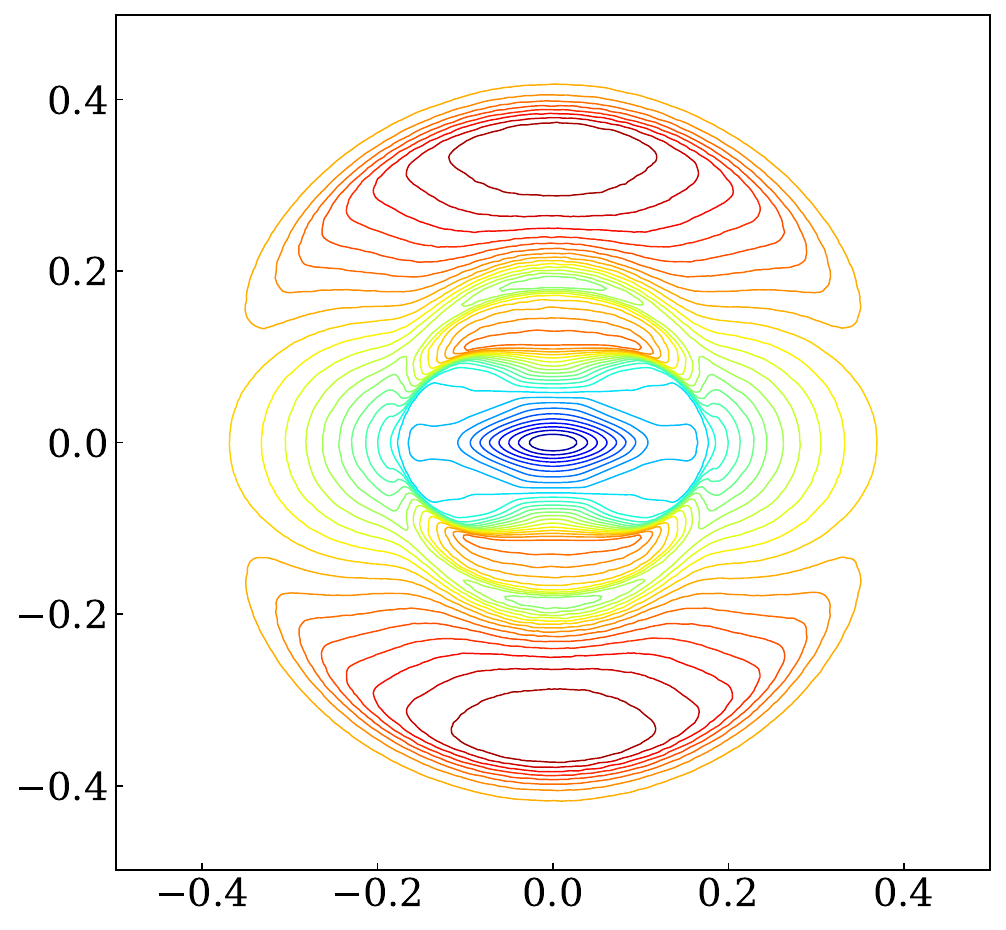}\hfill
		\includegraphics[width=0.37\linewidth]{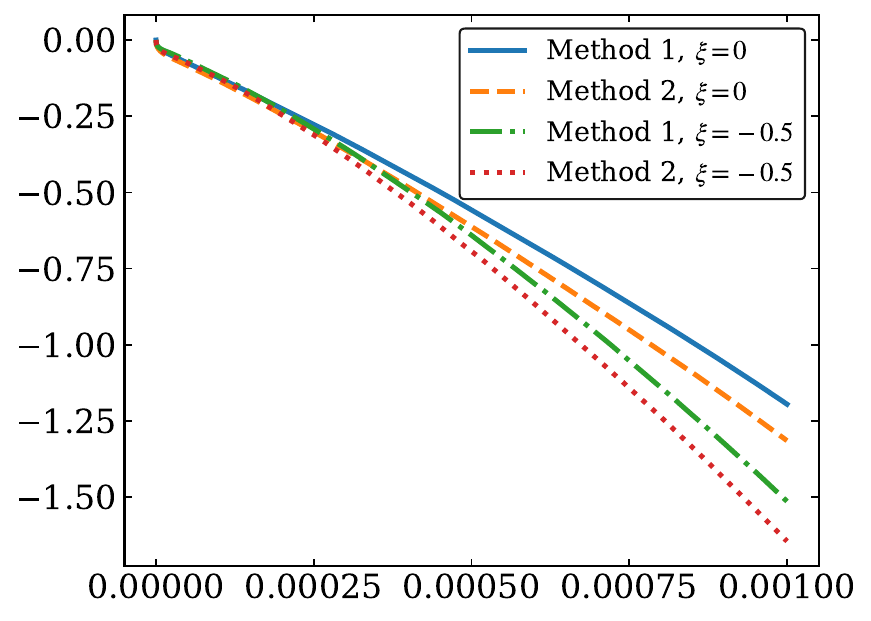}
		\caption{Example~\ref{ex:blast}, Case~II ($B_1=1000/\sqrt{4\pi}$): density $\rho$ (left) and magnetic pressure $|\bm B|^2/2$ (center) computed by Method~1 at $T=0.001$; entropy increments \eqref{eq:num-global-entropy} for Methods~1 and~2 with $\xi=0,-0.5$ (right).}
		\label{fig:num-blast1000}
	\end{figure}
\end{numexample}

\begin{numexample}[MHD rotor]
	\label{ex:rotor}
	The MHD rotor problem \cite{BalsaraSpicer1999} on $\Omega=[0,1]^2$ models a dense spinning disk driving torsional waves into a uniformly magnetized ambient medium. Define the radial coordinates
	\[
	r=\sqrt{(x_1-0.5)^2+(x_2-0.5)^2},
	\qquad
	r_1=0.1,\qquad
	r_2=0.115,
	\qquad
	\phi=\frac{r_2-r}{r_2-r_1}.
	\]
	The initial data are
	\begin{equation}\label{eq:num-rotor}
		(\rho,\bm v^\top,p,\bm B^\top)=
		\begin{cases}
			\displaystyle
			\left(
			10,
			-\frac{x_2-0.5}{r_1},\frac{x_1-0.5}{r_1},0,
			0.5,
			\frac{2.5}{\sqrt{4\pi}},0,0
			\right),
			& r\le r_1,\\[2ex]
			\displaystyle
			\left(
			1+9\phi,
			-\frac{\phi(x_2-0.5)}{r},
			\frac{\phi(x_1-0.5)}{r},0,
			0.5,
			\frac{2.5}{\sqrt{4\pi}},0,0
			\right),
			& r_1<r\le r_2,\\[2ex]
			\displaystyle
			\left(
			1,0,0,0,
			0.5,
			\frac{2.5}{\sqrt{4\pi}},0,0
			\right),
			& r>r_2 .
		\end{cases}
	\end{equation}
	Following \cite[Example~4]{ChenJinQiuZhao2026}, we impose outflow boundary conditions, employ a uniform $400\times400$ rectangular mesh, and compute up to $T=0.295$.
	
	Figure~\ref{fig:num-rotor} shows the density and sonic Mach number $\mathrm{Ma}=|\bm v|/\sqrt{\gamma p/\rho}$ computed by Method~1, sharply resolving the central rotor and the outgoing wave patterns. The global entropy increments for both entropies remain nonpositive throughout the evolution.
	
	\begin{figure}[!htbp]
		\centering
		\includegraphics[width=0.30\linewidth]{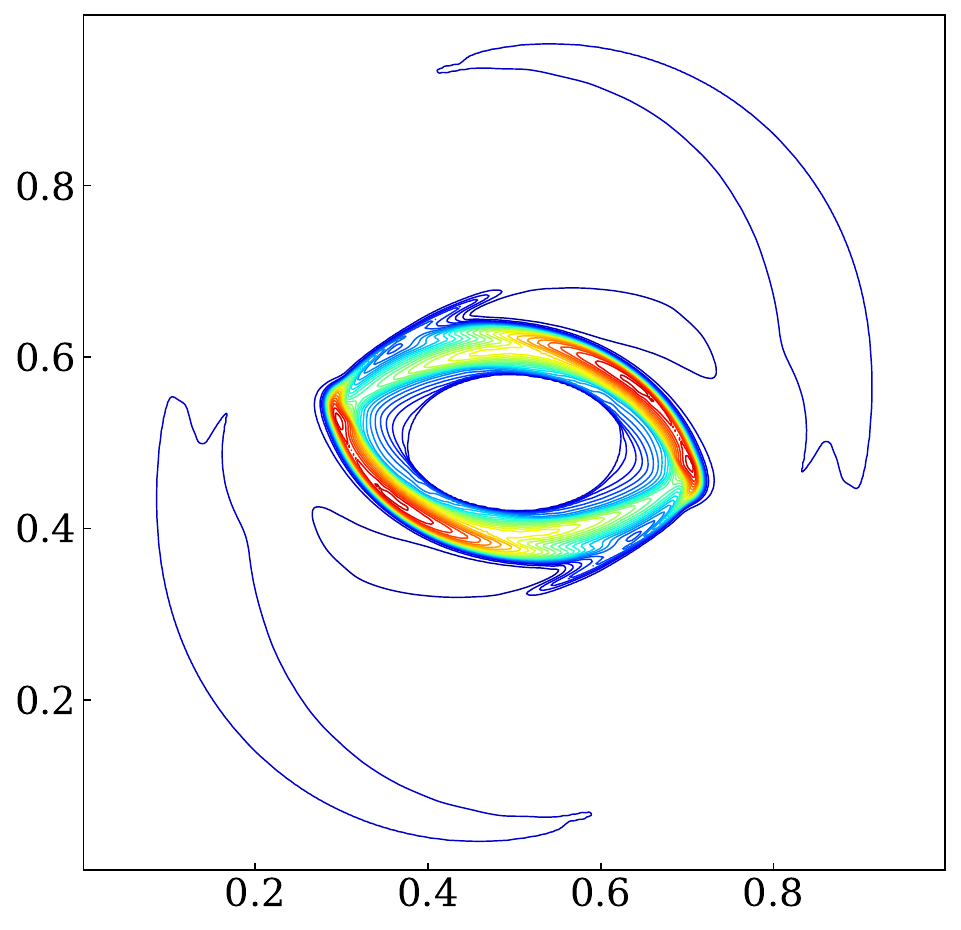}\hfill
		\includegraphics[width=0.30\linewidth]{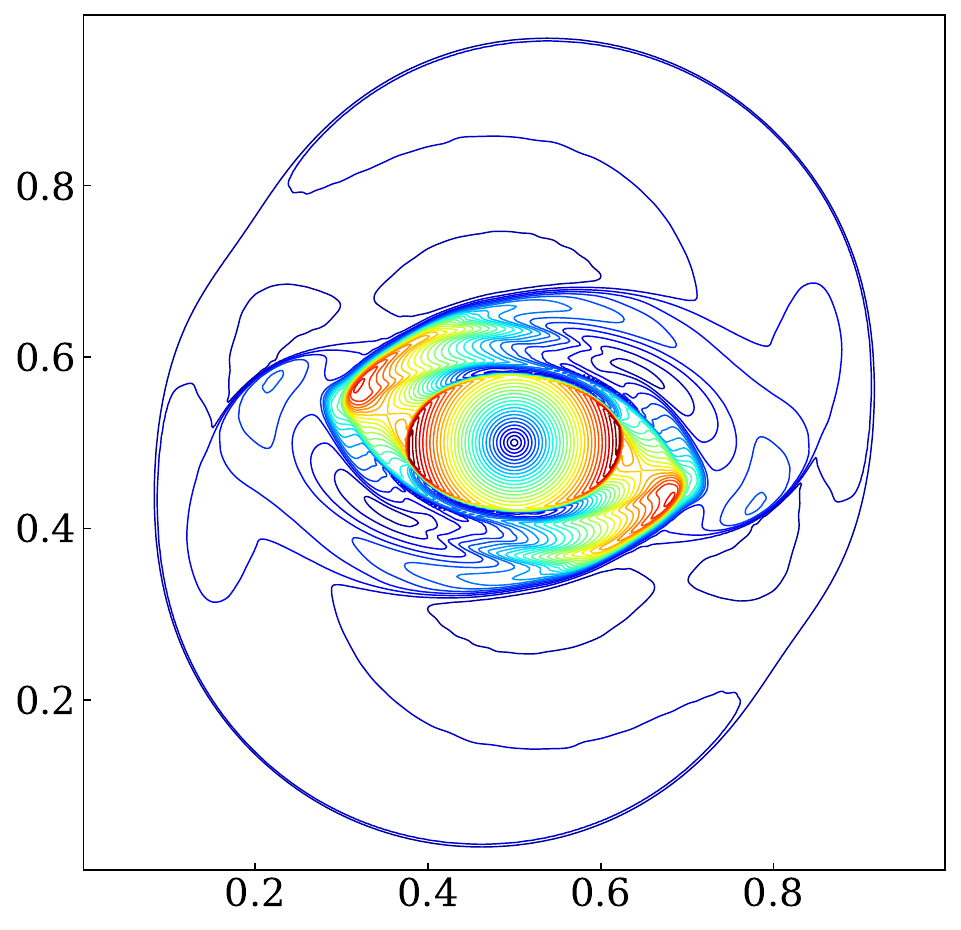}\hfill
		\includegraphics[width=0.37\linewidth]{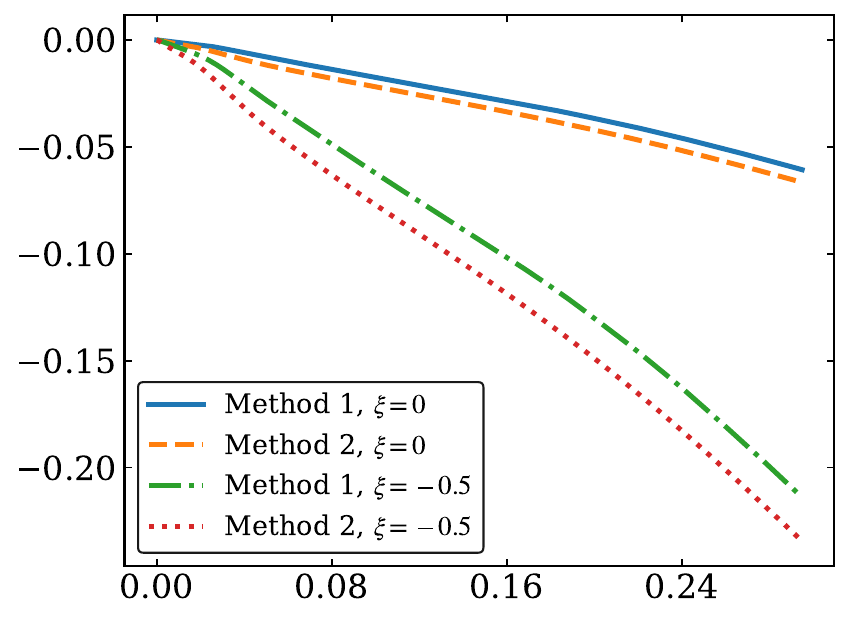}
		\caption{Example~\ref{ex:rotor}: density $\rho$ (left) and sonic Mach number (center) computed by Method~1 at $T=0.295$; entropy increments \eqref{eq:num-global-entropy} for Methods~1 and~2 with $\xi=0,-0.5$ (right).}
		\label{fig:num-rotor}
	\end{figure}
\end{numexample}

\begin{numexample}[Orszag--Tang vortex]
	\label{ex:ot}
	The Orszag--Tang vortex system \cite{JiangWu1999} on the periodic domain $\Omega=[0,2\pi]^2$ tests the transition to MHD turbulence and complex shock interactions. The initial state is
	\begin{equation}\label{eq:num-ot}
		(\rho,\bm v^\top,p,\bm B^\top)
		=
		\bigl(
		\gamma^2,\,
		-\sin x_2,\sin x_1,0,\,
		\gamma,\,
		-\sin x_2,\sin(2x_1),0
		\bigr).
	\end{equation}
	The solution is integrated to $T=4$ on a uniform $200\times200$ rectangular mesh.
	
	Figure~\ref{fig:num-ot} depicts the density and sonic Mach number at $T=4$ computed by Method~1, alongside the entropy history. The global entropy increments remain strictly nonpositive throughout the simulation for both methods, in agreement with the theoretical bound \eqref{eq:time-ms-global}.
	
	\begin{figure}[!htbp]
		\centering
		\includegraphics[width=0.29\linewidth]{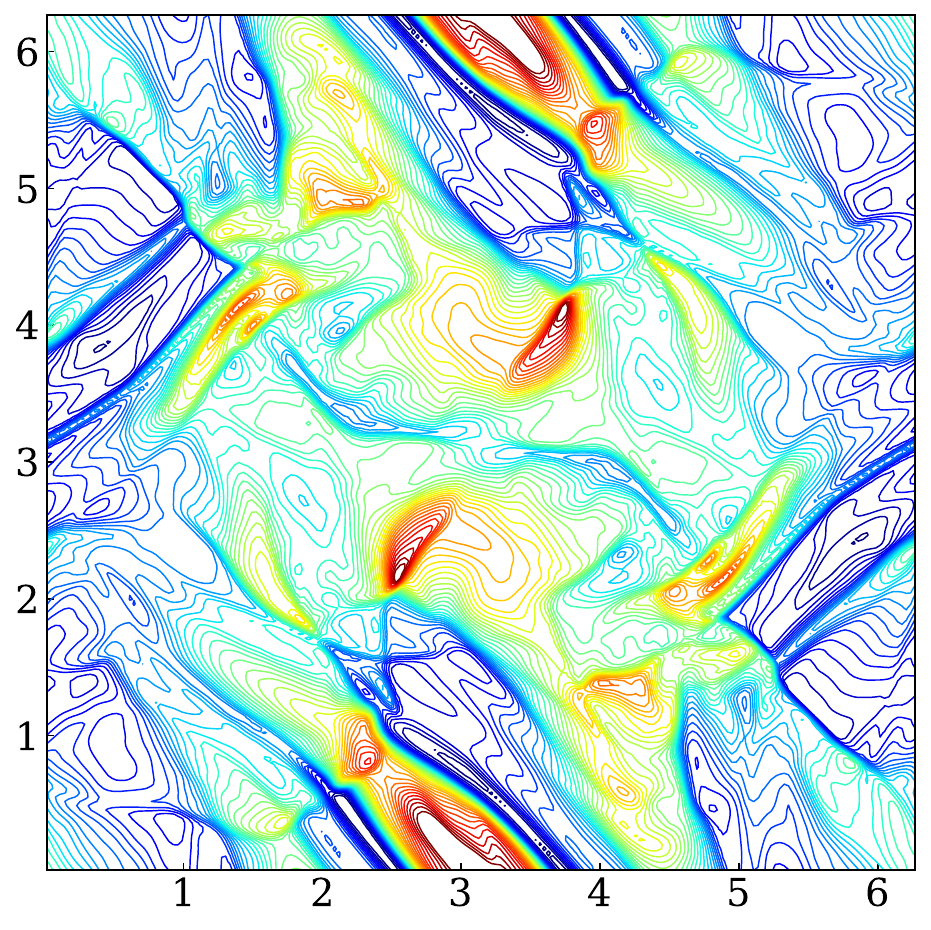}\hfill
		\includegraphics[width=0.29\linewidth]{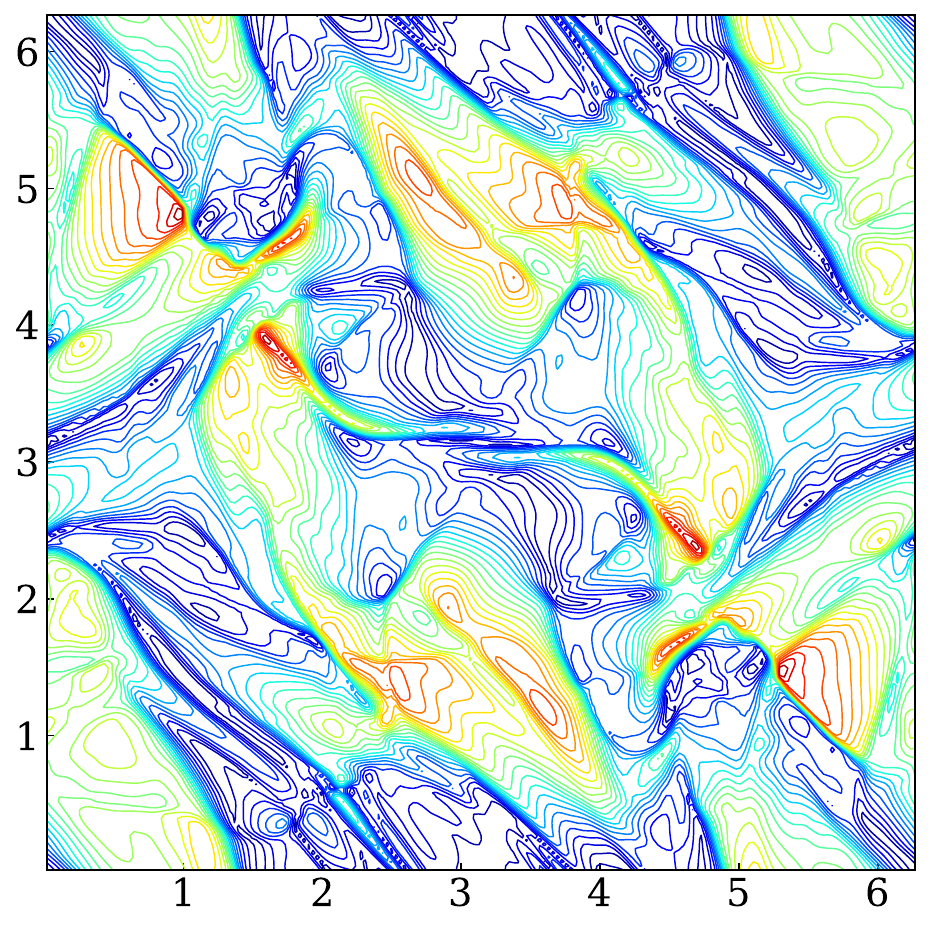}\hfill
		\includegraphics[width=0.39\linewidth]{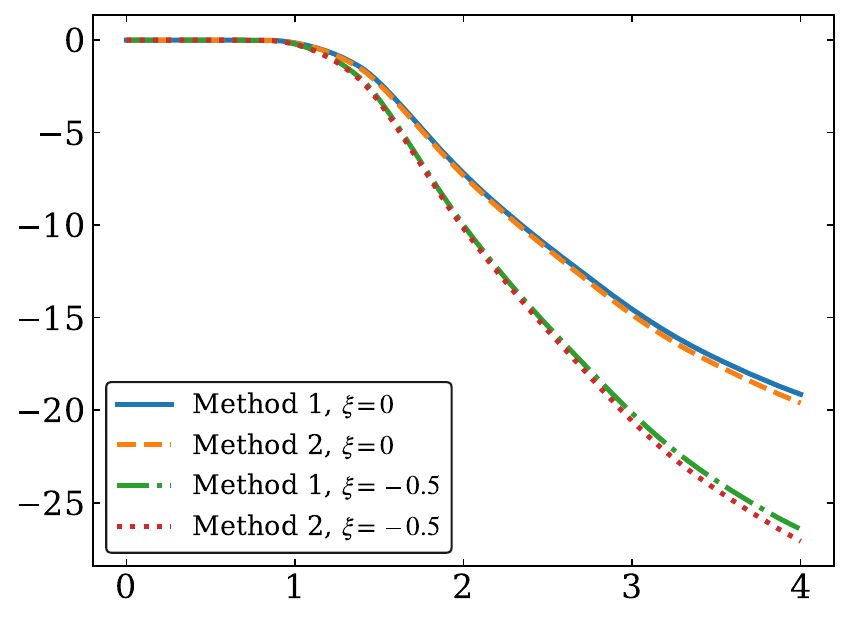}
		\caption{Example~\ref{ex:ot}: density $\rho$ (left) and sonic Mach number (center) computed by Method~1 at $T=4$; entropy increments \eqref{eq:num-global-entropy} for Methods~1 and~2 with $\xi=0,-0.5$ (right).}
		\label{fig:num-ot}
	\end{figure}
\end{numexample}

\begin{numexample}[Yee--Sj\"ogreen four-state Riemann problem]
	\label{ex:ys}
	The Yee--Sj\"ogreen four-state 2D Riemann problem \cite{YeeSjogreen2005,YueWuShu2026} is solved on $\Omega=[-1,1]^2$ with conservative initial states \eqref{eq:intro-state}:
	\begin{equation}\label{eq:num-ys}
		\bm U(x_1,x_2,0)=
		\begin{cases}
			(0.9308,1.4557,-0.4633,0.0575,0.3501,0.9830,0.3050,5.0838)^\top,
			& x_1,x_2>0,\\
			(1.0304,1.5774,-1.0455,-0.1016,0.3501,0.5078,0.1576,5.7813)^\top,
			& x_1<0<x_2,\\
			(1.0000,1.7500,-1.0000,0,0.5642,0.5078,0.2539,6.0000)^\top,
			& x_1,x_2<0,\\
			(1.8887,0.2334,-1.7422,0.0733,0.5642,0.9830,0.4915,12.999)^\top,
			& x_2<0<x_1 .
		\end{cases}
	\end{equation}
	We employ a uniform $256\times256$ mesh aligned with the initial discontinuity lines $x_1=0$ and $x_2=0$. Because the normal magnetic components are continuous across these lines ($B_1$ across $x_1=0$ and $B_2$ across $x_2=0$), the initial jump $\delta B_{\bm n}$ in \eqref{eq:md-corrected-hll-states} vanishes identically, reducing the corrected states $\bm T^\pm$ to the standard HLL state $\bm H$. With zero-gradient boundary extrapolation, the solution is advanced to $T=0.2$.
	
	Figure~\ref{fig:num-ys} shows the density and $B_1$ contours computed by Method~1. The entropy increments for both methods are virtually identical and remain strictly nonpositive at all times.
	
	\begin{figure}[!htbp]
		\centering
		\includegraphics[width=0.29\linewidth]{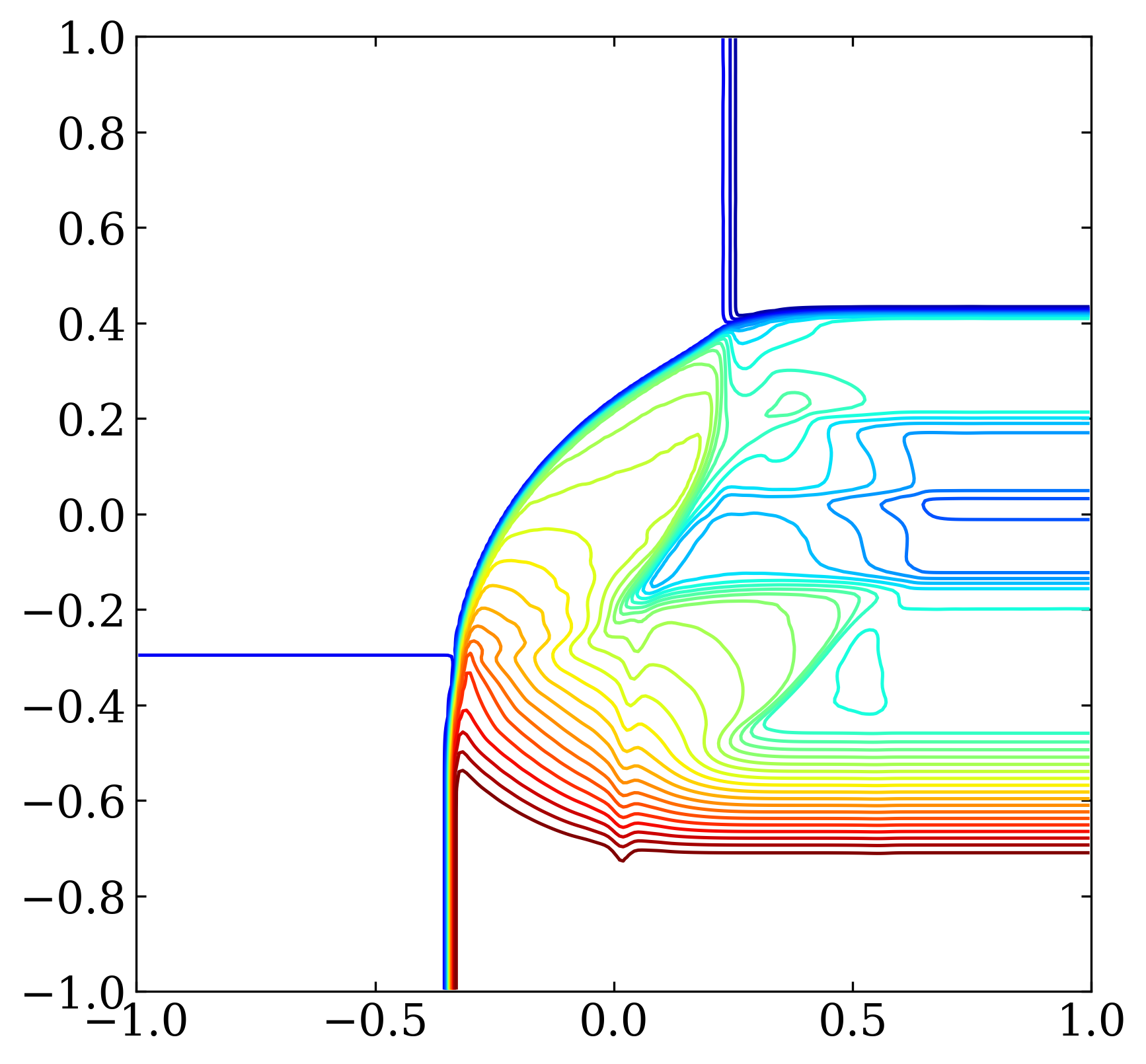}\hfill
		\includegraphics[width=0.29\linewidth]{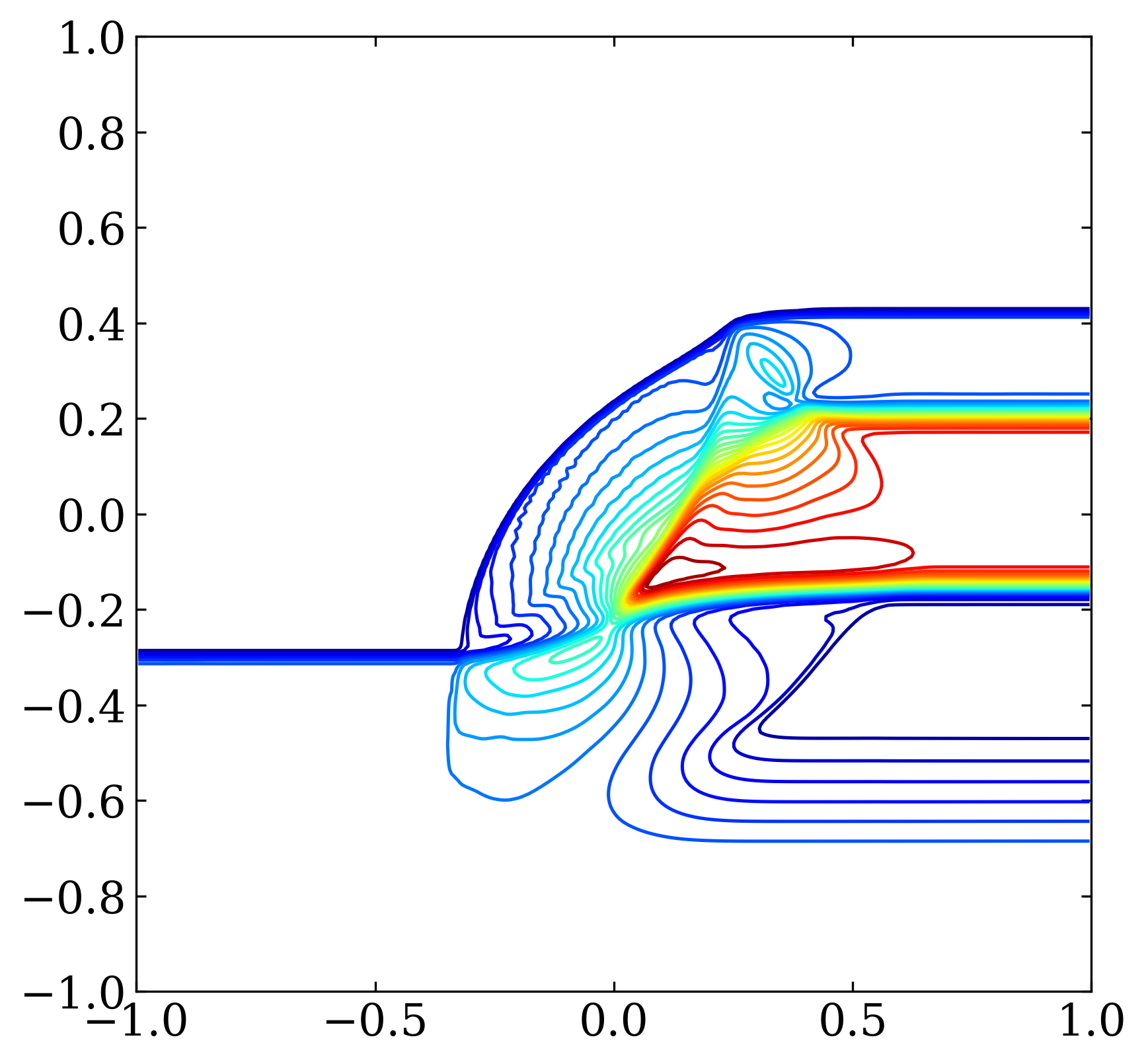}\hfill
		\includegraphics[width=0.39\linewidth]{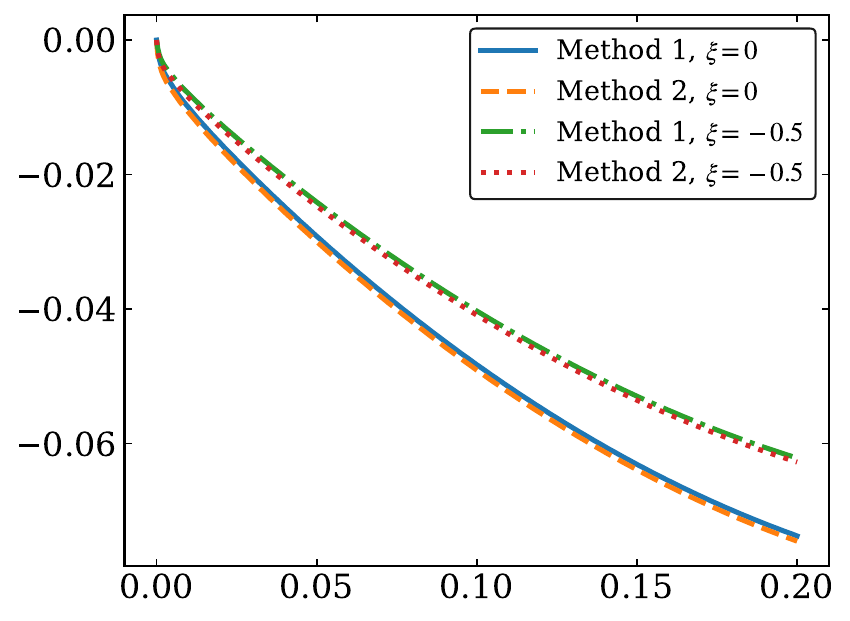}
		\caption{Example~\ref{ex:ys}: density $\rho$ (left) and magnetic-field component $B_1$ (center) computed by Method~1 at $T=0.2$; entropy increments \eqref{eq:num-global-entropy} for Methods~1 and~2 with $\xi=0,-0.5$ (right).}
		\label{fig:num-ys}
	\end{figure}
\end{numexample}

\begin{numexample}[Shock--cloud interaction]
	\label{ex:sc}
	The shock--cloud interaction problem \cite{DaiWoodward1998,Toth2000,WuShu2019} on $\Omega=[0,1]^2$ models a planar shock impacting a dense magnetic cloud. The background states are given by
	\begin{equation}\label{eq:num-sc}
		(\rho,\bm v^\top,p,\bm B^\top)(x_1,x_2,0)=
		\begin{cases}
			(3.86859, 0,0,0,167.345,
			0,2.1826182,-2.1826182), & x_1<0.6,\\
			(1,-11.2536,0,0,1,
			0,0.56418958,0.56418958), & x_1>0.6.
		\end{cases}
	\end{equation}
	A high-density circular cloud ($\rho=10$, radius $0.15$) centered at $(0.8,0.5)$ is embedded in the right pre-shock state. Supersonic inflow is imposed on the right boundary, with outflow on the other boundaries \cite{WuShu2019}. We compute to $T=0.06$ on a uniform $400\times400$ mesh.
	
	Figure~\ref{fig:num-sc} displays the density and pressure contours computed by Method~1, sharply resolving the bow shock, cloud deformation, and downstream vortex wake. The global entropy residual \eqref{eq:num-entropy-residual}, which incorporates boundary entropy fluxes, remains nonpositive for all three entropies under both methods, verifying the global stability bound of Theorem~\ref{thm:time-ssp-multistep}.
	
	\begin{figure}[!htbp]
		\centering
		\includegraphics[width=0.47\linewidth]{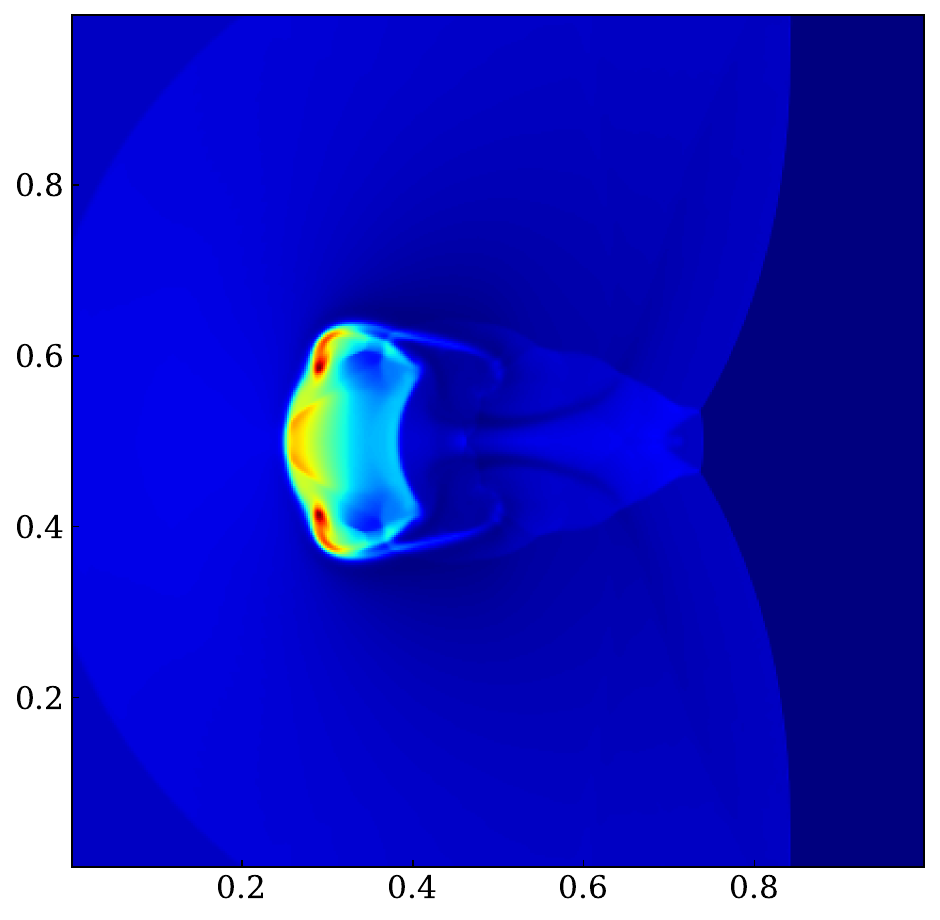}\hfill
		\includegraphics[width=0.47\linewidth]{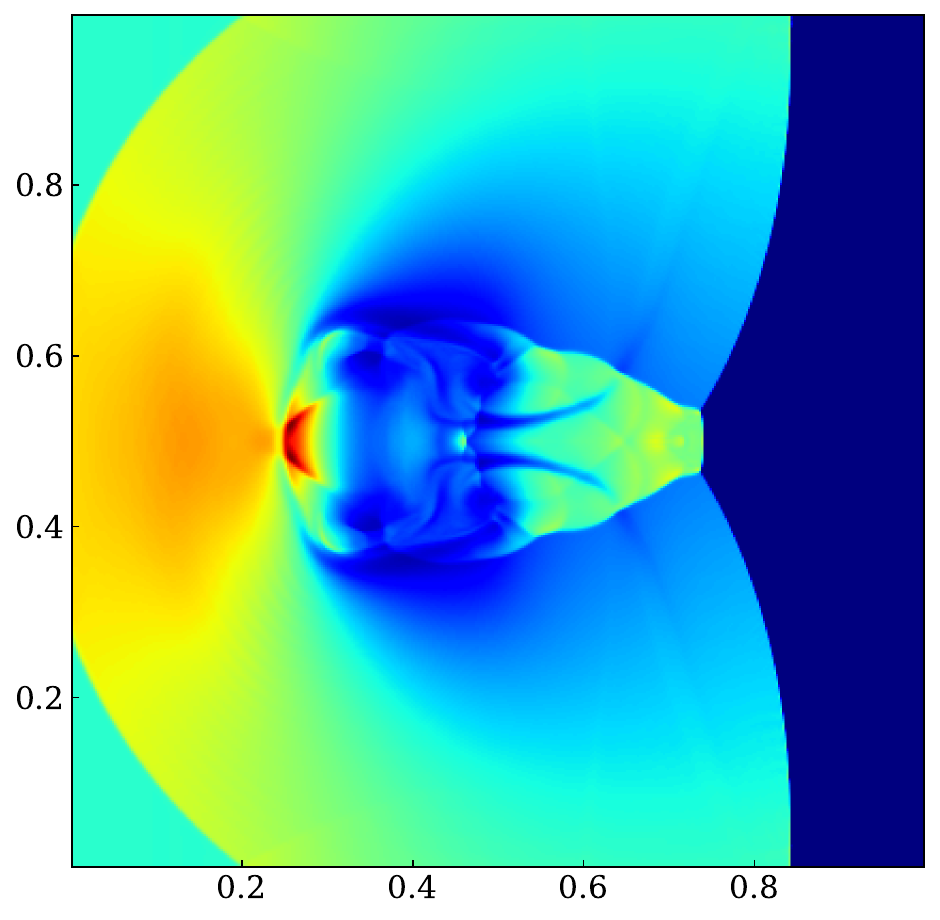}
		
		\medskip
		
		\includegraphics[width=0.47\linewidth]{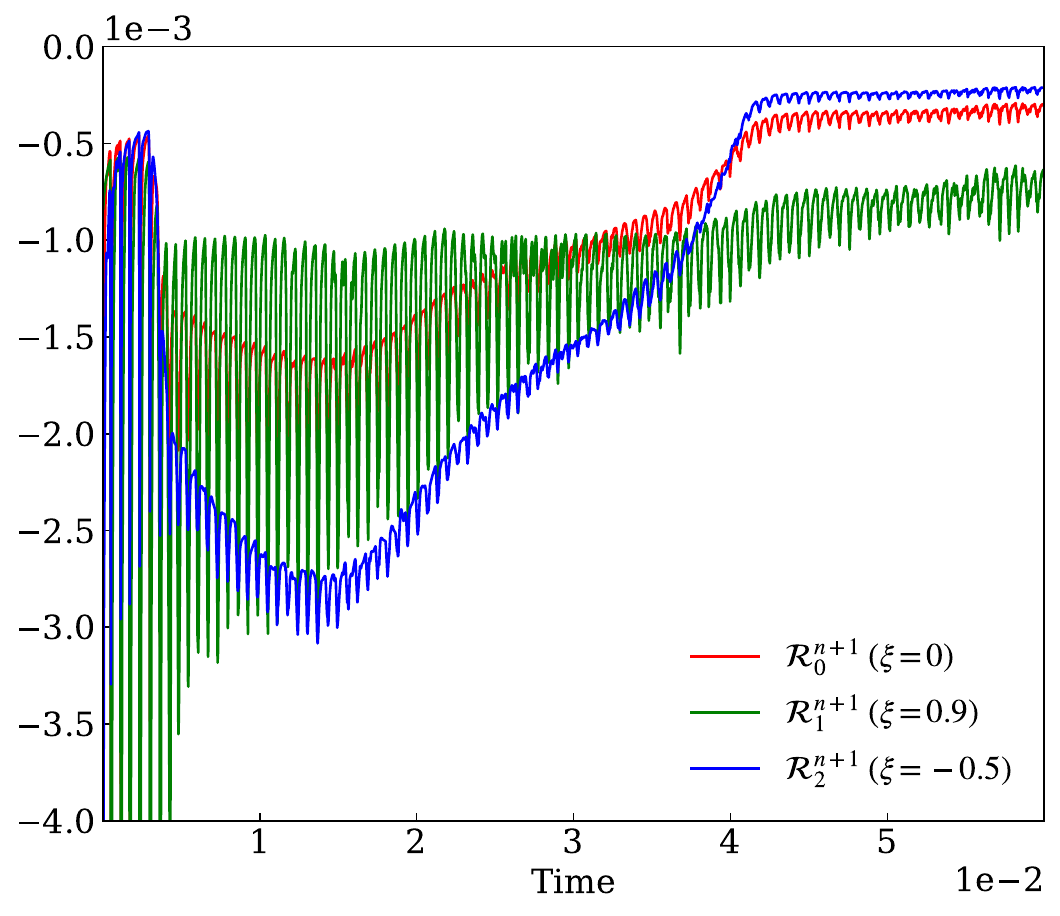}\hfill
		\includegraphics[width=0.47\linewidth]{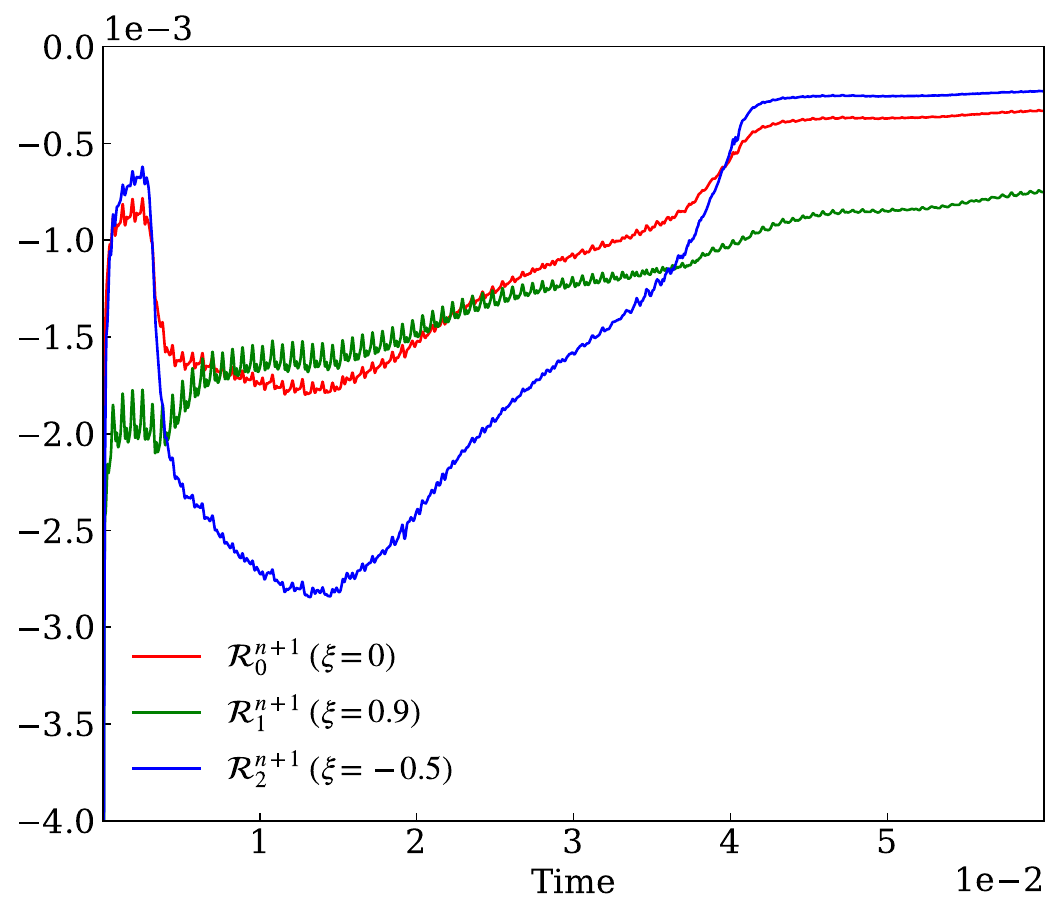}
		
		\caption{Example~\ref{ex:sc}: density $\rho$ (top left) and gas pressure $p$ (top right) computed by Method~1 at $T=0.06$; fully discrete global entropy residuals \eqref{eq:num-entropy-residual} for Method~1 (bottom left) and Method~2 (bottom right), with $\xi=0,0.9,-0.5$. The curves labeled $0,1,2$ correspond to $r=1,2,3$ in \eqref{eq:num-entropy-parameters}.}
		\label{fig:num-sc}
	\end{figure}
\end{numexample}

\begin{numexample}[Magnetized jets]\label{ex:jet}
	We simulate the Mach~800 astrophysical jet problem \cite{WuShu2019} under three magnetic field strengths. The ambient medium is initially at rest with
	\[
	(\rho,\bm v^\top,p,\bm B^\top)
	=
	(0.1\gamma,0,0,0,1,0,B_2,0),
	\qquad \gamma=1.4 .
	\]
	A dense supersonic jet is injected upward through the nozzle $\{(x_1,0):0\le x_1<0.05\}$ with state
	\[
	(\rho,\bm v^\top,p,\bm B^\top)
	=
	(\gamma,0,800,0,1,0,B_2,0).
	\]
	Exploiting symmetry, we compute on the half-domain $\Omega=[0,0.5]\times[0,1.5]$ with a reflecting wall at $x_1=0$ and outflow on other non-inflow boundaries, using a uniform $200\times600$ mesh up to $T=0.002$. Full-domain density profiles are mirrored about $x_1=0$. The transverse fields $B_2=\sqrt{200}$, $\sqrt{2000}$, and $\sqrt{20000}$ correspond to plasma betas $\beta=10^{-2}$, $10^{-3}$, and $10^{-4}$, respectively.
	
	For $\beta=10^{-2}$, Figure~\ref{fig:num-jet1} compares both methods, with Method~1 resolving slightly sharper shear structures and generating less HLL jump dissipation \eqref{eq:num-hll-jump}. Figure~\ref{fig:num-jet23} presents the strongly magnetized cases $\beta=10^{-3}$ and $10^{-4}$, illustrating tighter magnetic confinement along the jet axis. For all three field strengths, the global entropy residuals for both methods remain nonpositive and well controlled throughout the hypersonic evolution.
	
	\begin{figure}[!htbp]
		\centering
		\includegraphics[width=0.32\linewidth]{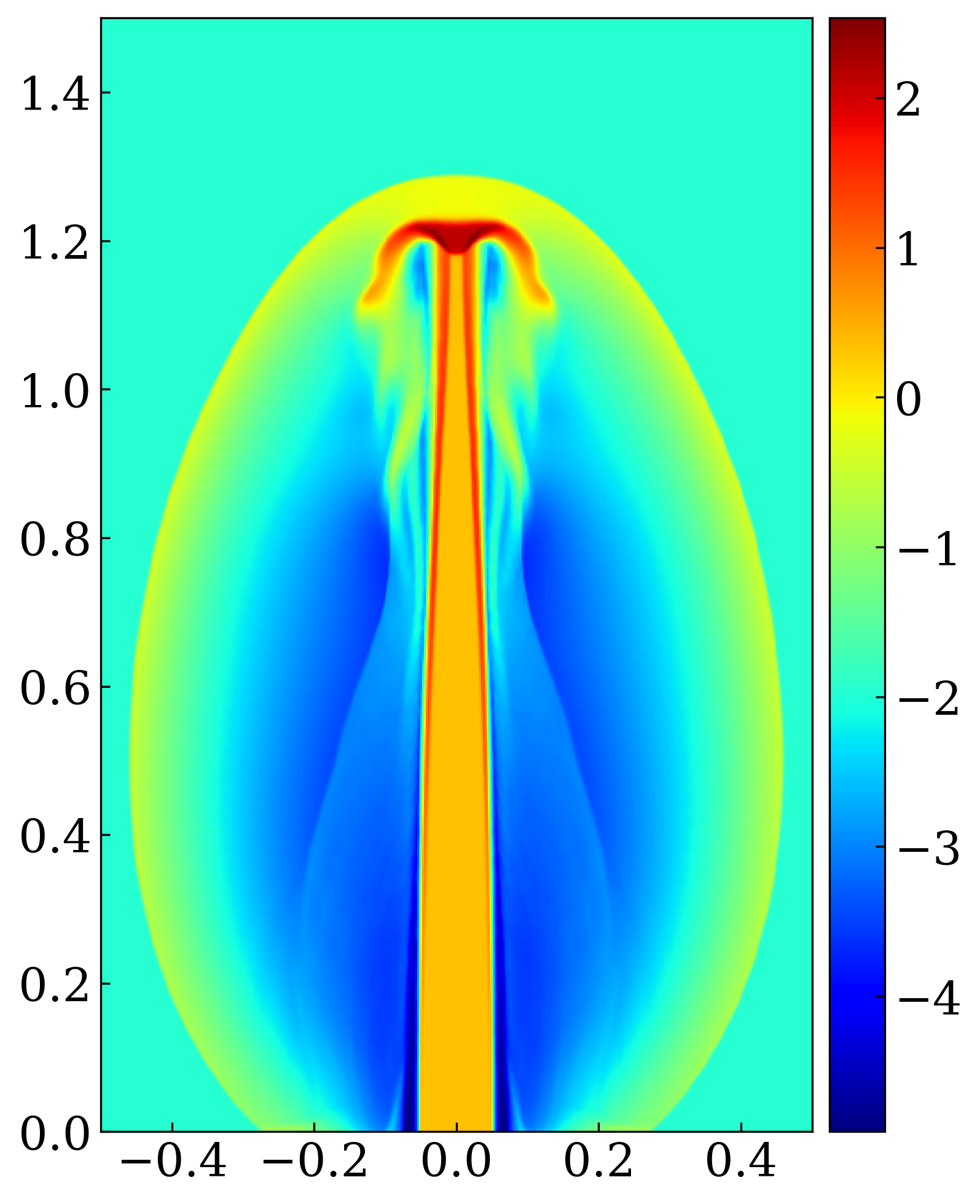}\hspace{0.08\linewidth}
		\includegraphics[width=0.32\linewidth]{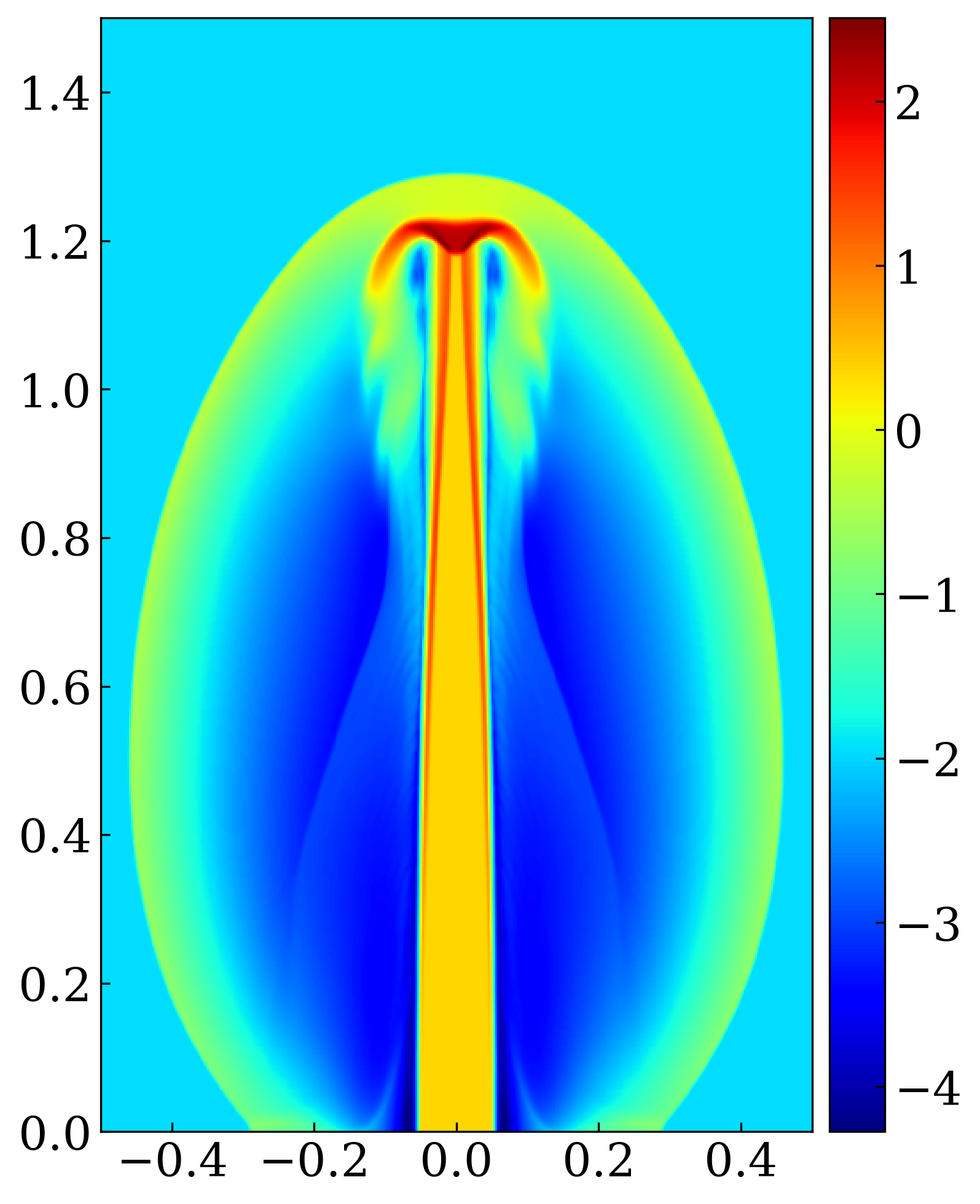}
		
		\medskip
		\includegraphics[width=0.43\linewidth]{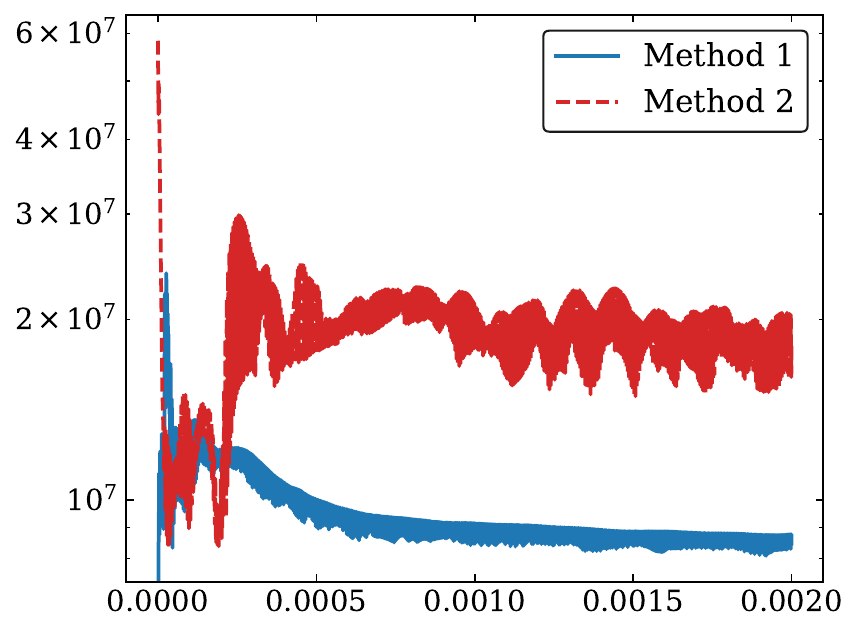}\hfill
		\includegraphics[width=0.5\linewidth]{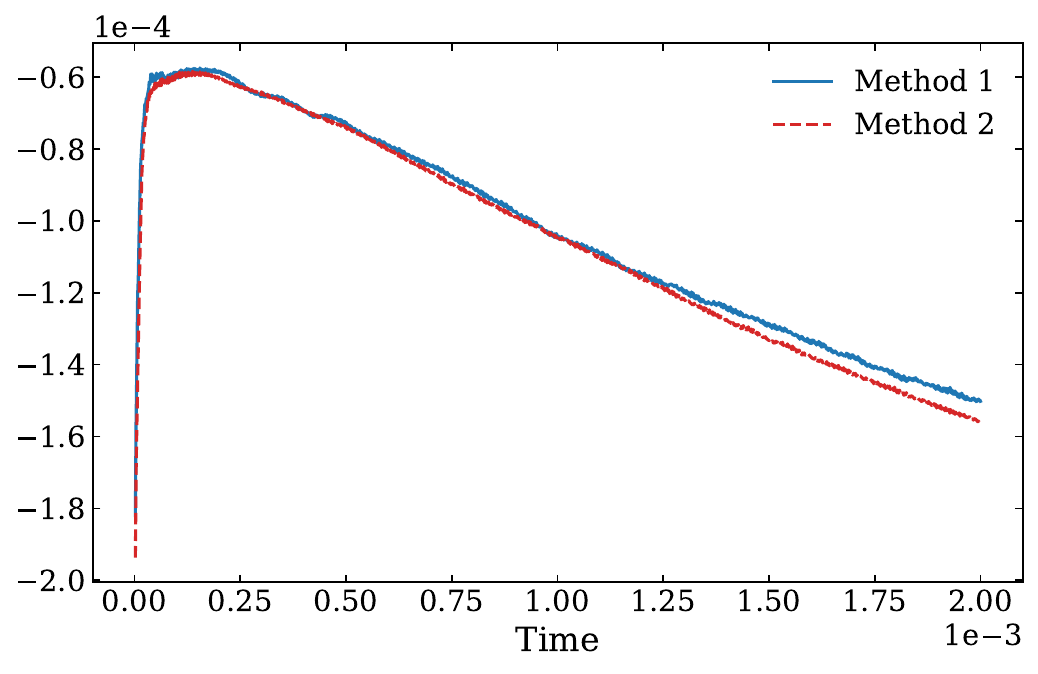}
		\caption{Example~\ref{ex:jet} with $\beta=10^{-2}$: $\log\rho$ at $T=0.002$ computed by Method~1 (top left) and Method~2 (top right); maximum HLL jump dissipation \eqref{eq:num-hll-jump} (bottom left) and fully discrete global entropy residuals \eqref{eq:num-entropy-residual} with $\xi=0$ (bottom right), for Methods~1 and~2.}
		\label{fig:num-jet1}
	\end{figure}
	
	\begin{figure}[!htbp]
		\centering
		\includegraphics[width=0.25\linewidth]{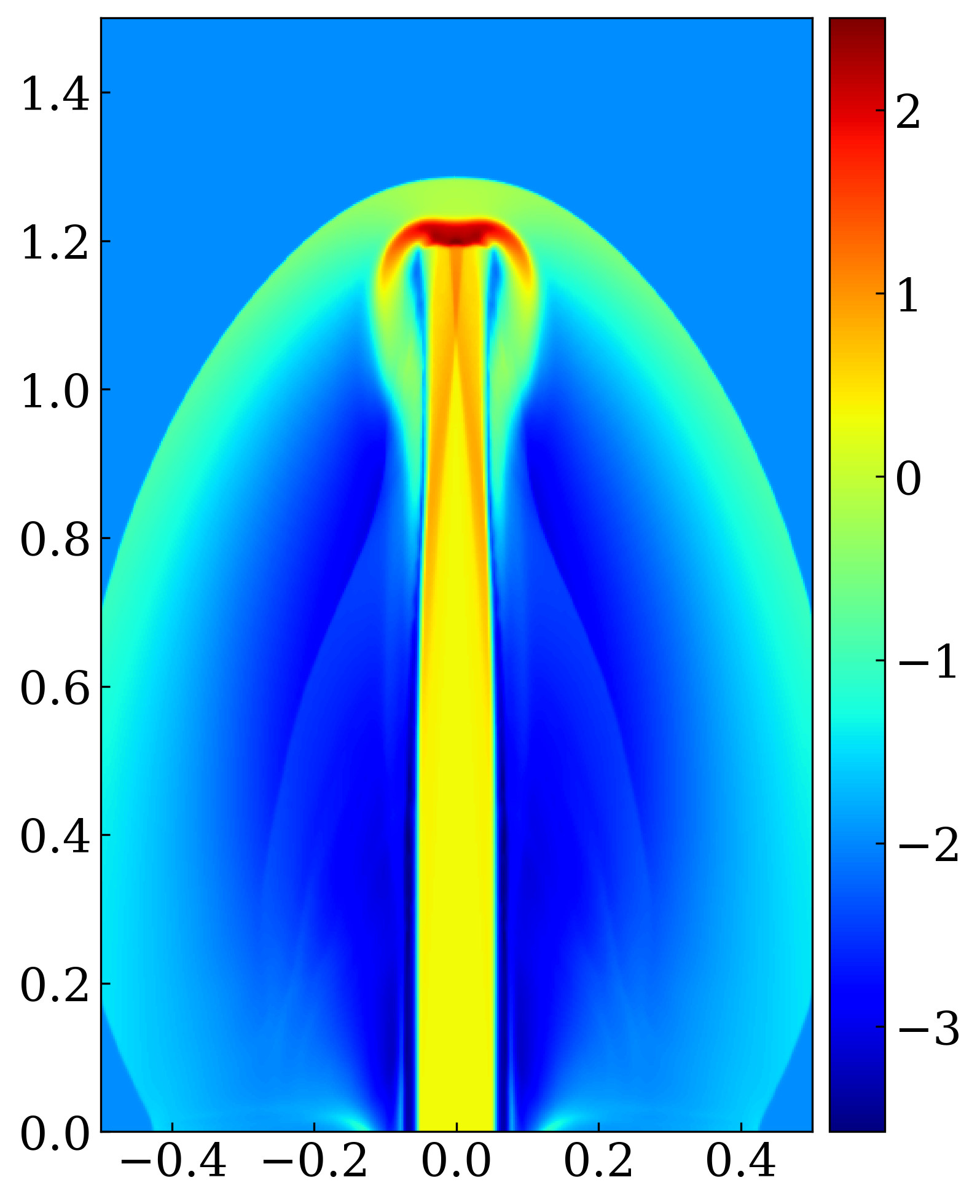}
		\includegraphics[width=0.25\linewidth]{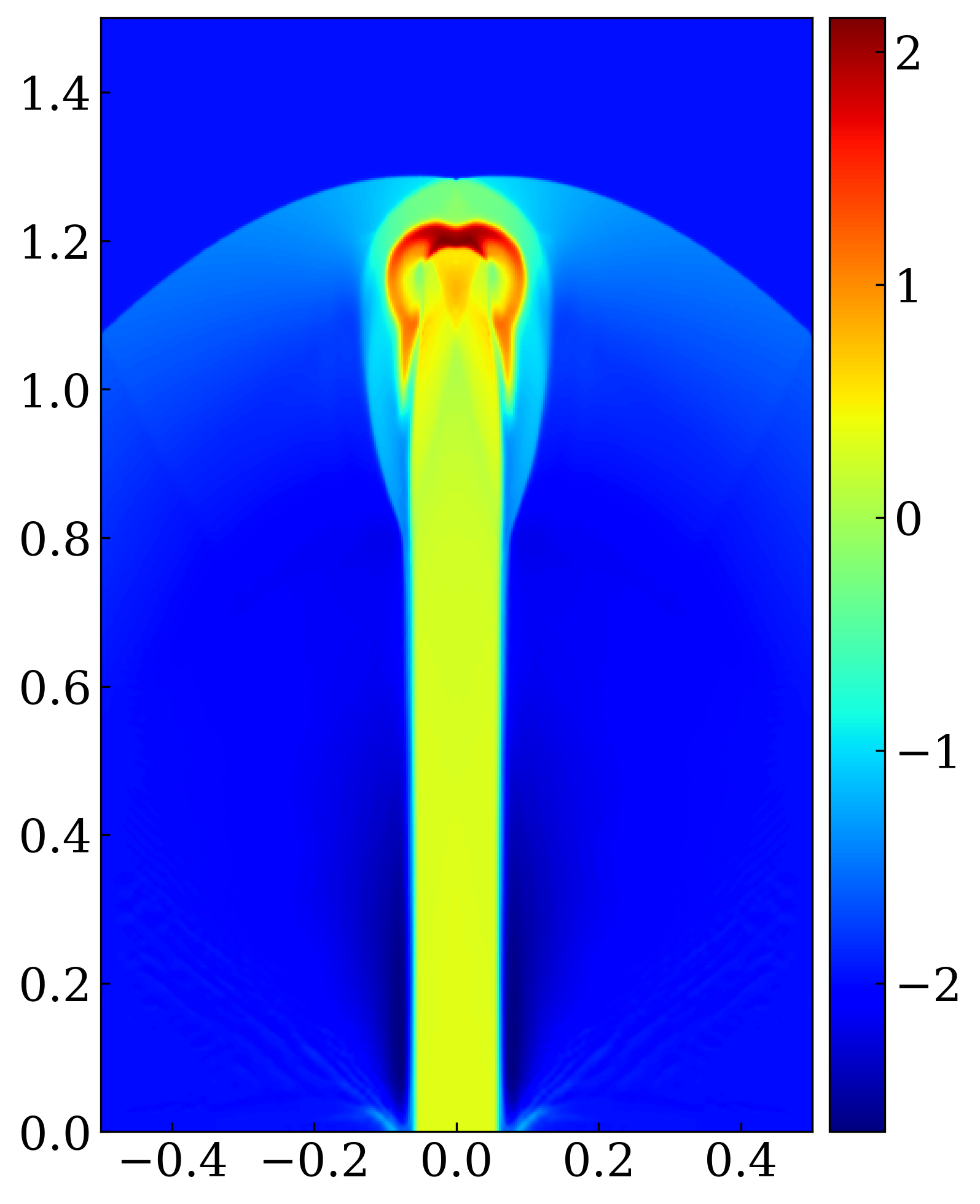}
		\includegraphics[width=0.4\linewidth]{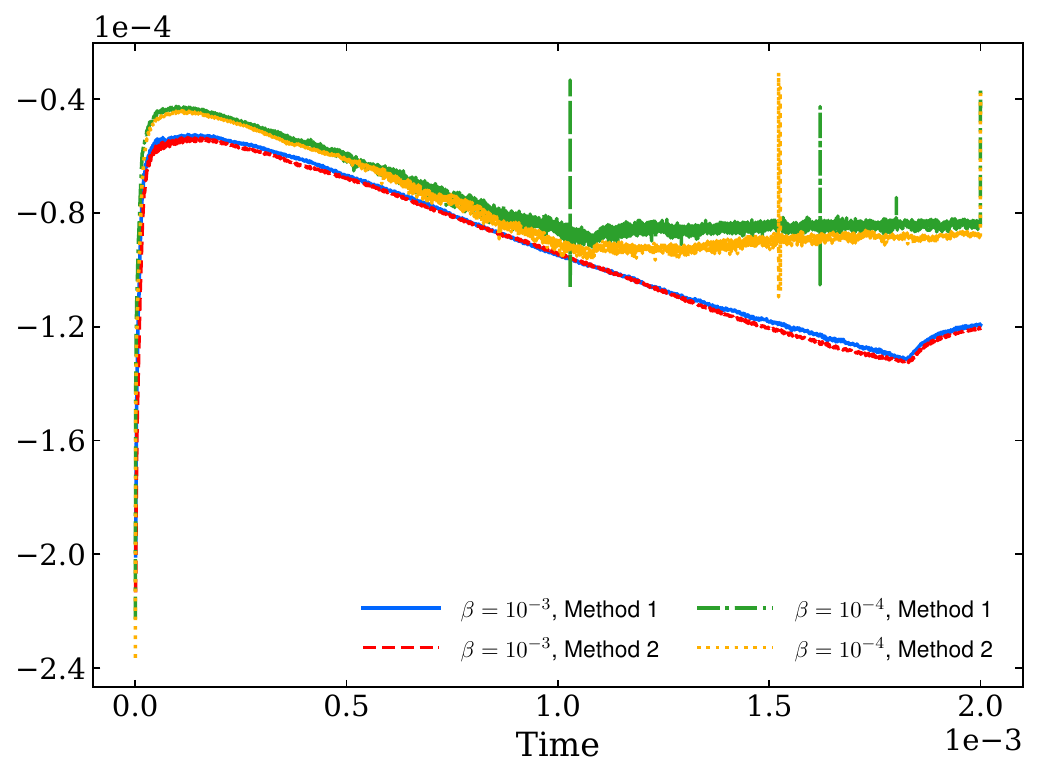}
		\caption{Example~\ref{ex:jet}: $\log\rho$ at $T=0.002$ computed by Method~1 for $\beta=10^{-3}$ (left) and $\beta=10^{-4}$ (center); fully discrete global entropy residuals \eqref{eq:num-entropy-residual} for Methods~1 and~2 in both cases, with $\xi=0$ (right).}
		\label{fig:num-jet23}
	\end{figure}
\end{numexample}

\section{Conclusions}
\label{sec:con}

In this work, we have established a fully discrete weak-to-strong (W2S) multi-entropy stability theory for arbitrarily high-order finite-volume and discontinuous Galerkin schemes approximating ideal compressible MHD on general polytopal meshes. Because thermodynamic entropies and their dual entropy variables are defined only on the admissible state space $\G$, physical admissibility ($\rho>0$, $p>0$) constitutes an indispensable prerequisite for discrete entropy stability. The framework developed here resolves the long-standing difficulty of reconciling discrete entropy dissipation with physical admissibility, the magnetic solenoidal constraint, and the nonconservative Godunov--Powell coupling, ensuring that a single numerical update simultaneously satisfies discrete entropy inequalities for any prescribed finite family of convex Harten entropy pairs.

Our cell-average weak stability analysis establishes two complementary criteria based on convex decompositions and Bregman relative entropies. Crucially, this analysis reveals the exact mechanism by which magnetic divergence affects fully discrete entropy stability: normal magnetic jumps across cell interfaces introduce a leading-order linear defect into the classical relative entropy flux, which the Godunov--Powell correction eliminates, allowing the corrected flux to be bounded by the relative entropy, while locally divergence-free (LDF) approximations cooperatively cancel the boundary corrections and remove an otherwise sign-indefinite entropy residual. Building upon this foundation, the W2S lifting operator unifies physical admissibility and simultaneous multi-entropy bounds into a single cellwise scaling limiter that preserves updated cell averages and the divergence-free magnetic involution, with entropy factors determined by scalar root solves when limiting is active.

High-order temporal accuracy is achieved via variable-step SSP multistep methods, whose convex decomposition into forward Euler steps allows the multi-entropy bounds to propagate under the SSP coefficient and step-size conditions of Theorem~\ref{thm:time-ssp-multistep}. The W2S lifting limiter is executed only once per complete time step, which avoids the overhead of intermediate substage limiting and circumventing temporal order reduction. By establishing the first fully discrete unification of physical admissibility, magnetic divergence control, and multi-entropy stability at arbitrary order, this work provides a rigorous foundation for simulating extreme plasma flows in near-vacuum, hypersonic, and strongly magnetized regimes, while offering a potential framework for broader constrained nonconservative hyperbolic systems.

\appendix
\section{Convexity and compatibility of the Harten entropies}
\label{app:harten}

Throughout this appendix, all derivatives are evaluated with respect to the conservative variable vector $\bm U\in\G$. For any scalar function $g\in C^2(\G)$ and direction $\bm Z\in\R^8$, we denote the first and second Fr\'echet derivatives by
\begin{equation*}
	Dg(\bm U)[\bm Z]=\nabla g(\bm U)^\top\bm Z,
	\qquad
	D^2g(\bm U)[\bm Z,\bm Z]
	=\bm Z^\top\nabla^2g(\bm U)\bm Z,
\end{equation*}
omitting the base state $\bm U$ when clear from context. For a vector-valued mapping $\bm G$, $D\bm G(\bm U)$ denotes its Jacobian matrix acting on $\bm Z$ as $D\bm G(\bm U)[\bm Z]$.

\begin{lemma}[Positive definiteness of the entropy Hessian]
	\label{lem:harten-convexity}
	Let $f\in C^2(\R)$ satisfy Harten's conditions \eqref{eq:intro-harten-conditions}, namely $f'>0$ and $f''<f'/\gamma$. Then the generalized entropy function $\eta_f(\bm U)=-\rho f(s)$ has a positive definite Hessian $\nabla^2\eta_f(\bm U)>0$ on the physical domain $\G$. Consequently, for any compact convex subset $\Kc\subset\G$, there exist positive constants $0<c_{\Kc,f}\le C_{\Kc,f}<\infty$ such that the relative entropy satisfies the two-sided quadratic bounds
	\begin{equation*}
		c_{\Kc,f}\norm{\bm U-\bm W}^2
		\le \eta_f(\bm U)-\eta_f(\bm W)
		-\nabla\eta_f(\bm W)^\top(\bm U-\bm W)
		\le C_{\Kc,f}\norm{\bm U-\bm W}^2,
	\end{equation*}
	for all $\bm U,\bm W\in\Kc$.
\end{lemma}

\begin{proof}
	Fix $\bm U\in\G$ and let $\varepsilon:=\rho e=p/(\gamma-1)>0$ denote the internal energy density. For an arbitrary perturbation direction $\bm Z=(z_\rho,\bm z_m^\top,\bm z_B^\top,z_E)^\top\in\R^8$, we introduce the dimensionless differentials
	\begin{equation*}
		a:=\frac{z_\rho}{\rho},\qquad
		b:=\frac{D\varepsilon[\bm Z]}{\varepsilon}
		=\frac{z_E-\bm v\cdot\bm z_m+
			\tfrac12|\bm v|^2z_\rho-\bm B\cdot\bm z_B}
		{\varepsilon}.
	\end{equation*}
	From the thermodynamic relation $s=\log((\gamma-1)\varepsilon)-\gamma\log\rho$, we have $Ds[\bm Z]=b-\gamma a$. Direct differentiation of the internal energy density $\varepsilon = E - \frac{|\bm m|^2}{2\rho} - \frac{|\bm B|^2}{2}$ gives
	\begin{equation*}
		D^2\varepsilon[\bm Z,\bm Z]
		=-\frac{|\bm z_m-\bm v z_\rho|^2}{\rho}-|\bm z_B|^2,
		\qquad
		D^2s[\bm Z,\bm Z]
		=\frac{D^2\varepsilon[\bm Z,\bm Z]}{\varepsilon}
		-b^2+\gamma a^2.
	\end{equation*}
	Differentiating $\eta_f(\bm U)=-\rho f(s)$ twice, substituting $D^2s[\bm Z,\bm Z]$, and completing the square in $a$ and $b$, we obtain the quadratic form decomposition
	\begin{equation}\label{eq:app-harten-hessian}
		\begin{aligned}
			D^2\eta_f[\bm Z,\bm Z]
			={}&\frac{\rho}{\gamma}
			\bigl[(f'-\gamma f'')(b-\gamma a)^2
			+(\gamma-1)f'b^2\bigr]\\
			&+\frac{f'}{\varepsilon}|\bm z_m-\bm v z_\rho|^2
			+\frac{\rho f'}{\varepsilon}|\bm z_B|^2,
		\end{aligned}
	\end{equation}
	where $f'$ and $f''$ are evaluated at $s$. Under Harten's conditions \eqref{eq:intro-harten-conditions} ($f'>0$ and $f'-\gamma f''>0$), together with $\rho>0$, $\varepsilon>0$, and $\gamma>1$, every coefficient in \eqref{eq:app-harten-hessian} is strictly positive. If $D^2\eta_f[\bm Z,\bm Z]=0$, the non-negativity of each squared term forces
	\begin{equation*}
		\bm z_B=\bm 0, \qquad \bm z_m-\bm v z_\rho=\bm 0, \qquad b=0, \qquad b-\gamma a=0.
	\end{equation*}
	Because $\gamma>0$, the relations $b=0$ and $b-\gamma a=0$ imply $a=z_\rho/\rho=0$, whence $z_\rho=0$. This in turn yields $\bm z_m=\bm v z_\rho=\bm 0$. Finally, the condition $b\varepsilon = z_E - \bm v\cdot\bm z_m + \frac{1}{2}|\bm v|^2z_\rho - \bm B\cdot\bm z_B = 0$ collapses to $z_E=0$. Thus $\bm Z=\bm 0$, proving that $\nabla^2\eta_f(\bm U)$ is strictly positive definite on $\G$.
	
	On any compact convex subset $\Kc\subset\G$, the extreme eigenvalues of $\nabla^2\eta_f(\bm U)$ are continuous in $\bm U$ and strictly positive, hence bounded from above and below by constants $0<c_{\Kc,f}\le C_{\Kc,f}<\infty$. The quadratic relative-entropy bounds then follow by Taylor's theorem in integral form,
	\begin{equation*}
		\eta_f(\bm U)-\eta_f(\bm W)-\nabla\eta_f(\bm W)^\top(\bm U-\bm W)
		=\int_0^1 (1-t) D^2\eta_f\bigl(\bm W+t(\bm U-\bm W)\bigr)[\bm U-\bm W,\bm U-\bm W]\,\mathrm{d}t,
	\end{equation*}
	which completes the proof.
\end{proof}

Direct differentiation of $\eta_f(\bm U)=-\rho f(s)$ with respect to the conservative variables $\bm U=(\rho,\bm m^\top,\bm B^\top,E)^\top$ yields the entropy variables $\bm V_f:=\nabla\eta_f(\bm U)$ and their contraction with the Godunov--Powell source vector $\bm S(\bm U)$:
\begin{equation}\label{eq:app-entropy-variables}
	\begin{aligned}
		\kappa_f&:=\frac{(\gamma-1)\rho f'(s)}p,\\
		\bm V_f
		&=\left(\gamma f'-f-\tfrac12\kappa_f|\bm v|^2,\,
		\kappa_f\bm v^\top,\,
		\kappa_f\bm B^\top,\,-\kappa_f\right)^\top,\\
		\bm V_f^\top\bm S&=\kappa_f\,\bm v\cdot\bm B.
	\end{aligned}
\end{equation}
For any spatial unit normal vector $\bm n$, adopting the directional embedding convention of Section~\ref{sec:md}, differentiating the directional entropy flux $q_{f,\bm n}=\eta_f\bm v\cdot\bm n$ and the directional physical flux $\bm F_{\bm n}$ along the increment $\bm Z$ yields the identity
\begin{equation*}
	Dq_{f,\bm n}(\bm U)[\bm Z]
	=\bm V_f(\bm U)^\top D\bm F_{\bm n}(\bm U)[\bm Z]
	+\kappa_f(\bm v\cdot\bm B)(\bm z_B\cdot\bm n),
	\qquad \forall\,\bm Z\in\R^8.
\end{equation*}
Since $DB_{\bm n}(\bm U)[\bm Z]=\bm z_B\cdot\bm n$ and $\phi_f(\bm U) = \bm V_f^\top\bm S = \kappa_f(\bm v\cdot\bm B)$, this relation establishes the multidimensional Godunov--Powell compatibility identity \eqref{eq:md-powell-compatibility}. Restricting this identity to the coordinate direction $\bm n=\bm e_1$ and increments with $z_{B_1}=0$ immediately recovers the classical one-dimensional compatibility relation \eqref{eq:1d-compatibility}. Combined with Taylor's theorem on compact convex subsets of $\G$, these compatibility relations furnish the uniform quadratic relative-flux bounds used throughout Sections~\ref{sec:1d} and~\ref{sec:md}.

\end{document}